\documentclass[11pt, twoside]{article}
\usepackage{amsfonts}

\usepackage{amsmath,amssymb}
\usepackage{multicol}
\usepackage{color}

\def\bar{\overline}

\numberwithin{equation}{section}
\newtheorem{theorem}{Theorem}[section]
\newtheorem{lemma}[theorem]{Lemma}

\newtheorem{definition}[theorem]{Definition}
\newtheorem{remark}[theorem]{Remark}

\newenvironment{proof}[1][Proof]{\noindent\textbf{#1.} }{\hfill $\Box$}
\allowdisplaybreaks
\allowdisplaybreaks \numberwithin{equation}{section}
 \makeatletter
\begin{document}

\author{ Wei-Jie Sheng\thanks{Corresponding author (E-mail address: 
shengwj09@hit.edu.cn).} \   and  Xin-Tian Zhang \\~\\
\footnotesize{School of Mathematics, Harbin Institute of Technology}, \\
\footnotesize{Harbin, Heilongjiang, 150001, People's Republic of China}
}

\title{\textbf{Existence and stability of curved fronts for spatially periodic combustion 
reaction-diffusion equations in $\mathbb{R}^N$ } } 

\date{}
\maketitle

\textbf{Abstract}: This paper is concerned with curved fronts of combustion reaction-diffusion equations in spatially periodic media in $\mathbb{R}^N$ $(N\geq2)$. Under the assumption that 
there are moving pulsating fronts for any given propagation direction 
$e \in \mathbb{S}^{N-1}$, and by constructing suitable super- and sub-solutions, 
we prove the existence of a  curved front with polytope-like shape in 
$\mathbb{R}^N$. Then we show that the curved front is unique and asymptotically stable.

\textbf{Keywords}: Reaction-diffusion equations; Curved fronts; Combustion; 
Spatial periodicity; Pulsating fronts.

\textbf{AMS Subject Classification (2020)}: 35B10, 35K10, 35K15, 35K57

\section{Introduction}
\noindent
In this paper, we study the spatially periodic reaction-diffusion equation
\begin{equation}\label{1.1}
u_t-\Delta_z u=f(z, u) \text { in }(t, z) \in \mathbb{R} \times \mathbb{R}^N,
\end{equation}
where $N \geq 2$, $u=u(t, z)$, $u_t=\frac{\partial u}{\partial t}$, $\Delta_{z}:=\partial^2 / \partial 
z_1^2+\partial^2 / \partial z_2^2+\cdots+\partial^2 / \partial z_N^2$ denotes the Laplace operator, 
and the reaction term $f(z, u)$ satisfies the following assumptions:
\begin{description}
\item [(F1)] $f: \mathbb{R}^N \times \mathbb{R} \rightarrow \mathbb{R}$ is of class 
$C^{\infty}\left(\mathbb{R}^{N+1}\right)$ and satisfies
\begin{equation}\label{1.2}
\|f\|_{C^k\left(\mathbb{R}^{N+1}\right)}=\sum_{i=0}^k\left\|D^i f\right\|_{L^{\infty}\left(\mathbb{R}^{N+1}\right)}<+\infty \text { for all } k \in \mathbb{N}.
\end{equation}
\item [(F2)] The function $f(\cdot, u): 
\mathbb{R}^N \rightarrow \mathbb{R}$ is $L$-periodic for any $u\in\mathbb{R}$, that is, 
$f\left(z_1, \ldots, z_N, u\right)=f\left(z_1+k_1 L_1, \ldots, z_N+k_N L_N, u\right)$ 
for all $\left(z_1, \ldots, z_N, u\right) \in \mathbb{R}^{N}\times \mathbb{R}$ and all $(k_1, k_2, \ldots, k_N)\in\mathbb{Z}^{N}$, 
where $L_1, \ldots, L_N$ are all positive constants. We call $\mathbb{L}^N:=\left(0, L_1\right) 
\times \cdots \times\left(0, L_N\right)$ the cell of periodicity.
\item [(F3)] There exists $\theta \in(0,1)$ such that
\begin{equation}\label{1.3}
\left\{\begin{array}{l}
\forall(z, u) \in \mathbb{R}^N \times[0, \theta] \cup\{1\}, f(z, u)=0 ;\\
\forall(z, u) \in \mathbb{R}^N \times(\theta, 1), f(z, u) \geq 0; \\
\forall u \in(\theta, 1), \exists z \in \mathbb{R}^N, \text { s.t. } f(z, u)>0.
\end{array}\right.
\end{equation}
\item[(F4)]There holds $\sup _{z \in \mathbb{R}^N} f_u(z, 1)<0$.
\end{description}

In fact, (F1) can be relaxed to $f \in C^m\left(\mathbb{R}^{N+1}\right)$ for $m \in \mathbb{N}$ large enough. However, for convenience, we assume $f \in C^{\infty}\left(\mathbb{R}^{N+1}\right)$. The reaction term $f(z,u)$ satisfying (F3) and (F4) is called the combustion nonlinearity. Let
\begin{equation}\label{1.4}
-K_1:=\inf _{z \in \mathbb{R}^N} f_u(z, 1) \text { and }-\kappa_1:=\sup _{z \in \mathbb{R}^N} f_u(z, 1).
\end{equation}
By (F1) and (F4), we get $0<\kappa_1 \leq K_1<+\infty$. From \eqref{1.2} and \eqref{1.4}, one has that there exists a positive constant $0<\gamma_{\star}\leq \min \{\theta / 2,1-\theta\}$ such that
\begin{equation}\label{1.5}
f_u(z, u) \leq-\frac{\kappa_1}{2}, \quad \forall(z, u) \in \mathbb{R}^N \times\left[1-\gamma_{\star}, 1+\gamma_{\star}\right].
\end{equation}
For mathematical convenience, we suppose that $f(z, u)=0$ for $(z, u) \in \mathbb{R}^N \times$ $(-\infty, 0)$ and $f(z, u)=f_u(z, 1)(u-1)$ for $(z, u) \in \mathbb{R}^N \times[1,+\infty)$. Futhermore, we can obtain that $f(z, u)$, $f_{u}(z, u)$ and $f_{uu}(z, u)$ are globally Lipschitz-continuous in $u$ uniformly for $z \in \mathbb{R}^N$, and
\begin{equation}\label{1.6}
f_u(z, u) \leq-\frac{\kappa_1}{2}, \quad \forall(z, u) \in \mathbb{R}^N \times[1-\gamma_{\star},+\infty).
\end{equation}

\subsection{Planar traveling fronts and curved fronts}
In this subsection, we present some results regarding traveling fronts and curved fronts in reaction-diffusion equations
\begin{equation}\label{ad.1}
\partial_t u(t, z)=\Delta u(t, z)+f(u(t, z)), \quad \forall z \in \mathbb{R}^N, t>0.
\end{equation}
The planar traveling front is a class of solutions of \eqref{ad.1} with the form $u(t, z)=\phi(z \cdot 
e-c t)$ satisfying
\begin{equation*}
\left\{\begin{array}{l}
	\phi^{\prime \prime}+c \phi^{\prime}+f(\phi)=0 \text { in } \mathbb{R}, \\
	\phi(-\infty)=1, \quad \phi(+\infty)=0,
\end{array}\right.
\end{equation*}
where $c > 0$ is the propagation speed and $e$ is a unit vector of $\mathbb{R}^N$. The level sets of such traveling fronts are parallel hyperplanes orthogonal to the 
propagation direction $e$. The existence, stability and other qualitative  properties of 
planar traveling waves  can be referred to 
\cite{BerestyckiH2002,Chen1997,Fife1977,Sheng2017.2,Sheng2017} and  references therein.

In high dimensional space, the propagation phenomena is very complex. Since besides planar 
traveling fronts, there are also some other traveling fronts with non-planar level sets, and we call 
such fronts non-planar traveling fronts or curved fronts in the sequel. Level sets of traveling fronts 
are quite diverse and have various possible shapes such as $V$-shape, pyramidal shape, conical 
shape, etc. For bistable case, Ninomiya and Taniguchi \cite{H. Ninomiya2005} established the 
existence and stability of $V$-shaped traveling fronts in $\mathbb{R}^2$ by the super-sub solution 
method and the comparison principle. Subsequently, Ninomiya and Taniguchi \cite{H. Ninomiya2006} further obtained the global stability of $V$-shaped traveling fronts obtained in \cite{H. Ninomiya2005} in $\mathbb{R}^2$. Taniguchi \cite{M. Taniguchi2007, M. Taniguchi2009} proved the existence and stability of the pyramidal traveling fronts in 
$\mathbb{R}^3$. For Fisher-KPP case, Hamel and Nadirashvili \cite{F. Hamel2001} proved the 
existence of non-planar traveling fronts of \eqref{ad.1} in $\mathbb{R}^N$. Then, Huang \cite{R. Huang2008} showed the stability of the non-planar traveling fronts of \eqref{ad.1} in 
$\mathbb{R}^N$. For combustion case, Bonnet and Hamel \cite{A. Bonnet1999} analyzed the conical premixed Bunsen flames and studied the existence of $V$-shaped traveling 
fronts in $\mathbb{R}^2$. Wang and Bu \cite{Wang2016} established the existence of pyramidal 
traveling fronts in $\mathbb{R}^3$ and $V$-shaped traveling fronts in $\mathbb{R}^2$, by 
constructing appropriate super- and subsolutions. Bu and Wang \cite{Bu2017,Bu2018} proved the 
stability of pyramidal traveling fronts and the global stability of $V$-shaped traveling fronts, 
respectively. More related results on non-planar traveling fronts, we refer to 
\cite{Brazhnik2000,Bu2016,El Smaily 
M2011,Haragus2006,Haragus2007,Sheng2025,ShengZhang2025-1,ShengZhang2025-2} 
and references therein.

\subsection{Pulsating fronts and transition fronts}

In recent years, an increasing attention has  been paid to  the propagation dynamics of 
reaction-diffusion equations in spatially periodic media. An natural extension of planar traveling 
fronts is the pulsating front. Let us first give the definition of pulsating fronts. 
\begin{definition}(\cite{BerestyckiH2002})\label{De 1.2}
A pair $(U_e, c_e)$ with $U_e: \mathbb{R} \times \mathbb{R}^N \rightarrow \mathbb{R}$ 
and $c_e \in \mathbb{R}$ is said to be a pulsating front of \eqref{1.1} with effective speed 
$c_e$ in the direction $e \in \mathbb{S}^{N-1}$ connecting two equilibria 0 and 1, if 
the following three properties are satisfied:
\begin{description}
\item[(i)] 	 the function $u(t, z):=U_e\left(z \cdot e-c_e t, z\right)$ is an entire (classical) solution of 
Eq.\eqref{1.1}.
\item [(ii)] the profile $U_e$ satisfies $U_e(s, z)=U_e(s, z+y)$ for all $(s, z) \in \mathbb{R} \times 
\mathbb{R}^N$ and $y \in \prod_{i=1}^N L_i \mathbb{Z}$.
\item [(iii)]the profile $U_e$ satisfies
	\begin{equation*}
		\lim _{s \rightarrow+\infty} U_e(s, z)=0 \text { and } \lim _{s \rightarrow-\infty} U_e(s, z)=1 
		\; \text { uniformly for } z \in \mathbb{R}^N.		
	\end{equation*}
\end{description}
\end{definition}
From the Definition \ref{De 1.2},   the pulsating front $(U_e\left(z \cdot e-c_e t, z\right),c_e)$ 
of Eq.\eqref{1.1} should satisfy
\begin{equation}\label{2.1}
c_e \partial_s U_e+\partial_{s s} U_e+2 \nabla_z \partial_s U_e \cdot e+\Delta_z U_e
+f\left(z, U_e\right)=0 \text { in } \mathbb{R} \times \mathbb{R}^N.
\end{equation}
Under assumptions (F1)-(F4), Berestycki and Hamel \cite{BerestyckiH2002} proved that
there exists a unique pulsating front $U_e\left(z \cdot e-c_e t, z\right)$ of  Eq.\eqref{1.1} 
with combustion nonlinearity. They also showed that the profile $U_e(s, z): \mathbb{R} 
\times \mathbb{R}^N \rightarrow \mathbb{R}$  is strictly decreasing in $s$ 
and the speed $c_e$ is unique. 
Ding et al.\cite{Ding2017} showed the existence and some qualitative properties of pulsating fronts 
for spatially periodic bistable reaction-diffusion equations. Guo \cite{Guo2018} studied the 
propagating speeds of transition fronts for spatially periodic bistable reaction-diffusion equations in 
$\mathbb{R}^N$. Xin \cite{Xin1993} discussed the existence and non-existence of traveling wave 
solutions for reaction-diffusion equations with combustion nonlinearity and periodic diffusion 
coefficients. For more researches of pulsating fronts, one can refer to 
\cite{Ducrot2014,Giletti2020,Liang2010,Xin X1991} and  references therein.

In order to handle the propagation phenomena in general unbounded domains,
Berestycki and Hamel \cite{H. Berestycki2012} introduced the notion of transition fronts
and their global mean speeds, which provide a unified framework for dealing 
with different traveling fronts. A notable feature of transition fronts is that they 
involve many classic traveling fronts. Let us briefly introduce what
transitions fronts are. For any two subsets $A$ and $B$ of $\mathbb{R}^N$, define
\begin{equation*}
	d(A, B)=\inf \{|z-y|,(z, y) \in A \times B\}
\end{equation*}
and $d(z, A)=d(\{z\}, A)$, where $z \in \mathbb{R}^N$ and $|\cdot|$ is the Euclidean norm in $\mathbb{R}^N$. Let $\left(\Omega_t^{-}\right)_{t \in \mathbb{R}}$ and $\left(\Omega_t^{+}\right)_{t \in \mathbb{R}}$ be two families of open nonempty subsets of $\mathbb{R}^N$ satisfying 
\begin{equation}\label{1.7}
\forall t \in \mathbb{R}, \quad\left\{\begin{array}{l}
	\Omega_t^{+} \cap \Omega_t^{-}=\emptyset, \\
	\partial \Omega_t^{+}=\partial \Omega_t^{-}=: \Gamma_t, \\
	\Omega_t^{+} \cup \Omega_t^{-} \cup \Gamma_t=\mathbb{R}^N, \\
	\sup \left\{d\left(z, \Gamma_t\right) ; z \in \Omega_t^{+}\right\}=\sup \left\{d\left(z, \Gamma_t\right) ; z \in \Omega_t^{-}\right\}=+\infty
\end{array}\right.
\end{equation}
and
\begin{equation}\label{1.8}
\left\{\begin{array}{l}
	\inf \left\{\sup \left\{d\left(y, \Gamma_t\right) ; y \in \Omega_t^{+},|y-z| \leq r\right\} ; t \in \mathbb{R}, z \in \Gamma_t\right\} \rightarrow+\infty \\
	\inf \left\{\sup \left\{d\left(y, \Gamma_t\right) ; y \in \Omega_t^{-},|y-z| \leq r\right\} ; t \in \mathbb{R}, z \in \Gamma_t\right\} \rightarrow+\infty
\end{array} \text { as } r \rightarrow +\infty .\right.
\end{equation}
Formulas \eqref{1.7} and \eqref{1.8} indicate that $\Gamma_t$ divides $\mathbb{R}^N$ into 
two unbounded parts $\Omega_t^{-}$ and $\Omega_t^{+}$, and sets $\Omega_t^{ \pm}$ 
contains points that are infinitely far from $\Gamma_t$ for each $t \in \mathbb{R}$. 
By \eqref{1.8}, one has that for any positive constant $M>0$, there is a positive constant $r_M$ such that for all $t \in \mathbb{R}$ and $z \in \Gamma_t$, there exist $y^{ \pm}=y_{t, z}^{ \pm} \in \mathbb{R}^N$ such that
\begin{equation*}
y^{ \pm} \in \Omega_t^{ \pm},\left|z-y^{ \pm}\right| \leq r_M \text { and } d\left(y^{ \pm}, \Gamma_t\right) \geq M .
\end{equation*}
Similar to \cite{H. Berestycki2012}, we suppose that the sets $\Gamma_t$ are composed of a finite number of graphs, that is, there exists an integer $n \geq 1$ such that, for any $t \in \mathbb{R}$, there are $n$ open subsets $\omega_{i, t} \subset \mathbb{R}^{N-1}(1 \leq i \leq n)$, $n$ continuous maps $\psi_{i, t}: \omega_{i, t} \rightarrow \mathbb{R}$ $(1 \leq i \leq n)$ and $n$ rotations $R_{i, t}$ of $\mathbb{R}^N$ $(1 \leq i \leq n)$, such that
\begin{equation}\label{1.9}
\Gamma_t \subset \bigcup_{1 \leq i \leq n} R_{i, t}\left(\left\{z \in \mathbb{R}^N ; \left(z_{1}, z_{2}, \ldots, z_{N-1}\right) \in \omega_{i, t}, z_N=\psi_{i, t}\left(z_{1}, z_{2}, \ldots, z_{N-1}\right)\right\}\right).
\end{equation}
\begin{definition}(\cite{H. Berestycki2012})\label{De 1.1}
	 For problem \eqref{1.1}, a transition front connecting 0 and 1 is a classical solution $u$ of \eqref{1.1} such that $u \not \equiv 0,1$, and there exist some sets $\left(\Omega_t^{ \pm}\right)_{t \in \mathbb{R}}$ and $\left(\Gamma_t\right)_{t \in \mathbb{R}}$ satisfying \eqref{1.7}, \eqref{1.8} and \eqref{1.9}, and for any $\varepsilon>0$, there exists $M_{\varepsilon}>0$ such that
\begin{equation}\label{1.10}
	\left\{\begin{array}{l}
		\forall t \in \mathbb{R}, \forall z \in \Omega_t^{+},\left(d \left(z, \Gamma_t\right) \geq M_{\varepsilon}\right) \Longrightarrow(u(t, z) \geq 1-\varepsilon), \\
		\forall t \in \mathbb{R}, \forall z \in \Omega_t^{-},\left(d \left(z, \Gamma_t\right) \geq M_{\varepsilon}\right) \Longrightarrow(u(t, z) \leq \varepsilon) .
	\end{array}\right.
\end{equation}
Moreover, $u$ is said to have a global mean speed $\gamma$ $(\geq 0)$ if
\begin{equation*}
	\frac{d\left(\Gamma_t, \Gamma_s\right)}{|t-s|} \rightarrow \gamma \text { as }|t-s| \rightarrow+\infty.
\end{equation*}
\end{definition}

For bistable transition fronts, Ding et al.\cite{Ding2015} studied the existence of transition fronts
for spatially periodic reaction-diffusion equations. Sheng and Wang
\cite{Sheng2018} obtained new entire solutions for monotone reaction-diffusion systems 
and showed transition fronts have a unique global mean speed. Guo and Wang 
\cite{H. Guo2024} constructed entire solutions (transition fronts) of reaction-diffusion 
equations by mixing finite planar fronts and showed the uniqueness and stability of the  
entire solutions. Hamel and Rossi \cite{Hamel2016} 
studied transition fronts for Fisher-KPP reaction-diffusion equations and obtained some 
qualitative properties of the transition fronts. Nolen and 
Ryzhik \cite{Nolen2009} proved the existence of  transition fronts in 
$\mathbb{R}$ for combustion reaction-diffusion equations. Bu, Guo and Wang \cite{Bu2019} 
established the existence and the uniqueness 
of the global mean speed of transition fronts for combustion reaction-diffusion equations 
in $\mathbb{R}^N (N\geq 1)$. For more results 
about transition fronts, one can refer to 
\cite{H. Guo2020,Hamel2016.2,Mellet2009,Sheng2018.2,Zlato2017,Suobing Zhang2025} 
and references therein.

The purpose of this paper is to obtain the existence, uniqueness and stability of curved fronts 
of spatially periodic combustion reaction-diffusion equations in $\mathbb{R}^N$ $(N\geq 2)$.
It should be mentioned that  the existence, uniqueness and stability of $V$-shaped traveling fronts 
of Eq.\eqref{1.1} in $\mathbb{R}^2$ is studied byLyu et al.\cite{LyuY2024}.  More recently, Guo 
and Wang \cite{Guo H2025} established the existence, uniqueness and asymptotic stability of 
curved fronts of spatially periodic bistable reaction-diffusion equations in $\mathbb{R}^N$ 
$(N\geq 2)$.

The rest of this paper is organized as follows. In Section $2$, we list some existing conclusions of curved fronts and pulsating fronts, and then present the main research results of this paper. In Section $3$, we construct suitable subsolutions and supersolutions to prove the existence and uniqueness of the curved front of combustion reaction-diffusion equations in spatially periodic media $(N \geq 2)$. In Section $4$, we prove the stability of the curved fronts obtained in Section $3$. Eventually, in the appendix, we give the lemma proofs that are not given in Section $2$.

\section{Preliminaries and Main Results}
\noindent
In this section, we first present some existing results of the pulsating fronts of \eqref{1.1}, 
then we give the main results of this paper.  In Subsection $2.2$, we introduce a
 hypersurface that asymptotically approaches a number of hyperplanes at infinity and 
then we  use the distance function of this hypersurface to construct supersolutions 
in Section $3$. In Subsection $2.3$, we show the main results obtained in this paper.
In the sequel, we always assume that  conditions (F1)-(F4)  hold.
\subsection{Preliminaries}

The following theorem comes from \cite[Theorems 2.4 and 2.5]{Alfaro2016} and \cite[Theorem 2.2]{Bu2022}, which states the continuity of the speed 
$c_e$ and the profile $U_e$ in the direction $e$, the bounds of $c_e$, and the exponential behavior of the profile $U_e$.  
\begin{theorem}\label{Theorem 2.3}  
	 Suppose that (F1)-(F4) hold and let $e \in \mathbb{S}^{N-1}$. Assume that $\left(U_e, 
	 c_e\right)$ is the unique pulsating front of \eqref{1.1}. Then
\begin{description}
\item[(i)]  the mapping $e \in \mathbb{S}^{N-1} \mapsto c_e$ is continuous.
\item[(ii)] there exist two constants $\kappa$ and $K$ such that
	\begin{equation*}
	0<\kappa:=\inf _{e \in \mathbb{S}^{N-1}} c_e \leq \sup _{e \in \mathbb{S}^{N-1}} c_e=: 
	K<+\infty.
	\end{equation*}
\item[(iii)]the mapping $e \in \mathbb{S}^{N-1} \mapsto U_e$ is continuous under the 
topology $\|\cdot\|_{L^{\infty}\left(\mathbb{R} \times \mathbb{R}^N\right)}$, by normalization 
as 
$\min_{z \in \mathbb{R}^N} U_e(0, z)=(1+\theta) / 2$, where $\theta$ is defined in \eqref{1.3}.
\item[(iv)] there exist two constants $C_1>0$ and $C_2<0$ such that
\begin{equation*}
	U_e(s, z) \sim C_1 e^{-c_e s} \text { and } \partial_s U_e(s, z) \sim C_2 e^{-c_e s} \text { as } s 
	\rightarrow+\infty
\end{equation*}
uniformly in $z \in \mathbb{R}^N$. Notice that $U_e(s, z) \sim C_1 e^{-c_e s}$ as $s 
\rightarrow+\infty$ uniformly in $z \in \mathbb{R}^N$ means
\begin{equation*}
	\liminf_{s \rightarrow+\infty} \min_{x \in \mathbb{R}^N} \frac{U_e(s, z)}{C_1 e^{-c_e 
	s}}=\limsup _{s \rightarrow+\infty} \max_{z \in \mathbb{R}^N} \frac{U_e(s, z)}{C_1 e^{-c_e 
	s}}=1.
\end{equation*}
\end{description}
\end{theorem}
The following theorem provides the exponentially asymptotic behaviors and uniform estimates of the pulsating front and its derivatives.
\begin{theorem}(\cite[Theorems 2.5 and 2.7]{LyuY2024})\label{Theorem 2.5}
	Suppose that (F1)-(F4) hold and let $e \in \mathbb{S}^{N-1}$. Assume that $\left(U_e, c_e\right)$ is the unique pulsating front of \eqref{1.1}. Then for any nonnegative integers $k$ and $l$, there exists a constant $C_{k l}$ dependent on $k$ and $l$, such that
	\begin{equation*}
		\lim _{s \rightarrow+\infty} \frac{D_z^k D_s^l U_e}{U_e}  =C_{k l},
		\lim _{s \rightarrow+\infty} \frac{\partial_s U_e}{U_e}  =-c_e , 
		\lim _{s \rightarrow+\infty} \frac{\partial_{s s} U_e}{U_e}  =c_e^2 ,
	\end{equation*}   
    \begin{equation*}
		\lim _{s \rightarrow+\infty} \frac{\left|\nabla_z U_e\right|,\left|\nabla_z \partial_s U_e\right|}{U_e}  =0\; \text{ and }
		\lim _{s \rightarrow+\infty} \frac{\Delta_z U_e}{U_e}  =0  \label{2.6}
	\end{equation*}
	uniformly in $z \in \mathbb{R}^N$, where $\nabla_z$ denotes the gradient operator with 
	respect to $z \in \mathbb{R}^N$. Furthermore, normalize $U_e$ as $\min _{z \in 
	\mathbb{R}^N} U_e(0, z)=(1+\theta) / 2$, where $\theta$ is defined in \eqref{1.3}. Then there 
	exist two constants $\bar{K}>0$ and $\kappa_2>0$, both independent of $e \in 
	\mathbb{S}^{N-1}$, such that
\begin{equation*}
\begin{aligned}
	\left|U_e(s, z)\right|,\left|D U_e(s, z)\right|,\left|D^2 U_e(s, z)\right|,\left|D^3 U_e(s, z)\right| & 
	\leq \bar{K} e^{-\frac{3 \kappa}{4} s} \text { in }[0,+\infty) \times \mathbb{R}^N, \\
	\left|1-U_e(s, z)\right|,\left|D U_e(s, z)\right|,\left|D^2 U_e(s, z)\right|,\left|D^3 U_e(s, z)\right| & 
	\leq \bar{K} e^{\kappa_2 s} \text { in }(-\infty, 0] \times \mathbb{R}^N,
\end{aligned}		
\end{equation*}
	where $\kappa$ is defined in Theorem \ref{Theorem 2.3}, $D$, $D^2$ and $D^3$ denote any 
	first-order, second-order and third-order derivative with respect to $(s, z) \in \mathbb{R} \times 
	\mathbb{R}^N$, respectively.
\end{theorem}
In the subsequent part of this paper, let the profile of pulsating fronts $U_e$ be always normalized as
\begin{equation}\label{2.7}
\int_{\mathbb{R}^{+} \times \mathbb{L}^N} U_e^2 \bar{\rho} \; d s d z =1 \text { for all } e \in \mathbb{S}^{N-1},
\end{equation}
where
\begin{equation*}
	\bar{\rho}=\bar{\rho}(s) := 1+e^{2 \epsilon s}, \quad \forall s \in \mathbb{R},
\end{equation*}
and $\epsilon$ is a positive constant satisfying $0<\epsilon \ll \kappa$. Under normalization 
\eqref{2.7}, one obtains the following theorems.
\begin{theorem}(\cite[Theorem 2.8 and Remark 2.9]{LyuY2024}) \label{Theorem 2.8}
     Suppose that (F1)-(F4) hold and let $e \in \mathbb{S}^{N-1}$. Assume that $\left(U_e, 
     c_e\right)$ is the unique pulsating front of Eq.\eqref{1.1}. Then the mapping $e \in 
     \mathbb{S}^{N-1} \mapsto U_e$ is continuous under the topology 
     $\|\cdot\|_{L^{\infty}\left(\mathbb{R} \times \mathbb{R}^N\right)}$, by normalization 
     \eqref{2.7}. Moreover, the 
	 conclusions in Theorem \ref{Theorem 2.5} still hold for the normalization \eqref{2.7}.
\end{theorem}
\begin{theorem}(\cite[Proposition 3.9]{LyuY2024})\label{Theorem 2.12}
	Normalize the profile $U_e$ as \eqref{2.7}. Then for any constant $q>0$, there exist two small positive constants $\gamma$ and $r$ independent of $e \in \mathbb{S}^{N-1}$, such that
	\begin{equation*}
		\gamma \leq U_e(s, z) \leq 1-\gamma \text { and }-\partial_s U_e(s, z) \geq r, \; \forall(s, z) \in[-q, q] \times \mathbb{R}^N.
	\end{equation*}
\end{theorem}

For any $b \in \mathbb{R}^N \backslash\{0\}$, define
\begin{equation*}
	U_b:=U_{\frac{b}{|b|}} \;\text { and }\; c_b:=c_{\frac{b}{|b|}}.
\end{equation*}
It follows from \cite[Theorem 2.10]{LyuY2024} that $U_b$ and $c_b$ are doubly continuously 
Fr\'{e}chet differentiable in $b \in \mathbb{R}^N$ everywhere at $\mathbb{R}^N 
\backslash\{0\}$ under the topology $\|\cdot\|_{C^2\left(\mathbb{R} \times \mathbb
{R}^N\right) \times \mathbb{R}}$ by normalization \eqref{2.7}.

The norm of the Fr\'{e}chet derivatives with respect to the propagation direction $e \in 
\mathbb{S}^{N-1}$ are given by the following formulas.
\begin{equation*}
\begin{aligned}
	\left\|U_e^{\prime}\right\| & =\sup _{h \in \mathbb{R}^N} \frac{\left\|U_e^{\prime} \cdot h\right\|_{L^{\infty}\left(\mathbb{R} \times \mathbb{L}^N\right)}}{|h|},  \\
	\left\|U_e^{\prime \prime}\right\|&=\sup _{\left(h_1, h_2\right) \in \mathbb{R}^N \times \mathbb{R}^N} \frac{\left\|\left(U_e^{\prime \prime} \cdot h_1\right) \cdot h_2\right\|_{L^{\infty}\left(\mathbb{R} \times \mathbb{L}^N\right)}}{\left|h_1\right|\left|h_2\right|}, \\
	\left\|\partial_{s} U_e^{\prime}\right\| & =\sup _{h \in \mathbb{R}^N} \frac{\left\|\partial_{s} U_e^{\prime} \cdot h\right\|_{L^{\infty}\left(\mathbb{R} \times \mathbb{L}^N\right)}}{|h|}, \\
	\left\|\partial_{x_i} U_e^{\prime}\right\| & =\sup _{h \in \mathbb{R}^N} \frac{\left\|\partial_{x_i} U_e^{\prime} \cdot h\right\|_{L^{\infty}\left(\mathbb{R} \times \mathbb{L}^N\right)}}{|h|} \quad(i=1, \ldots, N),
\end{aligned}
\end{equation*}
and
\begin{equation*}
\left\|c_e^{\prime}\right\|=\sup _{h \in \mathbb{R}^N} \frac{\left|c_e^{\prime} \cdot h\right|}{|h|}<+\infty, \quad\left\|c_e^{\prime \prime}\right\|=\sup _{\left(h_1, h_2\right) \in \mathbb{R}^N \times \mathbb{R}^N} \frac{\left|\left(c_e^{\prime \prime} \cdot h_1\right) \cdot h_2\right|}{\left|h_1 \| h_2\right|}<+\infty .
\end{equation*}
The estimates of Fr\'{e}chet derivatives of $U_e$ are given in  the following theorem.
\begin{theorem}(\cite[Proposition 4.9]{LyuY2024})\label{Theorem 2.13}
	 Suppose that (F1)-(F4) hold. Normalize $U_e$ as \eqref{2.7}. Then $\nabla U_b$ is Fr\'{e}chet 
	 differentiable in $b$, and	
	 \begin{equation*}
	 	\left(\nabla U_b\right)^{\prime} \cdot h_1=\nabla\left(U_b^{\prime} \cdot h_1\right),
	 \end{equation*}
	where $\nabla$ denotes the gradient operator with respect to $(s, z) \in \mathbb{R} \times \mathbb{R}^N$. Moreover, for any $e, h_1, h_2 \in \mathbb{S}^{N-1}$, there exists a positive constant $M_1$ independent of $e, h_1, h_2$ and $\kappa_2$ such that
	\begin{equation*}
	\left|\left(U_e^{\prime} \cdot h_1\right)(s, z)\right|,\left|\left(U_e^{\prime \prime} \cdot h_2 \cdot h_1\right)(s, z)\right|,\left|\nabla\left(U_b^{\prime} \cdot h_1\right)(s, z)\right| \leq M_1 e^{-\frac{\kappa}{2} s}
    \end{equation*}
	for all $(s, z) \in[0,+\infty) \times \mathbb{R}^N$, and
	\begin{equation*}
	\left|\left(U_e^{\prime} \cdot h_1\right)(s, z)\right|,\left|\left(U_e^{\prime \prime} \cdot h_2 \cdot h_1\right)(s, z)\right|,\left|\nabla\left(U_b^{\prime} \cdot h_1\right)(s, z)\right| \leq M_1 e^{\frac{\kappa_2}{2} s}
    \end{equation*}
	for all $(s, z) \in(-\infty, 0] \times \mathbb{R}^N$.
\end{theorem}

Eventually, we introduce the definitions of sub-invasion and super-invasion, and present a 
comparison principle for super- and subsolutions. 
\begin{definition}(\cite{Guo H2025})\label{Definition 2.14}
 By a sub-invasion (resp. super-invasion) $u$ of 0 by 1, we mean that $u \in[0,1]$ is a subsolution 
 (resp. supersolution) of \eqref{1.1} satisfying \eqref{1.10} and
\begin{equation*}
\Omega_t^{+} \supset \Omega_s^{+} \text { for all } t \geq s,
\end{equation*}
and
\begin{equation*}
d\left(\Gamma_t, \Gamma_s\right) \rightarrow+\infty \text { as }|t-s| \rightarrow+\infty.
\end{equation*}
\end{definition}
As a result, one gets the following lemma of the comparison principle, we prove it  in the 
appendix.
\begin{lemma}\label{Lemma 2.15}
	Assume that $u(t, z)$ is a sub-invasion of 0 by 1 of \eqref{1.1} associated to the families $\left(\Omega_t^{ \pm}\right)_{t \in \mathbb{R}}$ and $\left(\Gamma_t\right)_{t \in \mathbb{R}}$, and $\widetilde{u}(t, z)$ is a super-invasion of 0 by 1 of \eqref{1.1} with sets $\left(\widetilde{\Omega}_t^{ \pm}\right)_{t \in \mathbb{R}}$ and $\left(\widetilde{\Gamma}_t\right)_{t \in \mathbb{R}}$. If $\widetilde{\Omega}_t^{-} \subset \Omega_t^{-}$ for all $t \in \mathbb{R}$ and $u$, $\widetilde{u}$, $\nabla_{z} u$, $\nabla_{z} \widetilde{u}$, $\partial_{t} u$, $\partial_{t} \widetilde{u}$ are all globally bounded, then there is a smallest $T \in \mathbb{R}$ such that
	\begin{equation*}
	\tilde{u}(t+T, z) \geq u(t, z) \text { for all }(t, z) \in \mathbb{R} \times \mathbb{R}^N.
    \end{equation*}
	Moreover, there exists a sequence $\left\{\left(t_n, z_n\right)\right\}_{n \in \mathbb{N}}$ in $\mathbb{R} \times \mathbb{R}^N$ such that
	\begin{equation*}
	\left\{d\left(z_n, \Gamma_{t_n}\right)\right\}_{n \in \mathbb{N}} \text { is bounded and } \widetilde{u}\left(t_n+T, z_n\right)-u\left(t_n, z_n\right) \rightarrow 0 \text { as } n \rightarrow+\infty.	
    \end{equation*}
\end{lemma}
\subsection{Construction of a hypersurface}
Now let us turn to the construction of a hypersurface.  Let $e_0 \in \mathbb{S}^{N-1}$ and $n$ 
$(\geq 2)$ unit vectors $e_i \in \mathbb{S}^{N-1}$ $(i=1,2, \ldots, n)$ such that $e_i \neq e_j$ for 
any $i \neq j$ and $e_i \cdot e_0>0$ for each $i \in\{1, \ldots, n\}$. Suppose $Q_i$ is the 
hyperplane determined by $e_i$, namely,
\begin{equation*}
	Q_i=\left\{z \in \mathbb{R}^N ; z \cdot e_i=0\right\}.
\end{equation*}
We know that $Q_i$ and $Q_j$ are not parallel, since $e_i \neq e_j$ for $i \neq j$. Let $\mathcal{Q}$ be the polytope surrounded by $Q_1, Q_2, \ldots, Q_n$, namely,
\begin{equation*}
\mathcal{Q}=\left\{z \in \mathbb{R}^N ; \min_{1 \leq i \leq n} z \cdot e_i>0\right\}.
\end{equation*}
Since $e_i \cdot e_0>0$ for any $i \in\{1, \ldots, n\}$, the polytope $\mathcal{Q}$ is unbounded. Define $\partial \mathcal{Q}=\left\{z \in \mathbb{R}^N ; \min _{1 \leq i \leq n} \right.$ $\left.z \cdot e_i=0\right\}$ as the boundary of $\mathcal{Q}$. The joint part where $\partial \mathcal{Q}$ and $Q_i$ intersect is called a facet of the polytope, and the intersecting part is defined as $\widetilde{Q}_i=\partial \mathcal{Q} \cap Q_i$. Let $\mathcal{R}_{i j}=\widetilde{Q}_i \cap \widetilde{Q}_j$ $(i \neq j)$ be the ridges as the intersection of $\widetilde{Q}_i$ and $\widetilde{Q}_j$ for $i \neq j$. Denote $\mathcal{R}$ as the set of all ridges of $\mathcal{Q}$.

Since Eq.\eqref{1.1} remains unchanged under the rotation of the coordinates, without loss of generality, we consider the case $e_0=(0, 0,\ldots,1)$. Denoted as $z:=(x, y)$ and $y:= z_N$, then \eqref{1.1} can be rewritten as
\begin{equation*}
	u_t-\Delta_{x, y} u=f(x, y, u) \ \text {in}\ (t, x, y) \in \mathbb{R} \times \mathbb{R}^{N-1} \times \mathbb{R},
\end{equation*} 
and \eqref{2.1} can be rewritten as
\begin{equation}\label{2.2.1}
	c_e \partial_s U_e+\partial_{s s} U_e+2 \nabla_{x, y} \partial_s U_e \cdot e+\Delta_{x, y} U_e
	+f\left(x, y, U_e\right)=0,
\end{equation}
where $e\in \mathbb{S}^{N-1}$, $\nabla_{x, y}=\nabla_x+\partial_y$ and $(s, x, y) \in \mathbb{R} \times \mathbb{R}^{N-1} \times \mathbb{R}$. 

  We can take $n$ $(\geq 2)$ unit vectors $\nu_i \in \mathbb{S}^{N-2}$ and angles $\theta_i\in (0, \pi / 2]$ $(i=1,2, \ldots, n)$ such that $(\nu_i,\theta_i)\neq(\nu_j,\theta_j)$ for any $i \neq j$, and then define $e_i=\left(\nu_i \cos \theta_i, \sin \theta_i\right)$ for all $i \in\{1, 2, \ldots, n\}$. For any $(x, y) \in\mathbb{R}^{N-1}\times \mathbb{R}$ and all $i \in\{1, 2, \ldots, n\}$, define
\begin{equation*}
q_i(x, y):=x \cdot \nu_i \cos \theta_i+y \sin \theta_i.
\end{equation*}
Associated to each $q_i$ is the hyperplane in $ \mathbb{R}^N$,
\begin{equation*}
Q_i=\left\{(x, y) \in  \mathbb{R}^N ; q_i(x, y)=0\right\}.
\end{equation*}
Note that $q_i$ is the signed distance function to $Q_i$ and $\left\{y=\psi_i( x)\right\}$ is a graph in the $y$-direction, where
\begin{equation*}
\psi_i( x)=-x \cdot \nu_i \cot \theta_i
\end{equation*}
Then, by definitions of $q_i$ and $Q_i$, one has
\begin{equation*}
\mathcal{Q}=\left\{(x, y) \in \mathbb{R}^N ; \min _{1 \leq i \leq n} q_i(x, y) \geq 0\right\}.
\end{equation*}
Since $\partial \mathcal{Q}$ is the boundary of $\mathcal{Q}$, it has the form $\{y=\psi( x)\}$, where
\begin{equation*}
\psi(x):=\max _{1 \leq i \leq n} \psi_i(x).
\end{equation*}
Therefore, $\partial \mathcal{Q}=\cup_{i=1}^n \widetilde{Q}_i$, where $\widetilde{Q}_i=\partial \mathcal{Q} \cap Q_i$ are its facets. Let $\widehat{\mathcal{R}}$ be the projection of $\mathcal{R}$ on the $x$-plane, namely,
\begin{equation*}
\widehat{\mathcal{R}}:=\left\{x \in  \mathbb{R}^{N-1} ; \text { there exists one } y \in \mathbb{R} \text { such that }( x, y) \in \mathcal{R}\right\}.
\end{equation*}
Let $\widehat{Q}_i$ be the projection of $\widetilde{Q}_i$ on the $ x$-plane, namely,
\begin{equation*}
\begin{aligned}
	\widehat{Q}_i & :=\left\{ x\in  \mathbb{R}^{N-1} ; \text { there exist } y \in \mathbb{R} \text { such that }(x, y) \in \widetilde{Q}_i\right\} \\
	& =\left\{x \in  \mathbb{R}^{N-1} ; \psi_i(x)=\max _{1 \leq j \leq n} \psi_j(x)\right\} .
\end{aligned}
\end{equation*}
It follows from the graph properties of $\partial \mathcal{Q}$ that $\cup_{i=1}^n \partial \widehat{Q}_i=\widehat{\mathcal{R}}$ and $\cup_{i=1}^n \widehat{Q}_i=\mathbb{R}^{N-1}$.
Let $y=\varphi(x)$ be the surface determined by 
\begin{equation}\label{2.9}
	\sum_{i=1}^n e^{-q_i(x, y)}=1.
\end{equation}
By the implicit function theorem, one has the existence of $y=\varphi(x)$ and $\varphi \in C^{\infty}\left( \mathbb{R}^{N-1}\right)$. Define $\Sigma :=\{y=\varphi( x)\}$ as the graph of $\varphi$, $\hat{q}_i( x):=q_i( x, \varphi( x))$ and 
\begin{equation*}
h(x):=\sum_{i, j \in\{1, 2, \ldots, n\} ; i\neq j} e^{-\left(\hat{q}_i( x)+\hat{q}_j( x)\right)} \; \text{ for } x\in\mathbb{R}^{N-1} .
\end{equation*}
It should be pointed out that the above construction of the hypersurface comes from \cite{H. Guo2024,Guo H2025}.

By similar arguments in the proof of \cite[Lemma 3.1]{H. Guo2024}, we can get the following lemma which implies $h(x)$ is a measurement of flatness for the graph $\Sigma$.
\begin{lemma}\label{Theorem 2.16}
	The graph $\Sigma$ satisfies the following properties:
\begin{description}
\item[(i)] $\Sigma \subset \mathcal{Q}$;
\item[(ii)] $\Sigma$ stays at finite distance from $\partial \mathcal{Q}$, or equivalently $\sup _{ \mathbb{R}^{N-1}}|\varphi-\psi|<+\infty$;
\item[(iii)] $\Sigma$ approaches $\partial \mathcal{Q}$ exponentially away from $\mathcal{R}$, or equivalently, there is a constant $C>0$ such that
	\begin{equation}\label{2.10}
		|\varphi(x)-\psi(x)| \leq C \exp \left\{-\frac{1}{C} d(x, \widehat{\mathcal{R}})\right\} \text { or } C h( x), \quad \forall x \in \mathbb{R}^{N-1}.
	\end{equation}
\end{description}
\end{lemma}
Then one has $\min _{1 \leq i \leq n} q_i(x, y) \geq 0$ for every point $(x, y) \in \Sigma$ by \eqref{2.9}. Besides, $\varphi(x) \geq \psi_i( x)$ for any fixed $i \in\{1,2, \ldots, n\}$ and $x \in \widehat{Q}_i$ by (i) of Lemma \ref{Theorem 2.16} and the definition of $\mathcal{Q}$. Thus, by definitions of $\psi(x)$ and $\widehat{Q}_i$, we obtain $\varphi(x) \geq \psi(x)$ in $\mathbb{R}^{N-1}$. By \cite[Section 3.1]{H. Guo2024}, $h(x)$ is decaying exponentially in all $\widehat{Q}_i$ as $d(x, \widehat{\mathcal{R}}) \rightarrow+\infty$, namely, $\hat{q}_i(x) \rightarrow 0$ and $\hat{q}_j(x) \rightarrow+\infty$ for all $j \neq i$ and $x \in \widehat{Q}_i$ as $\operatorname{dist}(x, \widehat{\mathcal{R}}) \rightarrow+\infty$. Moreover, the surface $y=\varphi(x)$ has the following properties.
\begin{lemma}\cite[Lemma 2.6]{Guo H2025} \label{Lemma 2.17}
 There exists $C_3>0$ such that for every $i \in\{1, \ldots, n\}$,
\begin{equation}\label{2.11}
\left|\nabla \varphi(x)+\nu_i \cot \theta_i\right| \leq C_3 h(x), \quad \text { for } x \in \widehat{Q}_i ,
\end{equation}
and
\begin{equation}\label{2.12}
	\left|\nabla^2 \varphi(x)\right|,\left|\nabla^3 \varphi(x)\right| \leq C_3 h(x), \quad \text { for } x \in \mathbb{R}^{N-1}.
\end{equation}
\end{lemma}
\subsection{Main Results}
Now we present our main results. In the following, let $(U_{e},c_{e})$ be the unique pulsating 
front with the propagation direction $e\in\mathbb{S}^{N-1}$ in the sense of Definition \ref{De 
1.2}. Under the normalization \eqref{2.7}, define
\begin{equation*}
\underline{V}(t, z):=\max _{1 \leq i \leq n}\left\{U_{e_i}\left(z \cdot e_i-c_{e_i} t, z\right)\right\},
\end{equation*}
which is a subsolution of \eqref{1.1}. The first result indicates the existence of a curved front that converges to pulsating fronts along its asymptotic planes under certain conditions. 
\begin{theorem}\label{Theorem 2.18}
	 For any $\left\{e_i\right\}_{i=1}^n$ of $\mathbb{S}^{N-1}$ such that
	\begin{description}
	\item[(i)] $e_i \cdot e_0>0$ for a fixed $e_0 \in \mathbb{S}^{N-1}$ and $e_i \neq e_j$ for $i \neq j$,
	\item[(ii)] $\hat{c}:=g\left(e_i\right) \equiv$ constant for any $i \in\{1, \ldots, n\}$, where $g(z):=c_{\frac{z}{|z|}} /\left(\frac{z}{|z|} \cdot e_0\right)$ for $z \in \mathbb{R}^N \backslash\{0\}$,
	\item[(iii)] $\hat{c}>g(e)$ for $e \in \mathcal{L}(\mathcal{Q}) \backslash\left\{e_i\right\}_{i=1}^n$, where $\mathcal{L}(\mathcal{Q}):=\left\{e \in \mathbb{S}^{N-1} ; z \cdot e \geq 0\right.$ for all $\left.z \in \mathcal{Q}\right\}$,
	\item[(iv)] $\nabla g\left(e_i\right) \cdot e_j<0$ for every $i \neq j$,
\end{description}
there exists a transition front $V(t, z)$ of \eqref{1.1} with $\Omega_t^{ \pm}$, $\Gamma_t$ given 
by
	\begin{equation}\label{2.13}
		 \Omega_t^{-}=\mathcal{Q}+\hat{c} t e_0, \quad \Omega_t^{+}=\mathbb{R}^N \backslash \overline{\mathcal{Q}}+\hat{c} t e_0 \text{ and } \; \Gamma_t=\partial \mathcal{Q}+\hat{c} t e_0,
	\end{equation}
	satisfying $V_t(t, z)>0$ for any $(t, z) \in \mathbb{R} \times \mathbb{R}^N$ and
	\begin{equation}\label{2.14}
	\frac{\left|V(t, z)-\underline{V}(t, z)\right|}{\min \left\{1, e^{-v^{\star} \min_{1\leq i\leq n} \left\{\frac{z \cdot e_i}{e_i \cdot e_{0}}-\hat{c} t \right\}}\right\}} \rightarrow 0 \text { uniformly as } d\left(z, \mathcal{R}+\hat{c} t e_0\right) \rightarrow+\infty,
    \end{equation}
    where $v^{\star}$ is a positive constant.

\end{theorem}

The next theorem shows the uniqueness of the curved front $V(t, z)$ given in Theorem \ref{Theorem 2.18}.
\begin{theorem}\label{Theorem 2.19}
	Suppose that (F1)-(F4) hold and (i)-(iv) of Theorem \ref{Theorem 2.18} hold. Let $V(t, z)$ be given in Theorem \ref{Theorem 2.18}. If there exists an entire solution $V_1(t, z)$ of \eqref{1.1} satisfying $0 \leq V_1 \leq 1$ and
	\begin{equation}\label{2.15}
		\left|V_1(t, z)-\underline{V}(t, z)\right| \rightarrow 0, \text { uniformly as } d\left(z, \mathcal{R}+\hat{c} t e_0\right) \rightarrow+\infty.
	\end{equation}
	then $V_1(t, z) \equiv V(t, z)$ in $\mathbb{R} \times \mathbb{R}^N$.
\end{theorem}

Moreover, we prove that the curved front $V(t, z)$ given in Theorem \ref{Theorem 2.18}
is  asymptotically stable. To this end, we consider the following
Cauchy problem:
\begin{equation}\label{Cauchy problem}
	\begin{cases}\partial_t u-\Delta_{z} u=f(z, u) & \text { when } t>0 , z \in \mathbb{R}^N, \\ u(t, z)=u_0 (z) & \text { when } t=0 , z \in \mathbb{R}^N.\end{cases}
\end{equation}
Then we can obtain the stability of the curved front $V(t, z)$ given in Theorem \ref{Theorem 2.18}.
\begin{theorem}\label{Theorem 2.20}
   Suppose that (F1)-(F4) hold and (i)-(iv) of Theorem \ref{Theorem 2.18} hold. Let $V(t, z)$ be given in Theorem \ref{Theorem 2.18}. Assume that $u_0 \in C\left(\mathbb{R}^N,[0,1]\right)$ satisfies
	\begin{equation}\label{2.16}
		\underline{V}(0, z) \leq u_0(z)
	\end{equation}
	for all $z \in \mathbb{R}^N$, and
	\begin{equation}\label{2.17}
	 \frac{\left|u_0(z)-\underline{V}(0, z)\right|}{\min \left\{1, e^{- v \min_{1\leq i\leq n} 
	 \left\{\frac{z \cdot e_i}{e_i \cdot e_{0}} \right\}}\right\}}\rightarrow 0 \text { uniformly as } 
	 d\left(z, \mathcal{R}\right) \rightarrow+\infty
    \end{equation}
	for some constant $v>0$. Then the solution $u(t, z)$ of Cauchy problem \eqref{Cauchy 
	problem} for $t \geq 0$ with initial condition $u(0, z)=u_0(z)$ satisfies
	\begin{equation*}
	\lim _{t \rightarrow+\infty}\|u(t,  \cdot)-V(t, \cdot)\|_{L^{\infty}\left(\mathbb{R}^N\right)}=0.
    \end{equation*}
\end{theorem}
\section{Existence and uniqueness }
\noindent
Set $e_0=(0,0, \ldots, 1)$ and $\left\{e_i\right\}_{i=1}^n$ such that $e_i \cdot e_0>0$ for each $i$ $(1\leq i \leq n)$. Let $y=\varphi(x)$ be the surface determined by $\left\{e_i\right\}_{i=1}^n$ given in Subsection $2.2$ and rescale $y=\varphi(x)$ with the parameter $\alpha>0$ to $y=\varphi(\alpha x) / \alpha$. For the rest of this paper, we always define $\zeta:=\alpha x$ and $\zeta_{i}:=\alpha x_{i}$ for any $i\in\{1, \ldots, N-1\}$. We construct a vector-valued function
\begin{equation}\label{3.1}
e(x)=\left(-\frac{\nabla \varphi(\alpha x)}{\sqrt{1+|\nabla \varphi(\alpha x)|^2}}, \frac{1}{\sqrt{1+|\nabla \varphi(\alpha x)|^2}}\right).
\end{equation}
It can be obtained by directly calculating that 
\begin{equation*}
\partial_{x_i} e(x)=\left(-\frac{\alpha \left[\left(1+|\nabla \varphi|^2\right) \nabla\partial_{\zeta_i} \varphi -\left(\nabla \varphi \cdot \nabla\partial_{\zeta_i} \varphi\right) \nabla \varphi\right]}{\left(1+|\nabla \varphi|^2\right)^{\frac{3}{2}}}, -\frac{\alpha \nabla \varphi \cdot \nabla\partial_{\zeta_i} \varphi}{\left(1+|\nabla \varphi|^2\right)^{\frac{3}{2}}}\right)
\end{equation*}
and
\begin{equation*}
	\partial_{x_i x_j} e(x)= \left(P_{ij}, Q_{ij}\right),
\end{equation*}
where 
\begin{align*}
	P_{i j}&=\frac{\alpha^2}{\left(1+\left|\nabla \varphi\right|{ }^2\right)^{\frac{3}{2}}}\left[-\left(1+\left|\nabla \varphi\right|^2\right) \nabla \partial_{\zeta_i} \partial_{\zeta_j} \varphi+\left(\nabla \partial_{\zeta_i} \varphi \cdot \nabla \partial_{\zeta_j} \varphi\right) \nabla \varphi+\left(\nabla \varphi \cdot \nabla \partial_{\zeta_j} \varphi\right) \nabla \partial_{\zeta_i} \varphi\right. \\
	&\left.\quad+\left(\nabla \varphi \cdot \nabla \partial_{\zeta_i} \varphi\right) \nabla \partial_{\zeta_j} \varphi+\left(\nabla \varphi \cdot \nabla\partial_{\zeta_i} \partial_{\zeta_j} \varphi\right) \nabla \varphi-\frac{3\left(\nabla \varphi \cdot \nabla \partial_{\zeta_i} \varphi\right)\left(\nabla \varphi \cdot \nabla \partial_{\zeta_j} \varphi\right)}{1+\left|\nabla \varphi\right|^2} \nabla \varphi \right]  
\end{align*}
and 
\begin{equation*}
	Q_{i j}=-\frac{\alpha^2}{\left(1+\left|\nabla \varphi\right|^2\right)^{\frac{3}{2}}}\left[ \nabla \varphi \cdot \nabla \partial_{\zeta_i} \partial_{\zeta_j} \varphi+\nabla \partial_{\zeta_i} \varphi \cdot \nabla \partial_{\zeta_j} \varphi-\frac{3\left(\nabla \varphi \cdot \nabla \partial_{\zeta_i} \varphi\right)\left(\nabla \varphi \cdot \nabla \partial_{\zeta_j} \varphi\right)}{1+\left|\nabla \varphi\right|^2}\right]
\end{equation*}
for every $i, j \in\{1, \ldots, N-1\}$, and $\varphi$ takes the value at $\alpha x$. By \eqref{2.11} and \eqref{2.12}, there exist two positive constants $M_2$ and $M_3$ such that
\begin{equation}\label{3.2}
	\left|\partial_{x_i} e(x)\right| \leq \alpha M_2 h(\alpha x) \text { and }\left|\partial_{x_i x_j} e(x)\right| \leq \alpha^2 M_3 h(\alpha x)
\end{equation}
for all $x \in \mathbb{R}^{N-1}$ and $i, j \in\{1, \ldots, N-1\}$.
Define a smooth function $\omega(s)\in C^\infty(\mathbb{R})$ satisfying $\omega^{\prime}(s) \geq 0$ and
\begin{equation}\label{3.3}
	 \begin{cases}\omega(s)=0,  \text { if } s \leq-1, \\ 0<\omega(s)<1,  \text { if } s \in(-1,1), \\ \omega(s)=1,  \text { if } s \geq 1.\end{cases}
\end{equation}
Now we construct two functions $\xi$ and $\eta$, where
\begin{equation}\label{3.4}
	\xi(t, x, y)=\frac{y-\hat{c} t-\varphi(\alpha x) / \alpha}{\sqrt{1+|\nabla \varphi(\alpha x)|^2}}, \quad \eta(t, x, y)=y-\hat{c} t-\varphi(\alpha x) / \alpha,
\end{equation}
where $\hat{c}$ is a positive constant given in Theorem \ref{Theorem 2.18}. It follows from \cite[Lemma 3.1]{Guo H2025} that, assuming (i)-(iv) of Theorem \ref{Theorem 2.18} hold, there exists $C_{0}>0$ such that 
\begin{equation}\label{L 3.1}
	-\xi_t-c_{e(x)} \geq C_{0} h(\alpha x), \; \text { for some } C_{0}>0.
\end{equation}
 \subsection{Construction of the supersolution}
\noindent
\begin{lemma}\label{Lemma 3.2}
    Suppose that (F1)-(F4) hold and (i)-(iv) of Theorem \ref{Theorem 2.18} hold. Then there exists 
    a constant $\beta^*>0$ such that for any $\beta \in\left(0, 
    \beta^*\right]$ there exist positive constants $\varepsilon_0^{+}(\beta)$ and $\alpha_0^{+}(\beta, 
    \varepsilon)$ such that for any 
    $0<\varepsilon<\varepsilon_0^{+}(\beta)$ and any $0<\alpha<\alpha_0^{+}(\beta, \varepsilon)$, the function
	\begin{equation}\label{3.5}
	\overline{V}(t, x, y):=U_{e(x)}(\xi, x, y)+\varepsilon h(\alpha x) \times\left[U_{e_i}^\beta(\eta, 
	x, y) \omega(\xi)+(1-\omega(\xi))\right]
    \end{equation}
	is a supersolution of Eq.\eqref{1.1}, where $e_i$ is an arbitrary fixed unit vector in $\left\{e_i\right\}_{i=1}^n$ and $\left\{e_i\right\}_{i=1}^n$ are given in Subsection $2.2$. Moreover,
	\begin{gather}
		\left|\bar{V}(t, x, y)-\underline{V}(t, x, y)\right| \leq 2\varepsilon, \; d\left((x,y), 
		\mathcal{R}+\hat{c} t e_0\right) \rightarrow+\infty,\label{3.6}\\
		\bar{V}(t, x, y) \geq \underline{V}(t, x, y) \ \text {in}\  \mathbb{R}\times\mathbb{R}^N, 
		\label{3.7}\\
		\frac{\partial}{\partial t} \bar{V}(t, x, y)>0 \ \text {in}\ 
		\mathbb{R}\times\mathbb{R}^N.\label{3.8}
	\end{gather}
\end{lemma}
\begin{proof} We divide the proof into three steps.

{\it Step 1: proof of $\bar{V}(t, x, y)$ being a supersolution.} Our approach is to find 
two numbers $X^{\prime}>1$ and $X^{\prime \prime}>1$, and prove the inequality
	\begin{equation*}
	\mathcal{L} \bar{V}(t, x, y):=\partial_t \bar{V}(t, x, y)-\Delta_{x, y} \bar{V}(t, x, y)-f\left(x, y, \bar{V}(t, x, y)\right) \geq 0, \quad \forall(t, x, y) \in \mathbb{R}\times\mathbb{R}^N,
	\end{equation*}
	in three cases $\xi>X^{\prime}, \xi<-X^{\prime \prime}$, and $\xi \in\left[-X^{\prime \prime}, X^{\prime}\right]$, respectively. As a matter of convenience, denote
	\begin{equation*}
	I_1:=\left(\partial_t-\Delta_{x, y}\right)\left(U_{e(x)}(\xi, x, y)\right) \text { and } I_2:=\left(\partial_t-\Delta_{x, y}\right)\left(\varepsilon h(\alpha x) \times U_{e_i}^\beta(\eta, x, y)\right).
    \end{equation*}
	Then we can obtain that
	\begin{equation}\label{3.9}
		\begin{aligned}
			I_1= & \partial_{\xi} U_{e(x)} \partial_t \xi-\Delta_{x, y} U_{e(x)}-2\sum_{k=1}^{N-1} \partial_{x_{k} \xi} U_{e(x)} \partial_{x_{k}} \xi -2 \partial_{y \xi} U_{e(x)} \partial_{y} \xi \\
			& -\partial_\xi U_{e(x)}\left(\sum_{k=1}^{N-1}\partial_{x_{k}x_{k}} \xi+\partial_{yy} \xi\right)-\partial_{\xi \xi} U_{e(x)}\left(\sum_{k=1}^{N-1}(\partial_{x_{k}} \xi)^2+(\partial_{y} \xi)^2\right)-\sum_{k=1}^{N-1}U_{e(x)}^{\prime \prime} \cdot \partial_{x_{k}}e(x) \cdot \partial_{x_{k}}e(x) \\
			& -\sum_{k=1}^{N-1}U_{e(x)}^{\prime} \cdot \partial_{x_{k}x_{k}}e(x)-2 \sum_{k=1}^{N-1}\partial_{x_k} U_{e(x)}^{\prime} \cdot \partial_{x_{k}}e(x)-2\sum_{k=1}^{N-1} \partial_\xi U_{e(x)}^{\prime} \cdot \partial_{x_{k}}e(x) \partial_{x_{k}} \xi,
		\end{aligned}
	\end{equation}
where $U_{e(x)}$ and all of its derivatives are evaluated at $(\xi(t, x, y), x, y)$. We can figure out that
\begin{align}\label{3.10}
&\partial_t \xi =-\frac{\hat{c}}{\sqrt{1+|\nabla \varphi|^2}}, \nonumber\\ 
&\partial_y \xi  
=\frac{1}{\sqrt{1+|\nabla \varphi|^2}},\nonumber \\ 
&\nabla_x \xi  =-\frac{\nabla \varphi}{\sqrt{1+|\nabla \varphi|^2}}-\alpha \frac{ \nabla^2 \varphi 
\cdot \nabla \varphi}{1+|\nabla \varphi|^2} \xi,\nonumber\\  
&\nabla_{x, y} \xi-e(x)=\left(-\alpha \frac{\nabla^2 \varphi \cdot \nabla \varphi}{1+|\nabla 
\varphi|^2} \xi,  0\right), \\  	
&\nabla_{x, y} \xi+e(x)=\left(-2 \frac{\nabla \varphi}{\sqrt{1+|\nabla \varphi|^2}}-\alpha 
\frac{\nabla^2 \varphi \cdot \nabla \varphi}{1+|\nabla \varphi|^2} \xi,  \frac{2}{\sqrt{1+|\nabla 
\varphi|^2}}\right), \nonumber\\  
&\Delta_{x, y} \xi=-\alpha \frac{\Delta \varphi}{\sqrt{1+|\nabla \varphi|^2}}+ 2\alpha \frac{\nabla 
\varphi \cdot (\nabla^2 \varphi \cdot \nabla \varphi)}{\left(1+|\nabla 
\varphi|^2\right)^{\frac{3}{2}}}+3 \alpha^2 \frac{\left| \nabla^2 \varphi \cdot \nabla \varphi 
\right|^2}{\left(1+|\nabla \varphi|^2\right)^2} \xi-\alpha^2 \frac{\nabla \cdot\left(\nabla^2 \varphi 
\cdot \nabla \varphi\right)}{1+|\nabla \varphi|^2} \xi,\nonumber
\end{align}
where $\varphi$ takes value at $\alpha x$ and $\xi$ takes value at $(t, x, y)$. By \eqref{2.12} and the boundedness of $|\nabla \varphi(\alpha x)|$, there is a constant $C_4>0$ such that for any $\xi \in \mathbb{R}$,
\begin{equation*}
	\left|\Delta_{x, y} \xi\right| \leq \alpha C_4(1+\alpha|\xi|) h(\alpha x), \; | \nabla_{x, y} \xi-e(x)|\leq \alpha C_4| \xi | h(\alpha x) \text{ and }  \left|\nabla_{x, y} \xi+e(x)\right| \leq \alpha C_4(1+|\xi|) h(\alpha x).
\end{equation*}
It follows from Theorem \ref{Theorem 2.8} that $\left|\partial_{\xi} U_{e(x)} 
\xi\right|,\left|\partial_{\xi \xi} U_{e(x)} \xi\right|,\left|\partial_{\xi \xi} U_{e(x)} \xi^2\right|$ and 
$\left|\nabla_{x, y} \partial_{\xi} U_{e(x)} \xi\right|$ are bounded uniformly for $(\xi, x, y) \in 
\mathbb{R} \times \mathbb{R}^N$. Note that $\nabla_{x, y} \xi \cdot \nabla_{x, y} \xi-e(x) \cdot 
e(x)=\left(\nabla_{x, y} \xi-e(x)\right) \left(\nabla_{x, y} \xi+e(x)\right)$.

By \eqref{2.2.1} and \eqref{3.9}, we have that
\begin{equation}\label{3.11}
	\begin{aligned}
		I_1= & (\partial_t \xi+c_{e(x)})\partial_{\xi} U_{e(x)}  -2\sum_{k=1}^{N-1} \partial_{x_{k} \xi} U_{e(x)} \partial_{x_{k}} \xi -2 \partial_{y \xi} U_{e(x)} \partial_{y} \xi + 2\nabla_{x, y} \partial_{\xi} U_{e(x)} \cdot e(x)\\
		& -\partial_\xi U_{e(x)}\left(\sum_{k=1}^{N-1}\partial_{x_{k}x_{k}} \xi+\partial_{yy} \xi\right)-\partial_{\xi \xi} U_{e(x)}\left(\sum_{k=1}^{N-1}(\partial_{x_{k}} \xi)^2+(\partial_{y} \xi)^2-1\right) \\
		& -\sum_{k=1}^{N-1}U_{e(x)}^{\prime \prime} \cdot \partial_{x_{k}}e(x) \cdot \partial_{x_{k}}e(x)-\sum_{k=1}^{N-1}U_{e(x)}^{\prime} \cdot \partial_{x_{k}x_{k}}e(x)-2 \sum_{k=1}^{N-1}\partial_{x_k} U_{e(x)}^{\prime} \cdot \partial_{x_{k}}e(x)\\
		&-2\sum_{k=1}^{N-1} \partial_\xi U_{e(x)}^{\prime} \cdot \partial_{x_{k}}e(x) \partial_{x_{k}} \xi+f\left(x, y, U_{e(x)}\right)
	\end{aligned}
\end{equation}
where $U_{e(x)}$ and all of its derivatives are evaluated at $( \xi(t, x, y), x, y)$. By \eqref{L 3.1}, there is a positive constant $C_{0}$ such that
\begin{equation}\label{3.12}
	\partial_{t}\xi+c_{e(x)}=-\frac{\hat{c}}{\sqrt{1+|\nabla \varphi|^2}}+c_{e(x)} \leq- C_{0} h(\alpha x)<0 \text { for all } x \in \mathbb{R}^{N-1}.
\end{equation}
Calculating $I_2$ gives that
\begin{equation}\label{3.13}
	\begin{aligned}
		I_2= & \varepsilon \beta h(\alpha x) U_{e_i}^{\beta-1} \partial_{\eta} U_{e_i} \partial_{t} \eta-\varepsilon \alpha^2  U_{e_i}^\beta \sum_{k=1}^{N-1}(\partial_{\zeta_{k}\zeta_{k}}h(\alpha x)) \\
		& -2 \varepsilon \beta \alpha  U_{e_i}^{\beta-1} \times\left[\sum_{k=1}^{N-1}\partial_{x_k} U_{e_i}  \partial_{\zeta_{k}}h(\alpha x)+\partial_{\eta} U_{e_i}\sum_{k=1}^{N-1}\partial_{\zeta_{k}}h(\alpha x) \partial_{x_{k}}\eta\right] \\
		& -\varepsilon \beta(\beta-1) h(\alpha x) U_{e_i}^{\beta-2} \times\left[\sum_{k=1}^{N-1}\left(\partial_{x_k} U_{e_i}+\partial_{\eta} U_{e_i} \partial_{x_k}\eta\right)^2+\left(\partial_y U_{e_i}+\partial_{\eta} U_{e_i} \partial_{y}\eta\right)^2\right] \\
		& -\varepsilon \beta h(\alpha x) U_{e_i}^{\beta-1} \times\left[\Delta_{x, y} U_{e_i}+2 \nabla_{x, y} \partial_{\eta} U_{e_i} \nabla_{x, y}\eta +\partial_{\eta\eta} U_{e_i}\left(\sum_{k=1}^{N-1}(\partial_{x_k}\eta)^2+(\partial_{y}\eta)^2\right)\right. \\
		& \quad\left.+\partial_{\eta} U_{e_i}\left(\sum_{k=1}^{N-1}\partial_{x_k x_k}\eta+\partial_{yy}\eta\right)\right],
	\end{aligned}
\end{equation}
where $U_{e_i}$ and all of its derivatives are evaluated at $(\eta(t, x, y), x, y)$.

\textbf{Case 1:} $\xi(t, x, y)>X^{\prime}$, where $X^{\prime}>1$ is to be chosen.
In this case, $\bar{V}(t, x, y)=U_{e(x)}(\xi, x, y)+\varepsilon h(\alpha x) \times U_{e_i}^\beta(\eta, x, y)$, then
\begin{equation*}
	\mathcal{L} \bar{V}(t, x, y)=I_1+I_2-f\left(x, y, \bar{V}\right).
\end{equation*}
Calculating the derivatives of $\eta(t, x, y)$, we get
\begin{equation}\label{3.14}
	\left\{\begin{array}{l}
		\eta_t=-\hat{c},\;\eta_y=1,\;\eta_{y y}=0 ,\\
		\eta_{x_k}=-\partial_{\zeta_k}\varphi(\alpha x) ,\\
		\eta_{x_k x_k}=-\alpha \partial_{\zeta_k \zeta_k} \varphi(\alpha x) ,\\
		\sum_{k=1}^{N-1}\eta_{x_k}^2+\eta_y^2=\sum_{k=1}^{N-1}(\partial_{\zeta_k}\varphi (\alpha x))^2+1.
	\end{array}\right.
\end{equation}
By virtue of \eqref{3.13} and \eqref{3.14}, it holds that
\begin{equation}\label{3.15}
\begin{aligned}
	I_2= & -\varepsilon U_{e_i}^\beta \times \alpha\left[\alpha \sum_{k=1}^{N-1}\partial_{\zeta_{k}\zeta_{k}}h(\alpha x)+2 \beta \sum_{k=1}^{N-1}\partial_{\zeta_{k}}h(\alpha x) \frac{\partial_{x_k} U_{e_i}-\partial_{\zeta_k}\varphi(\alpha x) \partial_{\eta} U_{e_i}}{U_{e_i}}\right.\\
	&\quad\left.-\beta h(\alpha x) \sum_{k=1}^{N-1} \partial_{\zeta_k \zeta_k} \varphi(\alpha x) \frac{\partial_{\eta} U_{e_i}}{U_{e_i}}\right] \\
	& -\varepsilon U_{e_i}^\beta \times \beta h(\alpha x)\left[\frac{\Delta_{x, y} U_{e_i}+2 \nabla_{x, y} \partial_{\eta} U_{e_i} \cdot\left(-\nabla_{\zeta}\varphi(\alpha x), 1\right)}{U_{e_i}}\right. \\
	&\quad+\beta \frac{\sum_{k=1}^{N-1}\left(\partial_{x_k} U_{e_i}-\partial_{\zeta_k}\varphi(\alpha x)\partial_{\eta} U_{e_i}\right)^2+\left(\partial_y U_{e_i}+\partial_{\eta} U_{e_i}\right)^2}{U_{e_i}^2} \\
	&\quad-\frac{\sum_{k=1}^{N-1}\left(\partial_{x_k} U_{e_i}-\partial_{\zeta_k}\varphi(\alpha x)\partial_{\eta} U_{e_i}\right)^2+\left(\partial_y U_{e_i}+\partial_{\eta} U_{e_i}\right)^2}{U_{e_i}^2} \\
	&\quad\left.+\frac{\partial_{\eta \eta} U_{e_i}}{U_{e_i}}\left(\sum_{k=1}^{N-1}(\partial_{\zeta_k}\varphi (\alpha x))^2+1\right)+\hat{c} \frac{\partial_{\eta} U_{e_i}}{U_{e_i}}\right] 
	=  : J_1+J_2,
\end{aligned}
\end{equation}
where $U_{e_i}$ and all of its derivatives are evaluated at $( \eta(t, x, y), x, y)$. Since \eqref{3.4}, Lemma \ref{Lemma 2.17} and the fact $h=\sum_{i, j \in\{1, \ldots, n\} ; i\neq j}e^{-\left(\hat{q}_i+\hat{q}_j\right)}\le 1$, one has
\begin{equation}\label{3.16}
\xi \leq \eta=\xi \sqrt{1+|\nabla \varphi(\alpha x)|^2} \leq \xi \sqrt{(C_3+\max_{1\le i\le n} \{|\nu_i \cot \theta_i|\})^2+1}.    
\end{equation}
It follows from Theorem \ref{Theorem 2.5} that
\begin{equation}\label{3.17}
	\left\{\begin{array}{l}
		\frac{\Delta_{x, y} U_{e_i}+2 \nabla_{x, y} \partial_{\eta} U_{e_i} \cdot\left(-\nabla_{\zeta}\varphi(\alpha x), 1\right)}{U_{e_i}} \longrightarrow 0 ,\\
		\frac{\sum_{k=1}^{N-1}\left(\partial_{x_k} U_{e_i}-\partial_{\zeta_k}\varphi(\alpha x)\partial_{\eta} U_{e_i}\right)^2+\left(\partial_y U_{e_i}+\partial_{\eta} U_{e_i}\right)^2}{U_{e_i}^2} \longrightarrow c_{e_i}^2\left(\sum_{k=1}^{N-1}(\partial_{\zeta_k}\varphi(\alpha x))^2+1\right) ,\\
		\frac{\partial_{\eta \eta} U_{e_i}}{U_{e_i}}\left(\sum_{k=1}^{N-1}(\partial_{\zeta_k}\varphi (\alpha x))^2+1\right) \longrightarrow c_{e_i}^2\left(\sum_{k=1}^{N-1}(\partial_{\zeta_k}\varphi (\alpha x))^2+1\right),\\
		\hat{c} \frac{\partial_{\eta} U_{e_i}}{U_{e_i}} \longrightarrow-\hat{c} c_{e_i},
	\end{array}\right.
\end{equation}
as $\eta \rightarrow+\infty$ uniformly in $(x, y) \in \mathbb{R}^N$. 
Define
\begin{equation*}
	\beta_1^*:=\frac{\hat{c}}{2 c_{e_i}\left((C_3+\max_{1\le i\le n} \{|\nu_i \cot \theta_i|\})^2+1\right)}.
\end{equation*}
Then for any $\beta \in\left(0, \beta_1^*\right]$, there exists a large enough constant $X_1^{\prime}>1$ such that
\begin{equation}\label{3.18}
	\begin{aligned}
		J_2&\ge-\varepsilon \beta U_{e_i}^\beta  (\eta, x, y) h(\alpha x)\left[\beta c_{e_i}^2\left(\sum_{k=1}^{N-1}(\partial_{\zeta_k}\varphi(\alpha x))^2+1\right)-\hat{c} c_{e_i}\right]\\
		&\ge -\varepsilon \beta U_{e_i}^\beta  (\eta, x, y) h(\alpha x)\left(\frac{\hat{c}c_{e_i}^2}{2 c_{e_i}}-\hat{c} c_{e_i} \right)
		> \varepsilon  U_{e_i}^\beta  (\eta, x, y) h(\alpha x)\times \beta\frac{\hat{c}c_{e_i}}{4 },
	\end{aligned}
\end{equation}
for all $(\eta, x, y) \in\left(X_1^{\prime},+\infty\right) \times \mathbb{R}^N$. By Lemma \ref{Lemma 2.17}, there exists a constant $C_5>0$ such that
\begin{equation}\label{3.19}
	\begin{aligned}
		\left|\nabla h(\alpha x)\right| &=\left|\sum_{i, j \in\{1, \ldots, n\}, i \neq j} 
		e^{-\left(\hat{q}_j(\alpha x)+\hat{q}_k(\alpha x)\right)} \mid\nu_i \cos \theta_i+\nu_j \cos 
		\theta_j+\nabla \varphi(\alpha x)\left(\sin \theta_i+\sin \theta_j\right) \mid\right| \\
		& \leq  C_5 h(\alpha x).
	\end{aligned}
\end{equation}
and
\begin{equation}\label{3.20}
	\begin{aligned}
		& \left|\Delta h(\alpha x)\right|  =\left|\sum_{i, j \in\{1, \ldots, n\}, i \neq j} e^{-\left(\hat{q}_j(\alpha x)+\hat{q}_k(\alpha x)\right)}\left|\nu_i \cos \theta_i+\nu_j \cos \theta_j+\nabla \varphi(\alpha x)\left(\sin \theta_i+\sin \theta_j\right)\right|^2 \right.\\
		& \qquad\qquad\quad\quad\left.+\sum_{i, j \in\{1, \ldots, n\}, i \neq j} e^{-\left(\hat{q}_j(\alpha x)+\hat{q}_k(\alpha x)\right)} \Delta \varphi(\alpha x)\left(\sin \theta_i+\sin \theta_j\right)\right|  \leq  C_5 h(\alpha x) .
	\end{aligned}
\end{equation}
Thus, by \eqref{3.15} and \eqref{3.18}-\eqref{3.20}, there is a constant $\alpha_1^{+}(\beta)>0$ 
such that, for arbitrary $0<\alpha \leq \alpha_1^{+}(\beta)$, it holds that
\begin{equation}\label{3.21}
	I_2=J_1+J_2>\varepsilon  U_{e_i}^\beta  (\eta, x, y) h(\alpha x)\times \beta\frac{\hat{c}c_{e_i}}{8 }
\end{equation}
for all $(\eta, x, y) \in\left(X_1^{\prime},+\infty\right) \times \mathbb{R}^N$. By virtue of 
$\partial_{\xi} U_{e(x)}< 0$, $\partial_{y} \xi=e_{N}(x)$, \eqref{3.2}, \eqref{3.10}-\eqref{3.12}, 
Theorem \ref{Theorem 2.8}, Theorem \ref{Theorem 2.13} and Lemma \ref{Lemma 2.17}, we have
	\begin{equation*}
	\begin{aligned}
		I_1= & (\partial_t \xi+c_{e(x)})\partial_{\xi} U_{e(x)}  -2\sum_{k=1}^{N-1} \partial_{x_{k} \xi} U_{e(x)} (\partial_{x_{k}} \xi-e_{k}(x))-\partial_\xi U_{e(x)}\left(\sum_{k=1}^{N-1}\partial_{x_{k}x_{k}} \xi+\partial_{yy} \xi\right)\\
		& -\partial_{\xi \xi} U_{e(x)}\left(\sum_{k=1}^{N-1}(\partial_{x_{k}} \xi)^2+(\partial_{y} \xi)^2-1\right)-\sum_{k=1}^{N-1}U_{e(x)}^{\prime \prime} \cdot \partial_{x_{k}}e(x) \cdot \partial_{x_{k}}e(x) \\
		& -\sum_{k=1}^{N-1}U_{e(x)}^{\prime} \cdot \partial_{x_{k}x_{k}}e(x)-2 \sum_{k=1}^{N-1}\partial_{x_k} U_{e(x)}^{\prime} \cdot \partial_{x_{k}}e(x)-2\sum_{k=1}^{N-1} \partial_\xi U_{e(x)}^{\prime} \cdot \partial_{x_{k}}e(x) \partial_{x_{k}} \xi+f\left(x, y, U_{e(x)}\right),
	\end{aligned}
\end{equation*}
and
\begin{equation*}
(\partial_t \xi+c_{e(x)})\partial_{\xi} U_{e(x)}\ge - C_{0} h(\alpha x)\partial_{\xi} U_{e(x)}\ge 
C_{0} h(\alpha x)\bar{K} e^{-\frac{3 \kappa}{4} \xi}\ge 0,
\end{equation*}
\begin{equation*}
\begin{aligned}
-2\sum_{k=1}^{N-1} \partial_{x_{k} \xi} U_{e(x)} (\partial_{x_{k}} 
\xi-e_{k}(x))
\ge&-2\nabla_{x, y} \partial_{\xi} U_{e(x)}(\nabla_{x, y}\xi-e(x)) \\
\ge&-2\left|\nabla_{x, 
y} \partial_{\xi} U_{e(x)}(\nabla_{x, y}\xi-e(x))\right| \\
	\ge&-2\alpha C_4 | \xi | h(\alpha x)\left[N(\bar{K} e^{-\frac{3 \kappa}{4} 
	\xi})^2\right]^{\frac{1}{2}},
\end{aligned}
\end{equation*}
\begin{equation*}
	-\partial_\xi U_{e(x)}\left(\sum_{k=1}^{N-1}\partial_{x_{k}x_{k}} \xi+\partial_{yy} 
	\xi\right)\ge-\left|\partial_\xi U_{e(x)}\Delta_{x, y}\xi\right|\ge-\alpha C_4(1+\alpha|\xi|) 
	h(\alpha x)\bar{K} e^{-\frac{3 \kappa}{4} \xi},
\end{equation*}
\begin{equation*}
	\begin{aligned}
	-\partial_{\xi \xi} U_{e(x)}\left(\sum_{k=1}^{N-1}(\partial_{x_{k}} \xi)^2+(\partial_{y} \xi)^2-1\right)&\ge-\left|\partial_{\xi \xi} U_{e(x)}\left(\nabla_{x, y} \xi-e(x)\right) \left(\nabla_{x, y} \xi+e(x)\right)\right|\\
	&\ge-\bar{K} e^{-\frac{3 \kappa}{4} \xi}\alpha^2 C_4^2(1+|\xi|)| \xi |  h^2(\alpha x) \\
	&\ge-\bar{K} e^{-\frac{3 \kappa}{4} \xi}\alpha^2 C_4^2(1+|\xi|)| \xi |  h(\alpha x),
	\end{aligned}
\end{equation*}
\begin{equation*}
	-\sum_{k=1}^{N-1}U_{e(x)}^{\prime \prime} \cdot \partial_{x_{k}}e(x) \cdot 
	\partial_{x_{k}}e(x)\ge -(N-1)M_1 e^{-\frac{\kappa}{2} \xi}\alpha^2 M^2_2 h^2(\alpha x)\ge 
	-(N-1)M_1 e^{-\frac{\kappa}{2} \xi}\alpha^2 M^2_2 h(\alpha x),
\end{equation*}
\begin{equation*}
-\sum_{k=1}^{N-1}U_{e(x)}^{\prime} \cdot \partial_{x_{k}x_{k}}e(x)\ge-(N-1)M_1 e^{-\frac{\kappa}{2} \xi}\alpha M_3  h(\alpha x),
\end{equation*}
\begin{equation*}
-2 \sum_{k=1}^{N-1}\partial_{x_k} U_{e(x)}^{\prime} \cdot \partial_{x_{k}}e(x)\ge-2(N-1)M_1 e^{-\frac{\kappa}{2} \xi}\alpha M_2 h(\alpha x),
\end{equation*}
\begin{equation*}
   -2\sum_{k=1}^{N-1} \partial_\xi U_{e(x)}^{\prime} \cdot \partial_{x_{k}}e(x) \partial_{x_{k}} \xi\ge -(N-1)M_1 e^{-\frac{\kappa}{2} \xi}(\alpha C_4(1+|\xi|) h(\alpha x)+\alpha C_4|\xi| h(\alpha x))\alpha M_2 h(\alpha x),
\end{equation*}
for all $(\xi, x, y) \in[0,+\infty) \times \mathbb{R}^N$. Thus, there exists a constant $\Lambda_1>0$ such that
\begin{equation}\label{3.22}
	I_1 \geq-\alpha\Lambda_1 h(\alpha x) e^{-\frac{\kappa}{4} \xi}+f\left(x, y, U_{e(x)}\right),
\end{equation}
for all $(\xi, x, y) \in[0,+\infty) \times \mathbb{R}^N$, where $\kappa$ is given in Theorem \ref{Theorem 2.3}. It follows from Theorem \ref{Theorem 2.3} that there exists a large enough constant $X_2^{\prime}>1$ such that
\begin{equation}\label{3.23}
U_{e_i}(\eta, x, y) \geq \frac{C_1}{2} e^{-c_{e_i} \eta} \geq \frac{C_1}{2} e^{-K \eta}, \quad \forall(\eta, x, y) \in\left(X_2^{\prime},+\infty\right) \times \mathbb{R}^N,
\end{equation}
where $C_1$ is given in Theorem \ref{Theorem 2.3} and $K$ is given in Theorem \ref{Theorem 2.3}. Set
\begin{equation*}
	\beta_2^*:=\frac{\kappa}{8 K\sqrt{(C_3+\max_{1\le i\le n} \{|\nu_i \cot \theta_i|\})^2+1}}.
\end{equation*}
Since \eqref{3.16} and \eqref{3.23}, for any $\beta \in\left(0, \beta_2^*\right]$, we get 
\begin{align}\label{3.24}
U_{e_i}^\beta(\eta, x, y) \geq \left(\frac{C_1}{2}\right)^\beta e^{-\beta K \eta}\geq \left(\frac{C_1}{2}\right)^\beta e^{- \frac{\kappa\eta}{8 \sqrt{(C_3+\max_{1\le i\le n} \{|\nu_i \cot \theta_i|\})^2+1}}}\geq\left(\frac{C_1}{2}\right)^\beta e^{-\frac{\kappa}{8} \xi},
\end{align}
for all $(\eta, x, y) \in\left(X_2^{\prime},+\infty\right) \times \mathbb{R}^N$. By virtue of \eqref{1.3}, \eqref{3.16} and $\bar{V}(t, x, y)=U_{e(x)}(\xi, x, y)+\varepsilon h(\alpha x) \times U_{e_i}^\beta(\eta, x, y)$, if set $\varepsilon<\varepsilon_{1}^{+}:=\gamma_{\star}$, there exists a sufficiently large constant $X_3^{\prime}>1$ such that
\begin{align}\label{3.25}
	f\left(x, y, U_{e(x)}\right)-f\left(x, y, \bar{V}\right)= 0,\quad\forall(\xi, x, y) \in\left(X_3^{\prime},+\infty\right) \times \mathbb{R}^N.
\end{align}
Define $X^{\prime}:=\max\{X_1^{\prime},X_2^{\prime},X_3^{\prime}\}$ and $\beta^{*}:=\max\{\beta_1^*, \beta_2^*\}$. By \eqref{3.21}, \eqref{3.22}, \eqref{3.24} and \eqref{3.25}, for any $(\xi, x, y) \in\left(X^{\prime},+\infty\right) \times \mathbb{R}^N$, if $\alpha<\alpha^{+}_2(\beta, \varepsilon):=\varepsilon \beta\frac{\hat{c}c_{e_i}}{8 \Lambda_1} \left(\frac{C_1}{2}\right)^\beta e^{\frac{\kappa}{8} X^{\prime}}$, one has
\begin{equation*}
\begin{aligned}
	\mathcal{L} \bar{V} & =I_1+I_2-f\left(x, y, \bar{V}\right) \\
	& \geq-\alpha\Lambda_1 h(\alpha x) e^{-\frac{\kappa}{4} \xi}+\varepsilon  U_{e_i}^\beta  (\eta, x, y) h(\alpha x)\times \beta\frac{\hat{c}c_{e_i}}{8 } \\
	& \geq-\alpha\Lambda_1 h(\alpha x) e^{-\frac{\kappa}{4} \xi}+\varepsilon h(\alpha x)\beta\frac{\hat{c}c_{e_i}}{8 } \left(\frac{C_1}{2}\right)^\beta e^{-\frac{\kappa}{8} \xi} >0.
\end{aligned}
\end{equation*}

\textbf{Case 2:} $\xi(t, x, y)<-X^{\prime\prime}$, where $X^{\prime\prime}>1$ is to be chosen.
In this case, one has $\bar{V}(t, x, y)=U_{e(x)}(\xi, x, y)+\varepsilon h(\alpha x)$. It follows from \eqref{3.11}, \eqref{3.12} and $\partial_{\xi} U_{e(x)}< 0$ that
\begin{equation}\label{3.26}
	\begin{aligned}
		\mathcal{L} \bar{V}= & I_1-\varepsilon \alpha^2   \sum_{k=1}^{N-1}\partial_{\zeta_{k}\zeta_{k}}h(\alpha x)-f\left(x, y, \bar{V}\right) \\
		> & - C_{0} h(\alpha x)\partial_{\xi} U_{e(x)} -2 \sum_{k=1}^{N-1}\partial_{x_k \xi} U_{e(x)}\left(\partial_{x_k }\xi-e_k(x)\right)-\partial_{\xi} U_{e(x)}\left(\sum_{k=1}^{N-1}\partial_{x_k x_k}\xi+\partial_{y y}\xi\right) \\
		& -\partial_{\xi \xi} U_{e(x)}\left(\sum_{k=1}^{N-1}(\partial_{ x_k}\xi)^2+(\partial_{y}\xi)^2-1\right)-\sum_{k=1}^{N-1}U_{e(x)}^{\prime \prime} \cdot \partial_{x_k}e(x) \cdot \partial_{ x_k}e(x)-\sum_{k=1}^{N-1}U_{e(x)}^{\prime} \cdot \partial_{x_k x_k}e(x) \\
		& -2\sum_{k=1}^{N-1} \partial_{x_k} U_{e(x)}^{\prime} \cdot \partial_{x_k}e(x)-2 \partial_{\xi} U_{e(x)}^{\prime} \cdot \sum_{k=1}^{N-1}\partial_{x_k}e(x) \partial_{x_k}\xi\\
		& +f\left(x, y, U_{e(x)}\right)-f\left(x, y, \bar{V}\right)-\varepsilon \alpha^2 \sum_{k=1}^{N-1}\partial_{\zeta_{k}\zeta_{k}}h(\alpha x),
	\end{aligned}
\end{equation}
where $U_{e(x)}$ and all of its derivatives take value at $(\xi(t, x, y), x, y)$. 
Define $\varepsilon_2^{+}:=\gamma_{\star} / 3$, where $\gamma_{\star}$ is defined in 
\eqref{1.5}. Then for any $\varepsilon<\varepsilon_2^{+}$, by \eqref{1.5} and 
Theorem \ref{Theorem 2.8}, there exists a sufficiently large positive constant  
$X^{\prime \prime}>1$ such that
\begin{equation}\label{3.27}
	f\left(x, y, U_{e(x)}\right)-f\left(x, y, U^{+}\right)>\frac{\kappa_1}{2} \varepsilon h(\alpha x),
\end{equation}
for all $(\xi, x, y) \in\left(-\infty,-X^{\prime \prime}\right) \times \mathbb{R}^N$, where 
$\kappa_1$ is given in \eqref{1.4}. By \eqref{3.2}, \eqref{3.10}, \eqref{3.19}, \eqref{3.20}, 
\eqref{3.26}, \eqref{3.27}, Theorem \ref{Theorem 2.8}, Theorem \ref{Theorem 2.13} and Lemma \ref{Lemma 2.17}, we know that there exists a constant $\Lambda_2>0$ such that
\begin{equation}\label{3.28}
	\begin{aligned}
\mathcal{L} \bar{V}&=  I_1-\varepsilon \alpha^2   \sum_{k=1}^{N-1}\partial_{\zeta_{k}\zeta_{k}}h(\alpha x)-f\left(x, y, \bar{V}\right)\\
&>-\Lambda_2 \alpha h(\alpha x) -\varepsilon \alpha^2   \sum_{k=1}^{N-1}\partial_{\zeta_{k}\zeta_{k}}h(\alpha x)+\frac{\kappa_1}{2} \varepsilon h(\alpha x)> h(\alpha x)\left( -\Lambda_2 \alpha+\frac{\kappa_1}{2} \varepsilon -\varepsilon \alpha^2 C_5\right)  ,  
	\end{aligned}
\end{equation}
for all $(\xi, x, y) \in\left(-\infty,-X^{\prime \prime}\right) \times \mathbb{R}^N$. Obviously, there exists a constant $\alpha_3^{+}(\varepsilon)>0$ such that
\begin{equation}\label{3.29}
	-\Lambda_2 \alpha+\frac{\kappa_1}{2} \varepsilon -\varepsilon \alpha^2 C_5 \geq 0 \text { for all } \alpha \in\left(0, \alpha_3^{+}(\varepsilon)\right).
\end{equation}
Therefore, set $\varepsilon_0^{+} \leq \varepsilon_2^{+}$ and $\alpha_0^{+} \leq \alpha_3^{+}(\varepsilon)$, then one gets $\mathcal{L} \bar{V}>0$ in Case $2$, by \eqref{3.28} and \eqref{3.29}.

\textbf{Case 3:} $-X^{\prime \prime} \leq \xi(t, x, y) \leq X^{\prime}$.
\begin{align*}
	&\left(\partial_t-\Delta_{x, y}\right)\left(\varepsilon h(\alpha x)\left[U_{e_i}^\beta(\eta, x, y) \omega(\xi)+(1-\omega(\xi))\right]\right)\\
	&=\varepsilon h(\alpha x)\left[\beta U_{e_i}^{\beta-1}(\eta)\partial_{\eta}U_{e_i}\partial_{t}\eta\omega(\xi)+U_{e_i}^{\beta}\omega^{\prime}(\xi)\partial_{t}\xi-\omega^{\prime}(\xi)\partial_{t}\xi\right]\\
	&\quad-\varepsilon\alpha\left(\sum_{k=1}^{N-1}\partial_{\zeta_{k}\zeta_{k}}h(\alpha x)\right)\left[U_{e_i}^\beta(\eta, x, y) \omega(\xi)+(1-\omega(\xi))\right]\\
	&\quad-2\varepsilon\alpha\beta U_{e_i}^{\beta-1}(\eta)\partial_{\eta}U_{e_i}\omega(\xi)\left(\sum_{k=1}^{N-1}\partial_{\zeta_{k}}h(\alpha x)\partial_{x_k}\eta\right)\\
	&\quad-2\varepsilon\alpha\beta U_{e_i}^{\beta-1}(\eta)\omega(\xi)\left(\sum_{k=1}^{N-1}\partial_{\zeta_{k}}h(\alpha x)\partial_{x_k}U_{e_i}\right)-2\varepsilon\alpha U_{e_i}^{\beta}(\eta)\omega^{\prime}(\xi)\left(\sum_{k=1}^{N-1}\partial_{\zeta_{k}}h(\alpha x)\partial_{x_k}\xi\right)\\
	&\quad+2\varepsilon\alpha \omega^{\prime}(\xi)\left(\sum_{k=1}^{N-1}\partial_{\zeta_{k}}h(\alpha x)\partial_{x_k}\xi\right)-\varepsilon h(\alpha x)\beta(\beta-1)U_{e_i}^{\beta-2}(\eta)(\partial_{\eta}U_{e_i})^2\omega(\xi)\left(\sum_{k=1}^{N-1}(\partial_{x_k}\eta)^2+(\partial_{y}\eta)^2\right)\\
	&\quad-2\varepsilon h(\alpha x)\beta(\beta-1)U_{e_i}^{\beta-2}(\eta)\partial_{\eta}U_{e_i}\omega(\xi)\left(\sum_{k=1}^{N-1}\partial_{x_k}U_{e_i}\partial_{x_k}\eta+\partial_{y}U_{e_i}\partial_{y}\eta\right)\\
	&\quad-\varepsilon h(\alpha x)\beta U_{e_i}^{\beta-1}(\eta)\partial_{\eta\eta}U_{e_i}\omega(\xi)\left(\sum_{k=1}^{N-1}(\partial_{x_k}\eta)^2+(\partial_{y}\eta)^2\right)\\
	&\quad-2\varepsilon h(\alpha x)\beta U_{e_i}^{\beta-1}(\eta)\omega(\xi)\left(\sum_{k=1}^{N-1}\partial_{x_k\eta}U_{e_i}\partial_{x_k}\eta+\partial_{y\eta}U_{e_i}\partial_{y}\eta\right)\\
	&\quad-\varepsilon h(\alpha x)\beta U_{e_i}^{\beta-1}(\eta)\partial_{\eta}U_{e_i}\omega(\xi)\left(\sum_{k=1}^{N-1}\partial_{x_k x_k}\eta+\partial_{yy}\eta\right)\\
	&\quad-2\varepsilon h(\alpha x)\beta 
	U_{e_i}^{\beta-1}(\eta)\partial_{\eta}U_{e_i}\omega^{\prime}(\xi)\left(\sum_{k=1}^{N-1}
\partial_{x_k}\eta\partial_{x_k}\xi+\partial_{y}\eta\partial_{y}\xi\right)\\
	&\quad-\varepsilon h(\alpha x)\beta(\beta-1) 
	U_{e_i}^{\beta-2}(\eta)\omega(\xi)\left(\sum_{k=1}^{N-1}(\partial_{x_k}U_{e_i})^2
+(\partial_{y}U_{e_i})^2\right)\\
	&\quad-2\varepsilon h(\alpha x)\beta 
	U_{e_i}^{\beta-1}(\eta)\omega^{\prime}(\xi)\left(\sum_{k=1}^{N-1}\partial_{x_k}U_{e_i}
\partial_{x_k}\xi+\partial_{y}U_{e_i}\partial_{y}\xi\right)\\
&\quad-\varepsilon h(\alpha x)\beta 
U_{e_i}^{\beta-1}(\eta)\omega(\xi)\left(\sum_{k=1}^{N-1}\partial_{x_kx_k}U_{e_i}
+\partial_{yy}U_{e_i}\right)\\
&\quad-\varepsilon h(\alpha x) 
U_{e_i}^{\beta}(\eta)\omega^{\prime\prime}(\xi)\left(\sum_{k=1}^{N-1}
(\partial_{x_k}\xi)^2+(\partial_{y}\xi)^2\right)-\varepsilon h(\alpha x) 
U_{e_i}^{\beta}(\eta)\omega^{\prime}(\xi)\left(\sum_{k=1}^{N-1}
\partial_{x_kx_k}\xi+\partial_{yy}\xi\right) \qquad\qquad\qquad\\
	&\quad+\varepsilon h(\alpha x) 
	\omega^{\prime\prime}(\xi)\left(\sum_{k=1}^{N-1}(\partial_{x_k}\xi)^2
+(\partial_{y}\xi)^2\right)+\varepsilon h(\alpha x) 
\omega^{\prime}(\xi)\left(\sum_{k=1}^{N-1}\partial_{x_kx_k}\xi
+\partial_{yy}\xi\right)
\end{align*}
By virtue of $\xi$ and $\eta$ is bounded, $h(\alpha x)\le1$, \eqref{3.2}-\eqref{3.4}, 
\eqref{3.10}-\eqref{3.12}, \eqref{3.14}, Theorem \ref{Theorem 2.8}, 
Theorem \ref{Theorem 2.12}, Theorem \ref{Theorem 2.13} and Lemma \ref{Lemma 2.17}, we have
\begin{align*}
	&\varepsilon h(\alpha x)\beta U_{e_i}^{\beta-1}(\eta)\partial_{\eta}U_{e_i}\partial_{t}\eta\omega(\xi)=\varepsilon h(\alpha x)\beta U_{e_i}^{\beta-1}(\eta)(-\hat{c})\partial_{\eta}U_{e_i}\omega(\xi)>0,\\ \\
	&\varepsilon h(\alpha x)U_{e_i}^{\beta}\omega^{\prime}(\xi)\partial_{t}\xi=\varepsilon h(\alpha x)U_{e_i}^{\beta}\omega^{\prime}(\xi)\left(\frac{-\hat{c}}{\sqrt{1+|\nabla \varphi(\alpha x)|^2}}\right)\ge-\varepsilon\hat{c} h(\alpha x)U_{e_i}^{\beta}\omega^{\prime}(\xi),\\ \\
	&-\varepsilon h(\alpha x)\omega^{\prime}(\xi)\partial_{t}\xi=-\varepsilon h(\alpha x)\omega^{\prime}(\xi)\left(\frac{-\hat{c}}{\sqrt{1+|\nabla \varphi(\alpha x)|^2}}\right)\ge 0,\\ \\
	&-\varepsilon\alpha\left(\sum_{k=1}^{N-1}\partial_{\zeta_{k}\zeta_{k}}h(\alpha x)\right)\left[U_{e_i}^\beta(\eta, x, y) \omega(\xi)+(1-\omega(\xi))\right]\ge-\varepsilon\alpha\left|\Delta h(\alpha x)\right|\ge-\varepsilon\alpha C_5 h(\alpha x),\\ \\
	&-2\varepsilon\alpha\beta U_{e_i}^{\beta-1}(\eta)\partial_{\eta}U_{e_i}\omega(\xi)\left(\sum_{k=1}^{N-1}\partial_{\zeta_{k}}h(\alpha x)\partial_{x_k}\eta\right)\\ 
	&\quad\ge-2\varepsilon\alpha\beta\omega(\xi)|\partial_{\eta}U_{e_i}|\cdot|\nabla h(\alpha x)|\cdot|-\nabla\varphi(\alpha x)|\ge-2\varepsilon\alpha\beta r C_5 h(\alpha x) (C_3 h(\alpha x)+\max_{1\le i\le n} \{|\nu_i \cot \theta_i|\}),\\ \\
	&-2\varepsilon\alpha\beta U_{e_i}^{\beta-1}(\eta)\omega(\xi)\left(\sum_{k=1}^{N-1}\partial_{\zeta_{k}}h(\alpha x)\partial_{x_k}U_{e_i}\right)\\
	&\quad\ge-2\varepsilon\alpha\beta U_{e_i}^{\beta-1}(\eta)\omega(\xi)\left|\nabla h(\alpha x)\cdot \nabla_x U_{e_i}\right|\ge-2\varepsilon\alpha\beta C_5 h(\alpha x) \sqrt{(N-1)\left(\bar{K} \max\{e^{-\frac{3 \kappa}{4} |\eta|}, e^{-\kappa_2 |\eta|}\}\right)^2},\\  \\
   &-2\varepsilon\alpha U_{e_i}^{\beta}(\eta)\omega^{\prime}(\xi)\left(\sum_{k=1}^{N-1}\partial_{\zeta_{k}}h(\alpha x)\partial_{x_k}\xi\right)\\
   &\quad\ge-2\varepsilon\alpha\omega^{\prime}(\xi)\left|\nabla h(\alpha x)\cdot \nabla_x \xi\right|\ge -\varepsilon\alpha\omega^{\prime}(\xi) C_5 h(\alpha x) (\alpha C_4(1+|\xi|) h(\alpha x)+\alpha C_4|\xi| h(\alpha x)),\\ \\
	&2\varepsilon\alpha \omega^{\prime}(\xi)\left(\sum_{k=1}^{N-1}\partial_{\zeta_{k}}h(\alpha x)\partial_{x_k}\xi\right)\\
	&\quad\ge-2\varepsilon\alpha\omega^{\prime}(\xi)\left|\nabla h(\alpha x)\cdot \nabla_x \xi\right|\ge -\varepsilon\alpha\omega^{\prime}(\xi) C_5 h(\alpha x) (\alpha C_4(1+|\xi|) h(\alpha x)+\alpha C_4|\xi| h(\alpha x)),\\ \\
		&-\varepsilon h(\alpha  
		x)\beta(\beta-1)U_{e_i}^{\beta-2}(\eta)(\partial_{\eta}U_{e_i})^2\omega(\xi)
\left(\sum_{k=1}^{N-1}(\partial_{x_k}\eta)^2+(\partial_{y}\eta)^2\right)\\
&\quad\ge-\varepsilon h(\alpha x)|\beta(\beta-1)|(1+|\nabla \varphi|^2)\ge-\varepsilon h(\alpha x)|\beta(\beta-1)|\left[1+(C_3 h(\alpha x)+\max_{1\le i\le n} \{|\nu_i \cot \theta_i|\})^2\right],\\ \\
		&-2\varepsilon h(\alpha 
		x)\beta(\beta-1)U_{e_i}^{\beta-2}(\eta)\partial_{\eta}U_{e_i}\omega(\xi)\left(
\sum_{k=1}^{N-1}\partial_{x_k}U_{e_i}\partial_{x_k}\eta+\partial_{y}U_{e_i}\partial_{y}\eta\right)\\
		&\quad\geq-2\varepsilon h(\alpha x)|\beta(\beta-1)|U_{e_i}^{\beta-2}(\eta)\omega(\xi)|\partial_{\eta}U_{e_i}|\left|\nabla_{x,y} U_{e_i} \cdot \nabla_{x,y} \eta\right|\\
		&\quad\ge-2\varepsilon r h(\alpha x)|\beta(\beta-1)|\left|\nabla_{x,y} U_{e_i} \right|\cdot \left| (-\nabla \varphi,1)\right|\\
		&\quad\ge-2\varepsilon r h(\alpha x)|\beta(\beta-1)|\left(N\left(\bar{K} \max\{e^{-\frac{3 \kappa}{4} |\eta|}, e^{-\kappa_2 |\eta|}\}\right)^2\right)^\frac{1}{2}\left[1+(C_3 h(\alpha x)+\max_{1\le i\le n} \{|\nu_i \cot \theta_i|\})^2\right]^\frac{1}{2},\\ \\
	&-\varepsilon h(\alpha x)\beta U_{e_i}^{\beta-1}(\eta)\partial_{\eta\eta}U_{e_i}\omega(\xi)\left(\sum_{k=1}^{N-1}(\partial_{x_k}\eta)^2+(\partial_{y}\eta)^2\right)\\
	&\quad\ge-\varepsilon h(\alpha x)\beta |\partial_{\eta\eta}U_{e_i}|(1+|\nabla \varphi|^2)\\
	&\quad\ge-\varepsilon h(\alpha x)\beta\bar{K} \max\{e^{-\frac{3 \kappa}{4} |\eta|}, e^{-\kappa_2 |\eta|}\}\left[1+(C_3 h(\alpha x)+\max_{1\le i\le n} \{|\nu_i \cot \theta_i|\})^2\right],\\ \\
	&-2\varepsilon h(\alpha x)\beta U_{e_i}^{\beta-1}(\eta)\omega(\xi)\left(\sum_{k=1}^{N-1}\partial_{x_k\eta}U_{e_i}\partial_{x_k}\eta+\partial_{y\eta}U_{e_i}\partial_{y}\eta\right)\\
	&\quad\ge-2\varepsilon h(\alpha x)\beta\left|\nabla_{x,y}\partial_{\eta} U_{e_i} \cdot \nabla_{x,y} \eta\right|\\
	&\quad\ge-2\varepsilon h(\alpha x)\beta\left(N\left(\bar{K} \max\{e^{-\frac{3 \kappa}{4} |\eta|}, e^{-\kappa_2 |\eta|}\}\right)^2\right)^\frac{1}{2}\left[1+(C_3 h(\alpha x)+\max_{1\le i\le n} \{|\nu_i \cot \theta_i|\})^2\right]^\frac{1}{2},\\ \\
    &-\varepsilon h(\alpha x)\beta U_{e_i}^{\beta-1}(\eta)\partial_{\eta}U_{e_i}\omega(\xi)\left(\sum_{k=1}^{N-1}\partial_{x_k x_k}\eta+\partial_{yy}\eta\right)\\
    &\quad\ge-\varepsilon h(\alpha x)\beta r \left|- \sum_{k=1}^{N-1} \alpha \partial_{\zeta_k \zeta_k} \varphi(\alpha x)\right|\ge-\varepsilon h(\alpha x)\beta r \cdot \alpha (N-1)C_3 h(\alpha x)\ge-\varepsilon h(\alpha x)\beta r \alpha (N-1)C_3 ,\\ \\
	&-2\varepsilon h(\alpha x)\beta U_{e_i}^{\beta-1}(\eta)\partial_{\eta}U_{e_i}\omega^{\prime}(\xi)\left(\sum_{k=1}^{N-1}\partial_{x_k}\eta\partial_{x_k}\xi+\partial_{y}\eta\partial_{y}\xi\right)\\
	&\quad\ge -2\varepsilon h(\alpha x)\beta r \omega^{\prime}(\xi)\left| (-\nabla \varphi,1)\right|\cdot \left|\nabla_{x,y} \xi\right|\\
	&\quad\ge-\varepsilon h(\alpha x)\beta r \omega^{\prime}(\xi)\left[1+(C_3 h(\alpha x)+\max_{1\le i\le n} \{|\nu_i \cot \theta_i|\})^2\right]^\frac{1}{2}(\alpha C_4(1+|\xi|) h(\alpha x)+\alpha C_4|\xi| h(\alpha x)),\\ \\
&-\varepsilon h(\alpha x)\beta(\beta-1) U_{e_i}^{\beta-2}(\eta)\omega(\xi)\left(\sum_{k=1}^{N-1}(\partial_{x_k}U_{e_i})^2+(\partial_{y}U_{e_i})^2\right)\\
&\quad\ge-\varepsilon h(\alpha x)|\beta(\beta-1)|\left|\nabla_{x,y}U_{e_i}\right|^2\ge-\varepsilon h(\alpha x)|\beta(\beta-1)|N\left(\bar{K} \max\{e^{-\frac{3 \kappa}{4} |\eta|}, e^{-\kappa_2 |\eta|}\}\right)^2,\\ \\
	&-2\varepsilon h(\alpha x)\beta U_{e_i}^{\beta-1}(\eta)\omega^{\prime}(\xi)\left(\sum_{k=1}^{N-1}\partial_{x_k}U_{e_i}\partial_{x_k}\xi+\partial_{y}U_{e_i}\partial_{y}\xi\right)\\
	&\quad\ge-2\varepsilon h(\alpha x)\beta \omega^{\prime}(\xi)\left|\nabla_{x,y}U_{e_i} \cdot \nabla_{x,y} \xi\right|\\
	&\quad\ge-\varepsilon h(\alpha x)\beta \omega^{\prime}(\xi)\left(N\left(\bar{K} \max\{e^{-\frac{3 \kappa}{4} |\eta|}, e^{-\kappa_2 |\eta|}\}\right)^2\right)^\frac{1}{2}(\alpha C_4(1+|\xi|) h(\alpha x)+\alpha C_4|\xi| h(\alpha x)),\\ \\
		&-\varepsilon h(\alpha x)\beta U_{e_i}^{\beta-1}(\eta)\omega(\xi)\left(\sum_{k=1}^{N-1}\partial_{x_kx_k}U_{e_i}+\partial_{yy}U_{e_i}\right)\ge-\varepsilon h(\alpha x)\beta N\left(\bar{K} \max\{e^{-\frac{3 \kappa}{4} |\eta|}, e^{-\kappa_2 |\eta|}\}\right),\\ \\ 
		&-\varepsilon h(\alpha x) U_{e_i}^{\beta}(\eta)\omega^{\prime\prime}(\xi)\left(\sum_{k=1}^{N-1}(\partial_{x_k}\xi)^2+(\partial_{y}\xi)^2\right)\\
		&\quad\ge-\varepsilon h(\alpha x)|\omega^{\prime\prime}(\xi)|\cdot|\nabla_{x,y} \xi|^2\ge-\varepsilon h(\alpha x)|\omega^{\prime\prime}(\xi)|\left[\frac{1}{2}(\alpha C_4(1+|\xi|) h(\alpha x)+\alpha C_4|\xi| h(\alpha x))\right]^2,\\ \\
		&-\varepsilon h(\alpha x) U_{e_i}^{\beta}(\eta)\omega^{\prime}(\xi)\left(\sum_{k=1}^{N-1}\partial_{x_kx_k}\xi+\partial_{yy}\xi\right)\\
		&\quad\ge-\varepsilon h(\alpha x) U_{e_i}^{\beta}(\eta)|\omega^{\prime}(\xi)|\cdot|\Delta_{x,y}\xi|\\
		&\quad\ge-\varepsilon h(\alpha x)|\omega^{\prime}(\xi)|\alpha N C_4(1+\alpha|\xi|) h(\alpha x)\ge-\varepsilon h(\alpha x)|\omega^{\prime}(\xi)|\alpha N C_4(1+\alpha|\xi|) ,\\ \\
		&\varepsilon h(\alpha x) \omega^{\prime\prime}(\xi)\left(\sum_{k=1}^{N-1}(\partial_{x_k}\xi)^2+(\partial_{y}\xi)^2\right)\ge-\varepsilon h(\alpha x) |\omega^{\prime\prime}(\xi)|\left[\frac{1}{2}(\alpha C_4(1+|\xi|) h(\alpha x)+\alpha C_4|\xi| h(\alpha x))\right]^2,\\ 
		&\text{and}\\
		&\varepsilon h(\alpha x) \omega^{\prime}(\xi)\left(\sum_{k=1}^{N-1}\partial_{x_kx_k}\xi+\partial_{yy}\xi\right)\ge-\varepsilon h(\alpha x) \omega^{\prime}(\xi)|\Delta_{x,y}\xi|\ge-\varepsilon h(\alpha x) \omega^{\prime}(\xi)\alpha N C_4(1+\alpha|\xi|).
\end{align*}
Thus, there is a constant $\Lambda_3>0$ such that
\begin{align}\label{3.30}
	\left|\left(\partial_t-\Delta_{x, y}\right)\left(\varepsilon h(\alpha x)\left[U_{e_i}^\beta(\eta, x, y) \omega(\xi)+(1-\omega(\xi))\right]\right)\right|<\Lambda_3 \varepsilon h(\alpha x),
\end{align}
for all $(\xi, x, y) \in\left[-X^{\prime \prime}, X^{\prime}\right] \times \mathbb{R}^N$. It follows from \eqref{3.11} and \eqref{3.30} that
\begin{equation}\label{3.31}
\begin{aligned}
	\mathcal{L} \bar{V}> & I_1-\Lambda_3 \varepsilon h(\alpha x)-f\left(x, y, \bar{V}\right) \\
	= & (\partial_t \xi+c_{e(x)})\partial_{\xi} U_{e(x)}  -2\sum_{k=1}^{N-1} \partial_{x_{k} \xi} U_{e(x)} \partial_{x_{k}} \xi -2 \partial_{y \xi} U_{e(x)} \partial_{y} \xi\\
	& + 2\nabla_{x, y} \partial_{\xi} U_{e(x)} \cdot e(x)-\partial_\xi U_{e(x)}\left(\sum_{k=1}^{N-1}\partial_{x_{k}x_{k}} \xi+\partial_{yy} \xi\right)\\
	& -\partial_{\xi \xi} U_{e(x)}\left(\sum_{k=1}^{N-1}(\partial_{x_{k}} \xi)^2+(\partial_{y} \xi)^2-1\right)-\sum_{k=1}^{N-1}U_{e(x)}^{\prime \prime} \cdot \partial_{x_{k}}e(x) \cdot \partial_{x_{k}}e(x) \\
	& -\sum_{k=1}^{N-1}U_{e(x)}^{\prime} \cdot \partial_{x_{k}x_{k}}e(x)-2 \sum_{k=1}^{N-1}\partial_{x_k} U_{e(x)}^{\prime} \cdot \partial_{x_{k}}e(x)-2\sum_{k=1}^{N-1} \partial_\xi U_{e(x)}^{\prime} \cdot \partial_{x_{k}}e(x) \partial_{x_{k}} \xi\\
	&+f\left(x, y, U_{e(x)}\right)-f\left(x, y, \bar{V}\right)-\Lambda_3 \varepsilon h(\alpha x)\\
	=&:(\partial_t \xi+c_{e(x)})\partial_{\xi} U_{e(x)}+ X^* +f\left(x, y, U_{e(x)}\right)-f\left(x, y, \bar{V}\right)-\Lambda_3 \varepsilon h(\alpha x)
\end{aligned}
\end{equation}
for all $(\xi, x, y) \in\left[-X^{\prime \prime}, X^{\prime}\right] \times \mathbb{R}^N$, where $U_{e(x)}$ and all of its derivatives take value at $(\xi(t, x, y), x, y)$. It follows from Theorem \ref{Theorem 2.12} and \eqref{3.12} that there exists a number $r>0$ such that
\begin{align}\label{3.32}
	\left(\xi_t+c_{e(x)}\right) \partial_{\xi} U_{e(x)} \geq C_{0} r  h(\alpha x)>0 ,
\end{align}
for all $(\xi, x, y) \in\left[-X^{\prime \prime}, X^{\prime}\right] \times \mathbb{R}^N$. By the 
boundedness of $\xi$, \eqref{3.2}, \eqref{3.10}, Theorem \ref{Theorem 2.8}, Theorem \ref{Theorem 2.13} and Lemma \ref{Lemma 2.17}, one has 
	\begin{align*}
		&-2\sum_{k=1}^{N-1} \partial_{x_{k} \xi} U_{e(x)} \partial_{x_{k}} \xi-2 \partial_{y \xi} U_{e(x)} \partial_{y} \xi + 2\nabla_{x, y} \partial_{\xi} U_{e(x)} \cdot e(x)\\
		&\quad=-2\nabla_{x, y} \partial_{\xi} U_{e(x)}(\nabla_{x, y}\xi-e(x))\ge-2 \alpha C_4| \xi | h(\alpha x) \left(N\left(\bar{K} \max\{e^{-\frac{3 \kappa}{4} |\xi|}, e^{-\kappa_2 |\xi|}\}\right)^2\right)^\frac{1}{2},\\ \\
		&-\partial_\xi U_{e(x)}\left(\sum_{k=1}^{N-1}\partial_{x_{k}x_{k}} \xi+\partial_{yy} \xi\right)\ge-|\partial_\xi U_{e(x)}|\cdot|\Delta_{x,y}\xi|\ge -r \alpha C_4(1+\alpha|\xi|) h(\alpha x),\\ \\
		&-\partial_{\xi \xi} U_{e(x)}\left(\sum_{k=1}^{N-1}(\partial_{x_{k}} \xi)^2+(\partial_{y} \xi)^2-1\right)\\
		&\quad\ge-|\partial_{\xi \xi} U_{e(x)}|\cdot|\left(\nabla_{x, y} \xi-e(x)\right) \left(\nabla_{x, y} \xi+e(x)\right)|\\
		&\quad\ge-\bar{K} \max\{e^{-\frac{3 \kappa}{4} |\xi|}, e^{-\kappa_2 |\xi|}\}\alpha^2 C_4^2(1+|\xi|)| \xi |  h^2(\alpha x)\\
		&\quad\ge-\bar{K} \max\{e^{-\frac{3 \kappa}{4} |\xi|}, e^{-\kappa_2 |\xi|}\}\alpha^2 C_4^2(1+|\xi|)| \xi |  h(\alpha x),\\ \\
		&-\sum_{k=1}^{N-1}U_{e(x)}^{\prime \prime} \cdot \partial_{x_{k}}e(x) \cdot \partial_{x_{k}}e(x)\\
		&\quad\ge -(N-1)M_1 \max\{e^{-\frac{\kappa}{2} |\xi|}, e^{-\frac{\kappa_2}{2} |\xi|}\}\alpha^2 M^2_2 h^2(\alpha x)\ge -(N-1)M_1 \max\{e^{-\frac{\kappa}{2} |\xi|}, e^{-\frac{\kappa_2}{2} |\xi|}\}\alpha^2 M^2_2 h(\alpha x),\\ \\
		&-\sum_{k=1}^{N-1}U_{e(x)}^{\prime} \cdot \partial_{x_{k}x_{k}}e(x)\ge-(N-1)M_1 \max\{e^{-\frac{\kappa}{2} |\xi|}, e^{-\frac{\kappa_2}{2} |\xi|}\}\alpha M_3  h(\alpha x),\\ \\
        &-2 \sum_{k=1}^{N-1}\partial_{x_k} U_{e(x)}^{\prime} \cdot \partial_{x_{k}}e(x)\ge-2(N-1)M_1 \max\{e^{-\frac{\kappa}{2} |\xi|}, e^{-\frac{\kappa_2}{2} |\xi|}\} \alpha M_2 h(\alpha x),\\ 
        &\text{and}\\
        &-2\sum_{k=1}^{N-1} \partial_\xi U_{e(x)}^{\prime} \cdot \partial_{x_{k}}e(x) \partial_{x_{k}} \xi\\
        &\quad\ge -(N-1)M_1 \max\{e^{-\frac{\kappa}{2} |\xi|}, e^{-\frac{\kappa_2}{2} |\xi|}\}(\alpha C_4(1+|\xi|) h(\alpha x)+\alpha C_4|\xi| h(\alpha x))\alpha M_2 h(\alpha x).
	\end{align*}
So there exists a constant $\Lambda_4>0$ such that 
\begin{equation}\label{3.33}
	X^*\geq-\Lambda_4 \alpha h(\alpha x),
\end{equation}
for all $(\xi, x, y) \in\left[-X^{\prime \prime}, X^{\prime}\right] \times \mathbb{R}^N$. By \eqref{1.2} there is a constant $\Lambda_5>0$ such that
\begin{equation}\label{3.34}
f\left(x, y, U_{e(x)}\right)-f\left(x, y, \bar{V}\right)>-\Lambda_5 \varepsilon h(\alpha x).
\end{equation}
It follows from \eqref{3.31}-\eqref{3.34} that
\begin{equation}\label{3.35}
\mathcal{L} \bar{V}>h(\alpha x) \times\left[C_{0} r-\Lambda_4 \alpha-\Lambda_5 \varepsilon-\Lambda_3 \varepsilon\right], \quad \forall(\xi, x, y) \in\left[-X^{\prime \prime}, X^{\prime}\right] \times \mathbb{R}^N.
\end{equation}
Denote
\begin{equation}\label{3.36}
\alpha_4^{+}(\varepsilon):=\varepsilon \;\text { and }\; \varepsilon_3^{+}(\beta):=\frac{C_{0} r}{\Lambda_3+\Lambda_4+\Lambda_5}.
\end{equation}
By \eqref{3.35} and \eqref{3.36}, set $\varepsilon_0^{+} \leq \varepsilon_3^{+}(\beta)$ and $\alpha_0^{+} \leq \alpha_4^{+}(\varepsilon)$, then $\mathcal{L} \bar{V}>0$ holds in Case $3$. Thus, by setting $\varepsilon_0^{+} \leq \min \left\{\varepsilon_1^{+}, \varepsilon_2^{+}, \varepsilon_3^{+}(\beta) \right\}$, $\alpha_0^{+} \leq \min \left\{\alpha_1^{+}(\beta), \alpha_2^{+}(\beta, \varepsilon), \alpha_3^{+}(\varepsilon), \alpha_4^{+}(\varepsilon)\right\}$ and $\beta^*=\min \left\{1, \beta_1^*, \beta_2^*\right\}$, Step $1$ is complete. Besides, we can get immediately \eqref{3.8} by calculation. 

{\it Step 2: proof of \eqref{3.6}.} In fact, we only need to show that
\begin{equation*}
 \left|U_{e(x)}(\xi, x, y)-\underline{V}(t, x, y)\right| \leq \varepsilon, \text{ as } d\left((x,y), \mathcal{R}+\hat{c} t e_0\right) \rightarrow+\infty.
\end{equation*}

\textbf{Case 1:} For $(t, x, y) \in \mathbb{R} \times \mathbb{R}^N$ such that $d\left((x, y), \partial \mathcal{Q}+\hat{c} t e_0\right)<+\infty$ and $d\left((x, y), \mathcal{R}+\hat{c} t e_0\right) \geq \rho$ as $\rho \rightarrow+\infty$. 
Without loss of generality, we assume that $x \in \widehat{Q}_i$ $(i \in\{1, \ldots, n\})$. Then, one has $x \cdot \nu_i \cos \theta_i+(y-\hat{c}t) \sin \theta_i$ is bounded. However, $\psi( x)=\psi_i( x)=\max _{1 \leq j \leq n} \psi_j( x)$ in $\widehat{Q}_i$, then we have $-x \cdot \nu_i \cot \theta_i\ge-x \cdot \nu_j \cot \theta_j$ for all $j \neq i$ in $\widehat{Q}_i$. By $\theta_i\in(0, \pi / 2]$ for each $i$, one knows $x \cdot \nu_i \cos \theta_i+(y-\hat{c} t) \sin \theta_i \le x \cdot \nu_j \cos \theta_j+(y-\hat{c} t) \sin \theta_j$ for all $j \neq i$ in $\widehat{Q}_i$. Because $\hat{c}=c_{e_i} / e_i \cdot e_0=c_{e_i} / \sin \theta_i$ for each $i$, $U_e(+\infty, x, y)=0$ for any $e \in \mathbb{S}^{N-1}$ and $(x, y) \in \mathbb{L}^N$, it holds that
\begin{equation*}
	\underline{V}(t, x, y)=\max _{1 \leq k \leq n}\left\{U_{e_k}\left(x \cdot \nu_k \cos \theta_k+(y-\hat{c} t) \sin \theta_k, x, y\right)\right\}=U_{e_i}\left(x \cdot \nu_i \cos \theta_i+(y-\hat{c} t) \sin \theta_i, x, y\right).
\end{equation*}
Since $x \in \widehat{Q}_i$, $d\left((x, y), \partial \mathcal{Q}+\hat{c} t e_0\right)<+\infty$ and $d\left((x, y), \mathcal{R}+\hat{c} t e_0\right) \geq \rho$ as $\rho \rightarrow+\infty$, it holds that $d(x, \widehat{\mathcal{R}}) \rightarrow+\infty$. By \eqref{2.11} and $\hat{q}_i(\alpha x) \rightarrow 0$, it holds that $\xi(t, x, y) \rightarrow x \cdot \nu_i \cos \theta_i+(y- \hat{c}t)\sin \theta_i$. Thus,
\begin{equation}\label{3.37}
\left|U_{e_i}(\xi(t, x, y), x, y)-U_{e_i}\left(x \cdot \nu_i \cos \theta_i+(y-\hat{c} t) \sin \theta_i, x, y\right)\right| \leq \frac{\varepsilon}{2}.
\end{equation}
Using \eqref{2.11} again, we have $e(x) \rightarrow\left(\nu_i \cos \theta_i, \sin \theta_i\right)=e_i$. Then, by Theorem \ref{Theorem 2.13} and the boundedness of $U_{e_i}$, one gets
\begin{equation}\label{3.38}
\left|U_{e(x)}(\xi(t, x, y), x, y)-U_{e_i}(\xi(t, x, y), x, y)\right| \leq\left\|U_{e_i}^{\prime}\right\| \cdot\left|e(x)-e_i\right|+o\left(\left|e(x)-e_i\right|\right) \leq \frac{\varepsilon}{2}.
\end{equation}
By virtue of the definition of $\bar{V}$, \eqref{3.37} and \eqref{3.38}, we obtain
\begin{equation*}
\begin{aligned}
	&|U_{e(x)}(\xi, x, y)-\underline{V}(t, x, y)|\\
	=&  \left|U_{e(x)}(\xi, x, y)-U_{e_i}\left(x \cdot \nu_i \cos \theta_i+(y-\hat{c} t) \sin \theta_i, x, y\right)\right| \\
	\leq & \left|U_{e(x)}(\xi, x, y)-U_{e_i}(\xi(t, x, y), x, y)\right|+\mid U_{e_i}(\xi(t, x, y), x, y) -U_{e_i}\left(x \cdot \nu_i \cos \theta_i+(y-\hat{c} t) \sin \theta_i, x, y\right)|\\
	\leq&\frac{\varepsilon}{2}+\frac{\varepsilon}{2}\leq \varepsilon.
\end{aligned}
\end{equation*}

\textbf{Case 2:} For $(t, x, y) \in \mathbb{R} \times \mathbb{R}^N$ such that $d\left(( x, y), \partial \mathcal{Q}+\hat{c} t e_0\right) \geq \rho$ as $\rho \rightarrow+\infty$, we get that
\begin{equation*}
\min _{1 \leq i \leq n}\left\{x \cdot \nu_i \cos \theta_i+(y-\hat{c} t) \sin \theta_i\right\} \rightarrow+\infty \text { or }-\infty.
\end{equation*}

If $\min _{1 \leq i \leq n}\left\{x \cdot \nu_i \cos \theta_i+(y-\hat{c} t) \sin \theta_i\right\} \rightarrow+\infty$, then $0<\underline{V}(t, x, y) \leq \varepsilon / 2$ since $U_e(+\infty, x, y)=0$ for any $e \in \mathbb{S}^{N-1}$ and $(x, y) \in \mathbb{L}^N$. Firstly we claim that the surface $y=\varphi(x)$ is bounded away from $\partial \mathcal{Q}$. If $y=\varphi(x)$ is not bounded away from $\partial \mathcal{Q}$, there exists a sequence $\left\{x_k\right\}_{k \in \mathbb{N}}$ such that $\min _{1 \leq i \leq n}\left\{\hat{q}_i\left(x_k\right)\right\} \rightarrow$ $+\infty$ or $-\infty$ as $k \rightarrow+\infty$, which contradicts the function $\sum_{i=1}^n e^{-\hat{q}_i\left(x_k\right)}=1$. We suppose that $\hat{q}_l\left(x\right)=\min _{1 \leq i \leq n}\left\{\hat{q}_i\left(x\right)\right\}$. Because the surface $y=\varphi(x)$ is bounded away from $\partial \mathcal{Q}$ and $\min _{1 \leq i \leq n}\left\{\hat{q}_i\left(x\right)\right\}=\min _{1 \leq i \leq n}\left\{x \cdot \nu_i \cos \theta_i+\varphi(x) \sin \theta_i\right\}$ is bounded, one has $x \cdot \nu_l \cos \theta_l+\varphi(x) \sin \theta_l$ is bounded and $x \cdot \nu_l \cos \theta_l$ is bounded. Notice that $x \cdot \nu_l \cos \theta_l+(y-\hat{c} t) \sin \theta_l\rightarrow+\infty$. By $\theta_i\in(0, \pi/2]$, we have $y-\hat{c} t\rightarrow+\infty$ and $\alpha(y-\hat{c} t)-\varphi(\alpha x) \rightarrow+\infty$, and then $\xi(t, x, y) \rightarrow+\infty$. Thus, $0<U_{e(x)}(\xi(t, x, y), x, y) \leq \varepsilon / 2$ and it holds that
\begin{equation*}
\left|U_{e(x)}(\xi(t, x, y), x, y)-\underline{V}(t, x, y)\right|  \leq  \varepsilon .
\end{equation*}

If $\min _{1 \leq i \leq n}\left\{x \cdot \nu_i \cos \theta_i+(y-\hat{c} t) \sin \theta_i\right\} \rightarrow-\infty$, then $1-\varepsilon\leq\underline{V}(t, x, y) < 1$ since $U_e(-\infty, x, y)=1$ for any $e \in \mathbb{S}^{N-1}$ and $(x, y) \in \mathbb{L}^N$.  Because the surface $y=\varphi(x)$ is bounded away from $\partial \mathcal{Q}$ and $\min _{1 \leq i \leq n}\left\{\hat{q}_i\left(x\right)\right\}=\min _{1 \leq i \leq n}\left\{x \cdot \nu_i \cos \theta_i+\varphi(x) \sin \theta_i\right\}$ is bounded, we can obtain that each $x \cdot \nu_i \cos \theta_i+\varphi(x) \sin \theta_i$ has a lower bound for $1 \leq i \leq n$. Then one has that each $x \cdot \nu_i \cos \theta_i$ has a lower bound for $1 \leq i \leq n$, since $y=\varphi(x)$ is bounded. Suppose that $\hat{q}_l\left(x\right)=\min _{1 \leq i \leq n}\left\{\hat{q}_i\left(x\right)\right\}$, so $x \cdot \nu_l \cos \theta_l+(y-\hat{c} t) \sin \theta_l\rightarrow-\infty$. By $\theta_i\in(0, \pi/2]$, we can also get $\alpha(y-\hat{c} t)-\varphi(\alpha x) \rightarrow-\infty$, and then $\xi(t, x, y) \rightarrow-\infty$. Therefore, $1-\varepsilon\leq U_{e(x)}(\xi(t, x, y), x, y) < 1$ and it holds that
\begin{equation*}
	\left|U_{e(x)}(\xi(t, x, y), x, y)-\underline{V}(t, x, y)\right| \leq \varepsilon.
\end{equation*}

{\it Step 3: proof of \eqref{3.7}.} We just need to prove that for all $i \in\{1, \ldots, n\}$, the inequality
\begin{equation*}
	\bar{V}(t, x, y) \geq U_{e_i}\left(x \cdot \nu_i \cos \theta_i+(y-\hat{c} t) \sin \theta_i, x, y\right)
\end{equation*}
holds for $(t, x, y) \in \mathbb{R} \times \mathbb{R}^N$. For convenience, denote $u_i(t, x, y)=U_{e_i}\left(x \cdot \nu_i \cos \theta_i+(y-\hat{c} t) \sin \theta_i, x, y\right)$ in Step $3$. Since $U_e(-\infty, x, y)=1$ and $U_e(+\infty, x, y)=0$ for any $e \in \mathbb{S}^{N-1}$ and $(x, y) \in \mathbb{L}^N$, we define sets
\begin{equation*}
\left\{\begin{array}{l}
	\Gamma_t=\left\{(t, x, y) \in \mathbb{R} \times \mathbb{R}^N ; x \cdot \nu_i \cos \theta_i+(y-\hat{c} t) \sin \theta_i=0\right\}, \\
	\Omega_t^{+}=\left\{(t, x, y) \in \mathbb{R} \times \mathbb{R}^N ; x \cdot \nu_i \cos \theta_i+(y-\hat{c} t) \sin \theta_i<0\right\},\\
	\Omega_t^{-}=\left\{(t, x, y) \in \mathbb{R} \times \mathbb{R}^N ; x \cdot \nu_i \cos \theta_i+(y-\hat{c} t) \sin \theta_i>0\right\},
\end{array}\right.	
\end{equation*}
and 
\begin{equation*}
\left\{\begin{array}{l}
	\widetilde{\Gamma}_t=\left\{(x, y) \in \mathbb{R} \times \mathbb{R}^N ; \min _{1 \leq i \leq n}\left\{x \cdot \nu_i \cos \theta_i+(y-\hat{c} t) \sin \theta_i\right\}=0\right\}, \\
	\widetilde{\Omega}_t^{+}=\left\{(x, y) \in \mathbb{R} \times \mathbb{R}^N ; \min _{1 \leq i \leq n}\left\{x \cdot \nu_i \cos \theta_i+(y-\hat{c} t) \sin \theta_i\right\}<0\right\}, \\
	\widetilde{\Omega}_t^{-}=\left\{(x, y) \in \mathbb{R} \times \mathbb{R}^N ; \min _{1 \leq i \leq n}\left\{x \cdot \nu_i \cos \theta_i+(y-\hat{c} t) \sin \theta_i\right\}>0\right\}.
\end{array}\right.
\end{equation*}
It is easy to obtain that $\widetilde{\Omega}_t^{-} \subset \Omega_t^{-}$. From Step $2$ and definitions of $u_i(t, x, y)$ and $\bar{V}(t, x, y)$, for $0<\gamma_{\star}\leq \min \{\theta / 2,1-\theta\}$, there exist two constants $m>0$ and $\tilde{m}>0$ such that
\begin{equation*}
\left\{\begin{array}{l}
	\forall t \in \mathbb{R}, \forall(x, y) \in \Omega_t^{+},\left(d\left((x, y), \Gamma_t\right) \geq m\right) \Rightarrow\left(u_i(t, x, y) \geq 1-\gamma_{\star} / 2\right), \\
	\forall t \in \mathbb{R}, \forall(x, y) \in \Omega_t^{-},\left(d\left((x, y), \Gamma_t\right) \geq m\right) \Rightarrow\left(u_i(t, x, y) \leq \gamma_{\star}\right),
\end{array}\right.
\end{equation*}
and
\begin{equation*}
\left\{\begin{array}{l}
	\forall t \in \mathbb{R}, \forall(x, y) \in \widetilde{\Omega}_t^{+},\left(d\left((x, y), \widetilde{\Gamma}_t\right) \geq \tilde{m}\right) \Rightarrow(\bar{V}(t, x, y) \geq 1-\gamma_{\star} / 2), \\
	\forall t \in \mathbb{R}, \forall(x, y) \in \widetilde{\Omega}_t^{-},\left(d\left((x, y), \widetilde{\Gamma}_t\right) \geq \tilde{m}\right) \Rightarrow(\bar{V}(t, x, y) \leq \gamma_{\star}) .
\end{array}\right.
\end{equation*}
Similar to the proof of Lemma \ref{Lemma 2.15}, there is $T_{min}$ defined as
\begin{equation*}
	T_{min }=\inf \left\{T \in \mathbb{R} ; \bar{V}(t+T, x, y) \geq u_i(t, x, y) \text { for }(t, x, y) \in \mathbb{R} \times \mathbb{R}^N\right\}
\end{equation*}
such that $\bar{V}\left(t+T_{min}, x, y\right) \geq u_i(t, x, y)$ for any $(t, x, y) \in \mathbb{R} \times \mathbb{R}^N$. By \eqref{3.6}, it holds that $\bar{V}(t, x, y) \rightarrow U_{e_i}\left(x\cdot\nu_i \cos \theta_i+(y-\hat{c} t) \sin \theta_i, x, y\right)$ for $(t, x, y) \in \mathbb{R} \times \mathbb{R}^N$ such that $d\left((x, y), \widetilde{Q}_i+\hat{c} t e_0\right)<+\infty$ and $d\left((x, y), \mathcal{R}+\hat{c} t e_0\right) \rightarrow+\infty$ and this implies that $0 \leq T_{min }<+\infty$, then we just have to prove $T_{min }=0$.

Define sets $\omega_{m}^{ \pm}$ and $\Omega_{\tilde{m}}^{ \pm}$ as follows:
\begin{equation*}
\begin{aligned}
	& \omega_{m}^{ \pm}:=\left\{(t, x, y) \in \mathbb{R} \times \mathbb{R}^N ;(x, y) \in \Omega_t^{ \pm}, d\left((x, y), \Gamma_t\right) > m \right\}, \\
	& \Omega_{\tilde{m}}^{ \pm}:=\left\{(t, x, y) \in \mathbb{R} \times \mathbb{R}^N ;(x, y) \in \widetilde{\Omega}_t^{ \pm}, d\left((x, y), \widetilde{\Gamma}_t\right) > \tilde{m}\right\}.
\end{aligned}
\end{equation*}
Next, we prove $T_{min}=0$ by contradiction. Suppose $T_{min}>0$, then there may be two cases
\begin{equation}\label{3.39}
\inf \left\{\bar{V}\left(t+T_{min }, x, y\right)-u_i(t, x, y) ;(t, x, y) \in \mathbb{R} \times \mathbb{R}^N \backslash\left(\omega_{m}^{-} \cup \Omega_{\tilde{m}}^{+}\right)\right\}>0
\end{equation}
or
\begin{equation}\label{3.40}
\inf \left\{\bar{V}\left(t+T_{min }, x, y\right)-u_i(t, x, y) ;(t, x, y) \in \mathbb{R} \times \mathbb{R}^N \backslash\left(\omega_{m}^{-} \cup \Omega_{\tilde{m}}^{+}\right)\right\}=0.
\end{equation}

If \eqref{3.39} occurs, there is a constant $\eta_0>0$ such that for any $\eta \in\left(0, 
\eta_0\right]$, $\bar{V}\left(t+T_{min }-\eta, x, y\right)-u_i(t, x, y) \geq 0$ for any $(t, x, y) \in 
\mathbb{R} \times \mathbb{R}^N \backslash\left(\omega_{m}^{-} \cup 
\Omega_{\tilde{m}}^{+}\right)$. Clearly, it holds that $\bar{V}\left(t+T_{min}-\eta, x, 
y\right)-u_i(t, x, y) \geq 0$ on $\partial \omega_{m}^{-} \cup \partial\Omega_{\tilde{m}}^{+}$. 
Note that for any $\eta \in\left(0, \eta_0\right]$ (even if it indicates decreasing $\eta_0>0$), $u_i(t, 
x, y) \leq \gamma_{\star}$ in $\omega_{m}^{-}$ and $\bar{V}\left(t+T_{min }-\eta, x, y\right) 
\geq 1-\gamma_{\star}$ in $\Omega_{\tilde{m}}^{+}$. Then by the similar discussions of Steps 
$2$-$3$ in the proof of Lemma \ref{Lemma 2.15}, it holds that $\bar{V}\left(t+T_{min }-\eta, x, 
y\right) \geq u_i(t, x, y)$ for $(t, x, y) \in \mathbb{R} \times \mathbb{R}^N$, which contradicts 
the definition of $T_{min}$.

If \eqref{3.40} occurs, there exists a sequence $\left\{\left(t_n, x_n, y_n\right)\right\}_{n \in \mathbb{N}}$ of $\mathbb{R} \times \mathbb{R}^N \backslash\left(\omega_m^{-} \cup \Omega_{\tilde{m}}^{+}\right)$ such that
\begin{equation*}
	\bar{V}\left(t_n+T_{\min }, x_n, y_n\right)-u_i\left(t_n, x_n, y_n\right) \rightarrow 0 \text { as } n \rightarrow+\infty.
\end{equation*}
This shows that $d\left(\left(x_n, y_n\right), \mathcal{R}+\hat{c} t_n e_0\right)<+\infty$, since that for $T_{min }>0$,
\begin{equation*}
\left\{\begin{array}{r}
	\bar{V}\left(t+T_{min}, x, y\right) \rightarrow u_i\left(t+T_{min}, x, y\right)>u_i(t, x, y) \text { for all }(t, x, y) \in \mathbb{R} \times \mathbb{R}^N \text { such that } \\
	\qquad d\left((x, y), \widetilde{Q}_i+\hat{c} t e_0\right)<+\infty \text { and } d\left((x, y), \mathcal{R}+\hat{c} t e_0\right) \rightarrow+\infty, \\
	\bar{V}\left(t+T_{min}, x, y\right) \rightarrow u_j\left(t+T_{min}, x, y\right)>u_i(t, x, y) \text { for all }(t, x, y) \in \mathbb{R} \times \mathbb{R}^N \text { such that } \\
	d\left((x, y), \widetilde{Q}_j+\hat{c} t e_0\right)<+\infty \text { and } d\left((x, y), \mathcal{R}+\hat{c} t e_0\right) \rightarrow+\infty, \forall j \in\{1, \ldots, n\} \backslash\{i\}, \\
	\bar{V}\left(t+T_{min}, x, y\right) \rightarrow \underline{V}\left(t+T_{min}, x, y\right)>u_i(t, x, y) \text { for all }(t, x, y) \in \mathbb{R} \times \mathbb{R}^N \text { such that } \\
	d\left((x, y), \partial \mathcal{Q}+\hat{c} t e_0\right) \rightarrow+\infty .
\end{array}\right.
\end{equation*}
Then we take another sequence $\left\{\left(t_n-1, x_n^{\prime}, y_n^{\prime}\right)\right\}$ such that $d\left(\left(x_n^{\prime}, y_n^{\prime}\right),\left(x_n, y_n\right)\right)$ is bounded and $d\left(\left(x_n^{\prime}, y_n^{\prime}\right), \mathcal{R}+\hat{c} (t_n - 1) e_0\right)$ is large enough such that $\bar{V}\left(t_n-1+T_{min}, x_n^{\prime}, y_n^{\prime}\right)>$ $u_i\left(t_n-1, x_n^{\prime}, y_n^{\prime}\right)$. However, by parabolic estimates, one has $\bar{V}\left(t_n-1+T_{min}, x_n^{\prime}, y_n^{\prime}\right)-$ $u_i\left(t_n-1, x_n^{\prime}, y_n^{\prime}\right) \rightarrow 0$, which contradicts the definition of $T_{min}$. Thus, $T_{min}=0$ and $\bar{V}(t, x, y) \geq$ $U_{e_i}\left(x \cdot \nu_i \cos \theta_i+(y-\hat{c} t)\right.$ $\left. \sin \theta_i, x, y\right)$ for all $(t, x, y) \in \mathbb{R} \times \mathbb{R}^N$.

In conclusion, one has that
\begin{equation*}
	\bar{V}(t, x, y) \geq U_{e_i}\left(x \cdot \nu_i \cos \theta_i+(y-\hat{c} t) \sin \theta_i, x, y\right)
\end{equation*}
for all $i \in\{1, \ldots, n\}$ and thereby obtain $\bar{V}(t, x, y) \geq \underline{V}(t, x, y)$ for all $(t, x, y) \in \mathbb{R} \times \mathbb{R}^N$. This completes the proof.
\end{proof}

We point out that there exist positive constants $v_{\star}$ and $C_{\star}$ such that
\begin{equation}\label{3.41}
\frac{\left|\underline{V}(t, x, y)\right|+\left|\bar{V}(t, x, y)\right|}{\min \left\{1, e^{-2 v_{\star} \min_{1\leq i\leq n} \left\{x \cdot \nu_i \cot \theta_i+y-\hat{c} t \right\}}\right\}} \leq C_{\star},\; (t, x, y)\in\mathbb{R} \times \mathbb{R}^N.
\end{equation}
By the properties of $e^{-x}$, we only need to consider $\min_{1\leq i\leq n} \left\{x \cdot \nu_i \cot \theta_i+y-\hat{c} t \right\}>0$. 

\textbf{Situation 1:} $\min_{1\leq i\leq n} \left\{x \cdot \nu_i \cot \theta_i+y-\hat{c} t \right\}>0$ and $\xi(t, x, y)>1$.
By Theorem \ref{Theorem 2.8}, one has
\begin{equation*}
	\left|\underline{V}(t, x, y)\right|\leq\bar{K} e^{-\frac{3 \kappa}{4}\min_{1 \leq i \leq n}\{x \cdot \nu_i \cos \theta_i+(y-\hat{c} t) \sin \theta_i\} }\leq\bar{K} e^{-2v\min_{1 \leq i \leq n}\{x \cdot \nu_i \cot \theta_i+(y-\hat{c} t)\} }
\end{equation*}
for all $v \leq 3 \kappa \min_{1 \leq i \leq n} \{\sin\theta_i\} / 8$. Denote $\widetilde{C}:=\frac{\sup_{\mathbb{R}^{N-1}}|\varphi(\alpha x)-\psi(\alpha x)|}{\alpha}$. It holds that
	\begin{align*}
		&\left|U_{e(x)}(\xi, x, y)\right|+\left|U_{e_i}^\beta(\eta, x, y)\right|\\
		&\leq\bar{K} e^{-\frac{3 \kappa }{4}\xi}+\left(\bar{K} e^{-\frac{3 \kappa}{4}\eta}\right)^\beta\\
		&\leq\bar{K} e^{-\frac{3 \kappa }{4}\frac{y-\hat{c} t-\varphi(\alpha x) / \alpha}{\sqrt{1+|\nabla \varphi(\alpha x)|^2}}}+\left(\bar{K} e^{-\frac{3 \kappa}{4}(y-\hat{c} t-\varphi(\alpha x) / \alpha)}\right)^\beta\\
		&\leq\bar{K} e^{-\frac{3 \kappa }{4}(y-\hat{c} t-\varphi(\alpha x) / \alpha)\frac{1}{\sqrt{1+|\nabla \varphi(\alpha x)|^2}}}+\bar{K}^{\beta}e^{-\frac{3 \kappa}{4}\beta(y-\hat{c} t-\varphi(\alpha x) / \alpha)}\\
		&\leq\bar{K} e^{-\frac{3 \kappa }{4}(y-\hat{c} t-(\psi(x)+\sup_{\mathbb{R}^{N-1}}|\varphi(\alpha x)-\psi(\alpha x)| / \alpha))\frac{1}{\sqrt{1+|\nabla \varphi(\alpha x)|^2}}}+\bar{K}^{\beta}e^{-\frac{3 \kappa}{4}\beta(y-\hat{c} t-\varphi(\alpha x) / \alpha)}\\
		&\leq\bar{K} e^{-\frac{3 \kappa }{4}(y-\hat{c} t-(\max_{1\leq i \leq n}\{- x \cdot \nu_i \cot \theta_i\}+\sup_{\mathbb{R}^{N-1}}|\varphi(\alpha x)-\psi(\alpha x)| / \alpha))\frac{1}{\sqrt{1+|\nabla \varphi(\alpha x)|^2}}}+\bar{K}^{\beta}e^{-\frac{3 \kappa}{4}\beta(y-\hat{c} t-\varphi(\alpha x) / \alpha)}\\
		&\leq\bar{K} e^{-\frac{3 \kappa }{4}(y-\hat{c} t+(\min_{1\leq i \leq n}\{x \cdot \nu_i \cot \theta_i\}-\sup_{\mathbb{R}^{N-1}}|\varphi(\alpha x)-\psi(\alpha x)| / \alpha))\frac{1}{\sqrt{1+|\nabla \varphi(\alpha x)|^2}}}+\bar{K}^{\beta}e^{-\frac{3 \kappa}{4}\beta(y-\hat{c} t-\varphi(\alpha x) / \alpha)}\\
		&\leq\bar{K} e^{-\frac{3 \kappa }{4}(y-\hat{c} t+\min_{1\leq i \leq n}\{x \cdot \nu_i \cot \theta_i\}- \widetilde{C})\frac{1}{\sqrt{1+|\nabla \varphi(\alpha x)|^2}}}+\bar{K}^{\beta}e^{-\frac{3 \kappa}{4}\beta(y-\hat{c} t+\min_{1\leq i \leq n}\{x \cdot \nu_i \cot \theta_i\}- \widetilde{C})}\\
		&\leq\bar{K} e^{\frac{3 \kappa }{4}\widetilde{C}}\cdot e^{-\frac{3 \kappa }{4}\min_{1\leq i \leq n}\{y-\hat{c} t+x \cdot \nu_i \cot \theta_i\}\frac{1}{\sqrt{1+|\nabla \varphi(\alpha x)|^2}}}+\bar{K} e^{\frac{3 \kappa }{4}\beta\widetilde{C}}\cdot e^{-\frac{3 \kappa }{4}\beta\min_{1\leq i \leq n}\{y-\hat{c} t+x \cdot \nu_i \cot \theta_i\}}\\
		&\leq 2K_{1}^{*} e^{-2v \min_{1\leq i \leq n}\{x \cdot \nu_i \cot \theta_i+y-\hat{c} t\}},
	\end{align*}
where $K_{1}^{*} =\max\{\bar{K}, \bar{K} e^{\frac{3 \kappa }{4}\widetilde{C}}\}$ and $v \leq\frac{3 \kappa }{8}\min\{\beta, \frac{1}{\sqrt{1+(C_3+\max_{1\le i\le n} \{|\nu_i \cot \theta_i|\})^2}}\}$.

\textbf{Situation 2:} $\min_{1\leq i\leq n} \left\{x \cdot \nu_i \cot \theta_i+y-\hat{c} t \right\}>0$ and $0 \leq \xi(t, x, y)\leq1$. By Theorem \ref{Theorem 2.8}, we obtain
\begin{equation*}
	\left|\underline{V}(t, x, y)\right|\leq\bar{K} e^{-\frac{3 \kappa}{4}\min_{1 \leq i \leq n}\{x \cdot \nu_i \cos \theta_i+(y-\hat{c} t) \sin \theta_i\} }\leq\bar{K} e^{-2v\min_{1 \leq i \leq n}\{x \cdot \nu_i \cot \theta_i+(y-\hat{c} t)\} }
\end{equation*}
for all $v \leq 3 \kappa \min_{1 \leq i \leq n} \{\sin\theta_i\} / 8$. Moreover, we have
\begin{equation*}
	1\geq\xi(t, x, y)=\frac{y-\hat{c} t-\varphi(\alpha x) / \alpha}{\sqrt{1+|\nabla \varphi(\alpha x)|^2}}\geq\frac{y-\hat{c} t-\psi(x)-\widetilde{C}}{\sqrt{1+|\nabla \varphi(\alpha x)|^2}}= \frac{\min_{1 \leq i \leq n}\{y-\hat{c} t+x \cdot \nu_i \cot \theta_i\}-\widetilde{C}}{\sqrt{1+|\nabla \varphi(\alpha x)|^2}}
\end{equation*}
and then
\begin{equation*}
	0<\min_{1 \leq i \leq n}\{y-\hat{c} t+x \cdot \nu_i \cot \theta_i\}\leq\sqrt{1+|\nabla \varphi(\alpha x)|^2}+\widetilde{C}\leq \sqrt{1+(C_3+\max_{1\le i\le n} \{|\nu_i \cot \theta_i|\})^2}+\widetilde{C}=:\hat{C}.
\end{equation*}
Then one gets if that $0 \leq \xi(t, x, y)\leq 1$,
	\begin{align*}
		\bar{V}(t,x,y)&\leq \left|U_{e(x)}(\xi(t, x, y), x, y)\right|+1\\
		&\leq \bar{K}e^{\frac{3 \kappa}{4}\widetilde{C}} e^{-\frac{3 \kappa}{4}\frac{\min_{1 \leq i \leq n}\{y-\hat{c} t+x \cdot \nu_i \cot \theta_i\}}{\sqrt{1+|\nabla \varphi(\alpha x)|^2}}}+e^{\frac{3 \kappa}{4}\min_{1 \leq i \leq n}\{y-\hat{c} t+x \cdot \nu_i \cot \theta_i\}} e^{-\frac{3 \kappa}{4}\min_{1 \leq i \leq n}\{y-\hat{c} t+x \cdot \nu_i \cot \theta_i\}}\\
		&\leq 2K_{2}^{*}  e^{-2 v \min_{1\leq i\leq n} \left\{x \cdot \nu_i \cot \theta_i+y-\hat{c} t \right\}},
	\end{align*}
where $K_{2}^{*}:=\max\{\bar{K}, \bar{K}e^{\frac{3 \kappa}{4}\widetilde{C}}, e^{\frac{3 \kappa}{4}\hat{C}}\}$ and $v \leq\frac{3 \kappa }{8}\min\{1, \frac{1}{\sqrt{1+(C_3+\max_{1\le i\le n} \{|\nu_i \cot \theta_i|\})^2}}\}$.

\textbf{Situation 3:} $\min_{1\leq i\leq n} \left\{x \cdot \nu_i \cot \theta_i+y-\hat{c} t \right\}>0$ and $\xi(t, x, y)< 0$. By Theorem \ref{Theorem 2.8}, we know
\begin{equation*}
	\left|\underline{V}(t, x, y)\right|\leq\bar{K} e^{-\frac{3 \kappa}{4}\min_{1 \leq i \leq n}\{x \cdot \nu_i \cos \theta_i+(y-\hat{c} t) \sin \theta_i\} }\leq\bar{K} e^{-2v\min_{1 \leq i \leq n}\{x \cdot \nu_i \cot \theta_i+(y-\hat{c} t)\} }
\end{equation*}
for all $v \leq 3 \kappa \min_{1 \leq i \leq n} \{\sin\theta_i\} / 8$. Furthermore, one has
\begin{equation*}
	0\geq\xi(t, x, y)=\frac{y-\hat{c} t-\varphi(\alpha x) / \alpha}{\sqrt{1+|\nabla \varphi(\alpha x)|^2}}\geq\frac{y-\hat{c} t-\psi(x)-\widetilde{C}}{\sqrt{1+|\nabla \varphi(\alpha x)|^2}}= \frac{\min_{1 \leq i \leq n}\{y-\hat{c} t+x \cdot \nu_i \cot \theta_i\}-\widetilde{C}}{\sqrt{1+|\nabla \varphi(\alpha x)|^2}},
\end{equation*}
\begin{equation*}
	0<\min_{1 \leq i \leq n}\{y-\hat{c} t+x \cdot \nu_i \cot \theta_i\}\leq \widetilde{C},
\end{equation*}
and 
\begin{equation*}
	\frac{y-\hat{c} t-\varphi(\alpha x) / \alpha}{\sqrt{1+|\nabla \varphi(\alpha x)|^2}} \leq \frac{y-\hat{c} t-\psi(x)}{\sqrt{1+|\nabla \varphi(\alpha x)|^2}}= \frac{\min_{1 \leq i \leq n}\{y-\hat{c} t+x \cdot \nu_i \cot \theta_i\}}{\sqrt{1+|\nabla \varphi(\alpha x)|^2}}.
\end{equation*}
Thus, one gets if that $\xi(t, x, y)< 0$,
\begin{equation*}
	\begin{aligned}
		\bar{V}(t, x, y)&\leq \left|U_{e(x)}(\xi(t, x, y), x, y)\right|+1\\
		&\leq 1+\bar{K} e^{\kappa_2 \frac{y-\hat{c} t-\varphi(\alpha x) / \alpha}{\sqrt{1+|\nabla \varphi(\alpha x)|^2}}}+1\\
		&\leq \bar{K} e^{\kappa_2 \frac{y-\hat{c} t-\varphi(\alpha x) / \alpha}{\sqrt{1+(C_3+\max_{1\le i\le n} \{|\nu_i \cot \theta_i|\})^2}}}+2\\
		&\leq \bar{K}  e^{\kappa_2 \frac{\min_{1 \leq i \leq n}\{y-\hat{c} t+x \cdot \nu_i \cot \theta_i\} }{\sqrt{1+(C_3+\max_{1\le i\le n} \{|\nu_i \cot \theta_i|\})^2}}}+2\\
		&\leq \bar{K} e^{\kappa_2 \min_{1 \leq i \leq n}\{y-\hat{c} t+x \cdot \nu_i \cot \theta_i\} }+2\\
		&\leq \bar{K} e^{2\kappa_2 \min_{1 \leq i \leq n}\{y-\hat{c} t+x \cdot \nu_i \cot \theta_i\} } \cdot e^{-\kappa_2 \min_{1 \leq i \leq n}\{y-\hat{c} t+x \cdot \nu_i \cot \theta_i\} }\\
		&\quad+2 e^{\kappa_2 \min_{1 \leq i \leq n}\{y-\hat{c} t+x \cdot \nu_i \cot \theta_i\} } \cdot e^{-\kappa_2 \min_{1 \leq i \leq n}\{y-\hat{c} t+x \cdot \nu_i \cot \theta_i\} }\\
		&\leq \bar{K} e^{2\kappa_2 \widetilde{C} } \cdot e^{-\kappa_2 \min_{1 \leq i \leq n}\{y-\hat{c} t+x \cdot \nu_i \cot \theta_i\} }+2 e^{\kappa_2 \widetilde{C} } \cdot e^{-\kappa_2 \min_{1 \leq i \leq n}\{y-\hat{c} t+x \cdot \nu_i \cot \theta_i\} }\\
		&\leq 3 K_{3}^{*} e^{-2 v \min_{1\leq i\leq n} \left\{x \cdot \nu_i \cot \theta_i+y-\hat{c} t \right\}}
	\end{aligned}
\end{equation*}
where $K_{3}^{*}:=\max\{\bar{K}, \bar{K}e^{2 \kappa_2 \widetilde{C}}, e^{2 \kappa_2 \widetilde{C}}\}$ and $v\leq\frac{ \kappa_2 }{2}$. In conclusion, let $C_{\star}:=\max\{3K_{1}^{*},  3K_{2}^{*}, 4K_{3}^{*}\}$ and $v_{\star}\leq\min\{\frac{ \kappa_2 }{2}, \frac{3 \kappa }{8}\min_{1\leq i \leq n}\{\sin\theta_i,\beta, 1, \frac{1}{\sqrt{1+(C_3+\max_{1\le i\le n} \{|\nu_i \cot \theta_i|\})^2}}\}\}$, then \eqref{3.41} is valid.

Moreover, we can also get that there exists positive constants $v^{\star}$ and $C^{\star}$ such that
\begin{equation}\label{3.41.2}
	\frac{\left|\bar{V}(t, x, y)-\underline{V}(t, x, y)\right|}{\min \left\{1, e^{-2 v^{\star} \min_{1\leq i\leq n} \left\{x \cdot \nu_i \cot \theta_i+y-\hat{c} t \right\}}\right\}} \leq C^{\star} \varepsilon,\; (t, x, y)\in\mathbb{R} \times \mathbb{R}^N.
\end{equation}
We only need to consider $\min_{1\leq i\leq n} \left\{x \cdot \nu_i \cot \theta_i+y-\hat{c} t \right\}>0$. For $\min_{1\leq i\leq n} \left\{x \cdot \nu_i \cot \theta_i+y-\hat{c} t \right\}\leq 0$, we can see \eqref{3.6}.

\textbf{Situation 1:} $d(x, \widehat{\mathcal{R}}) \rightarrow +\infty$.
By \eqref{2.10}, we have $|\varphi(x)-\psi(x)| \rightarrow 0$ as $d(x, \widehat{\mathcal{R}}) \rightarrow+\infty$. If $y-\hat{c} t-\psi(x)\rightarrow+\infty$, then one gets that $\xi(t, x, y)=\frac{y-\hat{c} t-\varphi(\alpha x) / \alpha}{\sqrt{1+|\nabla \varphi(\alpha x)|^2}}\rightarrow+\infty$ and $\eta(t, x, y)=y-\hat{c} t-\varphi(\alpha x) / \alpha\rightarrow+\infty$. For convenience, denote $\mu(t, x, y)=\min_{1 \leq i \leq n}\{x \cdot \nu_i \cos \theta_i+(y-\hat{c} t) \sin \theta_i\}$. It holds that 
\begin{equation*}
	0<\min_{1\leq i\leq n} \left\{x \cdot \nu_i \cot \theta_i+y-\hat{c} t \right\}\cdot \min_{1\leq i\leq n}\left\{\sin \theta_{i} \right\} \leq \mu(t, x, y) \leq \min_{1\leq i\leq n} \left\{x \cdot \nu_i \cot \theta_i+y-\hat{c} t \right\}= y-\hat{c} t-\psi(x)
\end{equation*}
and
\begin{equation*}
0<\frac{y-\hat{c} t-\psi(x)-\widetilde{C}}{\sqrt{1+(C_3+\max_{1\le i\le n} \{|\nu_i \cot \theta_i|\})^2}}\leq \frac{y-\hat{c} t-\psi(x)-\widetilde{C}}{\sqrt{1+|\nabla \varphi(\alpha x)|^2}} \leq \xi(t, x, y) \leq \eta(t, x, y) \leq y-\hat{c} t-\psi(x).
\end{equation*}
Then one has
	\begin{align*}
	&\frac{\left|\bar{V}(t, x, y)-\underline{V}(t, x, y)\right|}{ e^{-2 v^{\star} \min_{1\leq i\leq n} \left\{x \cdot \nu_i \cot \theta_i+y-\hat{c} t \right\}}} \\
	&=\frac{\left|U_{e(x)}(\xi, x, y)+\varepsilon h(\alpha x) \times U_{e_i}^\beta(\eta, x, y)-U_{e_i}(\mu, x, y)\right|}{ e^{-2 v^{\star} \min_{1\leq i\leq n} \left\{x \cdot \nu_i \cot \theta_i+y-\hat{c} t \right\}}}\\
	&\leq \frac{\left|U_{e(x)}(\xi, x, y)-U_{e(x)}(\mu, x, y)+U_{e(x)}(\mu, x, y)-U_{e_i}(\mu, x, y)+\varepsilon h(\alpha x) \times U_{e_i}^\beta(\eta, x, y)\right|}{ e^{-2 v^{\star} \min_{1\leq i\leq n} \left\{x \cdot \nu_i \cot \theta_i+y-\hat{c} t \right\}}}\\
	&\leq \frac{\left|\partial_{\xi}U_{e(x)}((1-\tilde{\theta})\xi+\tilde{\theta}\mu, x, y)(\xi-\mu)\right|+\left|U_{e_i}^\prime(\mu, x, y) (e(x)-e_i)\right|+\left|\varepsilon h(\alpha x) \times U_{e_i}^\beta(\eta, x, y)\right|}{ e^{-2 v^{\star} \min_{1\leq i\leq n} \left\{x \cdot \nu_i \cot \theta_i+y-\hat{c} t \right\}}}\\
	&\leq \frac{\bar{K} e^{-\frac{3 \kappa}{4} ((1-\tilde{\theta})\xi+\tilde{\theta}\mu)}\left|\xi-\mu\right|+M_{1} e^{-\frac{\kappa}{2} \mu}\left|e(x)-e_i\right|+\left|\varepsilon \bar{K} e^{-\frac{3 \kappa}{4} \eta}\right|}{ e^{-2 v^{\star} \min_{1\leq i\leq n} \left\{x \cdot \nu_i \cot \theta_i+y-\hat{c} t \right\}}}\\
	&\leq \frac{\bar{K} e^{-\frac{3 \kappa}{4} \left((1-\tilde{\theta})\frac{y-\hat{c} t-\psi(x)-\widetilde{C}}{\sqrt{1+(C_3+\max_{1\le i\le n} \{|\nu_i \cot \theta_i|\})^2}}+\tilde{\theta}\min_{1\leq i\leq n} \left\{x \cdot \nu_i \cot \theta_i+y-\hat{c} t \right\}\cdot \min_{1\leq i\leq n}\left\{\sin \theta_{i} \right\}\right)}\left|\xi-\mu\right|}{ e^{-2 v^{\star} \min_{1\leq i\leq n} \left\{x \cdot \nu_i \cot \theta_i+y-\hat{c} t \right\}}}\\
	&\quad+\frac{M_{1} e^{-\frac{\kappa}{2} (\min_{1\leq i\leq n} \left\{x \cdot \nu_i \cot \theta_i+y-\hat{c} t \right\}\cdot \min_{1\leq i\leq n}\left\{\sin \theta_{i} \right\} )}\left|e(x)-e_i\right|}{ e^{-2 v^{\star} \min_{1\leq i\leq n} \left\{x \cdot \nu_i \cot \theta_i+y-\hat{c} t \right\}}}+ \frac{\left|\varepsilon  \bar{K} e^{-\frac{3 \kappa}{4} \beta(y-\hat{c} t-\psi(x)-\widetilde{C})}\right|}{ e^{-2 v^{\star} \min_{1\leq i\leq n} \left\{x \cdot \nu_i \cot \theta_i+y-\hat{c} t \right\}}}\\
	&=:O_{1}+O_{2}+O_{3},
	\end{align*}
where $\tilde{\theta}\in(0,1)$ is some constant and $\widetilde{C}=\frac{\sup_{\mathbb{R}^{N-1}}|\varphi(\alpha x)-\psi(\alpha x)|}{\alpha} < +\infty$. Futhermore, we can get
	\begin{align*}
		O_{1}&= \frac{\bar{K} e^{-\frac{3 \kappa}{4} \left((1-\tilde{\theta})\frac{\min_{1\leq i\leq n} \left\{x \cdot \nu_i \cot \theta_i+y-\hat{c} t \right\}-\widetilde{C}}{\sqrt{1+(C_3+\max_{1\le i\le n} \{|\nu_i \cot \theta_i|\})^2}}+\tilde{\theta}\min_{1\leq i\leq n} \left\{x \cdot \nu_i \cot \theta_i+y-\hat{c} t \right\}\cdot \min_{1\leq i\leq n}\left\{\sin \theta_{i} \right\}\right)}\left|\xi-\mu\right|}{ e^{-2 v^{\star} \min_{1\leq i\leq n} \left\{x \cdot \nu_i \cot \theta_i+y-\hat{c} t \right\}}}\\
		&\leq \bar{K} e^{\frac{3 \kappa}{4} \left((1-\tilde{\theta})\frac{\widetilde{C}}{\sqrt{1+(C_3+\max_{1\le i\le n} \{|\nu_i \cot \theta_i|\})^2}}\right)} \\
		&\qquad\times \frac{ e^{-\frac{3 \kappa}{4} \left((1-\tilde{\theta})\frac{\min_{1\leq i\leq n} \left\{x \cdot \nu_i \cot \theta_i+y-\hat{c} t \right\}}{\sqrt{1+(C_3+\max_{1\le i\le n} \{|\nu_i \cot \theta_i|\})^2}}+\tilde{\theta}\min_{1\leq i\leq n} \left\{x \cdot \nu_i \cot \theta_i+y-\hat{c} t \right\}\cdot \min_{1\leq i\leq n}\left\{\sin \theta_{i} \right\}\right)}\left|\xi-\mu\right|}{ e^{-2 v^{\star} \min_{1\leq i\leq n} \left\{x \cdot \nu_i \cot \theta_i+y-\hat{c} t \right\}}}\\
		&\leq 2\bar{K} e^{\frac{3 \kappa}{4} \left((1-\tilde{\theta})\frac{\widetilde{C}}{\sqrt{1+(C_3+\max_{1\le i\le n} \{|\nu_i \cot \theta_i|\})^2}}\right)} \cdot \min_{1\leq i\leq n} \left\{x \cdot \nu_i \cot \theta_i+y-\hat{c} t \right\} \\
		&\qquad \times\frac{ e^{-\frac{3 \kappa}{4} \left((1-\tilde{\theta})\frac{\min_{1\leq i\leq n} \left\{x \cdot \nu_i \cot \theta_i+y-\hat{c} t \right\}}{\sqrt{1+(C_3+\max_{1\le i\le n} \{|\nu_i \cot \theta_i|\})^2}}+\tilde{\theta}\min_{1\leq i\leq n} \left\{x \cdot \nu_i \cot \theta_i+y-\hat{c} t \right\}\cdot \min_{1\leq i\leq n}\left\{\sin \theta_{i} \right\}\right)}}{ e^{-2 v^{\star} \min_{1\leq i\leq n} \left\{x \cdot \nu_i \cot \theta_i+y-\hat{c} t \right\}}}\\
		&\leq 2\bar{K} e^{\frac{3 \kappa}{4} \left((1-\tilde{\theta})\frac{\widetilde{C}}{\sqrt{1+(C_3+\max_{1\le i\le n} \{|\nu_i \cot \theta_i|\})^2}}\right)} \cdot \min_{1\leq i\leq n} \left\{x \cdot \nu_i \cot \theta_i+y-\hat{c} t \right\} \\
		&\qquad \times\frac{ e^{-\frac{3 \kappa}{4} \left((1-\tilde{\theta})\frac{1}{\sqrt{1+(C_3+\max_{1\le i\le n} \{|\nu_i \cot \theta_i|\})^2}}+\tilde{\theta}\min_{1\leq i\leq n}\left\{\sin \theta_{i} \right\}\right)\min_{1\leq i\leq n} \left\{x \cdot \nu_i \cot \theta_i+y-\hat{c} t \right\}}}{ e^{-2 v^{\star} \min_{1\leq i\leq n} \left\{x \cdot \nu_i \cot \theta_i+y-\hat{c} t \right\}}}.
	\end{align*}
If set $v^{\star}<v_{1}^{\star}:=\frac{3 \kappa}{8}\min\left\{\frac{1}{\sqrt{1+(C_3+\max_{1\le i\le n} \{|\nu_i \cot \theta_i|\})^2}}, \min_{1\leq i\leq n}\left\{\sin \theta_{i} \right\}\right\}$, then one gets $O_{1}\leq 2\bar{K} e^{\frac{3 \kappa}{4} \widetilde{C}}\varepsilon$ as $d(x, \widehat{\mathcal{R}}) \rightarrow +\infty$. If set $v^{\star}<v_{2}^{\star}:=\frac{\kappa}{4} \min_{1\leq i\leq n}\left\{\sin \theta_{i} \right\}$, then one knows
	\begin{align*}
		O_{2}&=\frac{M_{1} e^{-\frac{\kappa}{2} (\min_{1\leq i\leq n} \left\{x \cdot \nu_i \cot \theta_i+y-\hat{c} t \right\}\cdot \min_{1\leq i\leq n}\left\{\sin \theta_{i} \right\} )}\left|e(x)-e_i\right|}{ e^{-2 v^{\star} \min_{1\leq i\leq n} \left\{x \cdot \nu_i \cot \theta_i+y-\hat{c} t \right\}}}\\
		&\leq 2 M_{1} \frac{e^{-\frac{\kappa}{2} \min_{1\leq i\leq n}\left\{\sin \theta_{i} \right\} (\min_{1\leq i\leq n} \left\{x \cdot \nu_i \cot \theta_i+y-\hat{c} t \right\} )}}{ e^{-2 v^{\star} \min_{1\leq i\leq n} \left\{x \cdot \nu_i \cot \theta_i+y-\hat{c} t \right\}}},
	\end{align*}
thus it holds that $O_{2}\leq 2 M_{1} \varepsilon$ as $d(x, \widehat{\mathcal{R}}) \rightarrow +\infty$. If $v^{\star}<v_{3}^{\star}:=\frac{3\kappa}{8} \beta$, then we know
\begin{align*}
	O_{3}&=\frac{\left|\varepsilon  \bar{K} e^{-\frac{3 \kappa}{4} \beta(y-\hat{c} t-\psi(x)-\widetilde{C})}\right|}{ e^{-2 v^{\star} \min_{1\leq i\leq n} \left\{x \cdot \nu_i \cot \theta_i+y-\hat{c} t \right\}}}\\
	&=\frac{\varepsilon  \bar{K} e^{-\frac{3 \kappa}{4}\beta (\min_{1\leq i\leq n} \left\{x \cdot \nu_i \cot \theta_i+y-\hat{c} t \right\}-\widetilde{C})}}{ e^{-2 v^{\star} \min_{1\leq i\leq n} \left\{x \cdot \nu_i \cot \theta_i+y-\hat{c} t \right\}}}\\ &\leq \varepsilon \bar{K} e^{\frac{3 \kappa}{4} \beta\widetilde{C}} \cdot \frac{ e^{-\frac{3 \kappa}{4}\beta \min_{1\leq i\leq n} \left\{x \cdot \nu_i \cot \theta_i+y-\hat{c} t \right\}}}{ e^{-2 v^{\star} \min_{1\leq i\leq n} \left\{x \cdot \nu_i \cot \theta_i+y-\hat{c} t \right\}}}\leq \varepsilon \bar{K} e^{\frac{3 \kappa}{4} \beta\widetilde{C}}.
\end{align*}
Therefore, if we set $v^{\star}<\min\{v_{1}^{\star}, v_{2}^{\star}, v_{3}^{\star}\}$ and $C_{1}^{\star}=\bar{K} e^{\frac{3 \kappa}{4} \beta\widetilde{C}}+2\bar{K} e^{\frac{3 \kappa}{4} \widetilde{C}}+2 M_{1}$, one has 
\begin{equation*}
	\frac{\left|\bar{V}(t, x, y)-\underline{V}(t, x, y)\right|}{ e^{-2 v^{\star} \min_{1\leq i\leq n} \left\{x \cdot \nu_i \cot \theta_i+y-\hat{c} t \right\}}}\leq C_{1}^{\star}\varepsilon \;\text{ as }\; d(x, \widehat{\mathcal{R}}) \rightarrow +\infty.
\end{equation*}

If $y-\hat{c} t-\psi(x)$ is bounded as $d(x, \widehat{\mathcal{R}}) \rightarrow +\infty$, then $\min_{1\leq i\leq n} \left\{x \cdot \nu_i \cot \theta_i+y-\hat{c} t \right\}>0$ is bounded. There exists $l \in\{1, \ldots, n\}$ such that as $d((x, y), \mathcal{R}) \rightarrow+\infty$, 
\begin{equation*}
	x \cdot \nu_l \cot \theta_l+y-\hat{c} t \text { is bounded and } x \cdot \nu_j \cot \theta_j+y-\hat{c} t  \rightarrow+\infty \text { for } j \neq l.
\end{equation*}
By $\theta_{i} \in (0, \pi/2]$ $(i \in\{1, \ldots, n\})$, we can also get 
\begin{equation*}
	x \cdot \nu_l \cos \theta_l+(y-\hat{c} t) \sin \theta_l \text { is bounded and } x \cdot \nu_j \cos \theta_j+(y-\hat{c} t) \sin \theta_j  \rightarrow+\infty \text { for } j \neq l.
\end{equation*}
This implies that $(x, y)$ is bounded away from $\widetilde{Q}_l +\hat{c} t e_0$ and $d((x, y), \mathcal{R}+\hat{c} t e_0) \rightarrow +\infty$. It also implies that $x \in \widehat{Q}_l$ and $d(x, \widehat{R}) \rightarrow+\infty$. Then by $|\varphi(x)-\psi(x)| \rightarrow 0$ as $d(x, \widehat{\mathcal{R}}) \rightarrow+\infty$, we deduce that
\begin{equation*}
\left|\varphi(x)-\left(- x \cdot \nu_l \cot \theta_l\right)\right| \rightarrow 0 \;\text{ and }\; \left|\nabla \varphi(x)+\nu_l \cot \theta_l\right| \rightarrow 0.
\end{equation*}
Therefore,
\begin{equation}\label{3.41.3}
\xi(t, x, y) \rightarrow x \cdot \nu_l \cos \theta_l+(y-\hat{c} t) \sin \theta_l=\min_{1 \leq i \leq n}\{x \cdot \nu_i \cos \theta_i+(y-\hat{c} t) \sin \theta_i\}=\mu(t, x, y).
\end{equation}
Then it follows that
	\begin{align*}
		&\frac{\left|\bar{V}(t, x, y)-\underline{V}(t, x, y)\right|}{ e^{-2 v^{\star} \min_{1\leq i\leq n} \left\{x \cdot \nu_i \cot \theta_i+y-\hat{c} t \right\}}} \\
		&=\frac{\left|U_{e(x)}(\xi, x, y)+\varepsilon h(\alpha x) \times U_{e_i}^\beta(\eta, x, y)-U_{e_l}(\mu, x, y)\right|}{ e^{-2 v^{\star} \min_{1\leq i\leq n} \left\{x \cdot \nu_i \cot \theta_i+y-\hat{c} t \right\}}}\\
		&\leq \frac{\left|U_{e(x)}(\xi, x, y)-U_{e(x)}(\mu, x, y)+U_{e(x)}(\mu, x, y)-U_{e_l}(\mu, x, y)+\varepsilon h(\alpha x) \times U_{e_i}^\beta(\eta, x, y)\right|}{ e^{-2 v^{\star} \min_{1\leq i\leq n} \left\{x \cdot \nu_i \cot \theta_i+y-\hat{c} t \right\}}}\\
		&\leq \frac{\left|U_{e(x)}(\xi, x, y)-U_{e(x)}(\mu, x, y)\right|+\left|U_{e_i}^\prime(\mu, x, y) (e(x)-e_l)\right|+\left|\varepsilon h(\alpha x) \times U_{e_i}^\beta(\eta, x, y)\right|}{ e^{-2 v^{\star} \min_{1\leq i\leq n} \left\{x \cdot \nu_i \cot \theta_i+y-\hat{c} t \right\}}}\\
		&=: O_{4}+O_{5}+O_{6}.
	\end{align*}
By \eqref{3.41.3}, It is easy to obtain $O_{4}\leq \varepsilon$ as $d(x, \widehat{\mathcal{R}}) \rightarrow+\infty$. Notice that $e(x) \rightarrow\left(\nu_l \cos \theta_l, \sin \theta_l\right)=e_l$ for $x \in \widehat{Q}_l$ as $d(x, \widehat{\mathcal{R}}) \rightarrow +\infty$. If set $v^{\star}<v_{2}^{\star}=\frac{\kappa}{4} \min_{1\leq i\leq n}\left\{\sin \theta_{i} \right\}$, we have
	\begin{align*}
		O_{5}&= \frac{\left|U_{e_i}^\prime(\mu, x, y) (e(x)-e_l)\right|}{ e^{-2 v^{\star} 
		\min_{1\leq i\leq n} \left\{x \cdot \nu_i \cot \theta_i+y-\hat{c} t 
		\right\}}}\leq\frac{\left|U_{e_i}^\prime(\mu, x, y) (e(x)-e_l)\right|}{ e^{-2 v^{\star} 
		\min_{1\leq i\leq n} \left\{x \cdot \nu_i \cot \theta_i+y-\hat{c} t \right\}}}\\
&\leq\frac{M_{1} e^{-\frac{\kappa}{2} (\min_{1\leq i\leq n} \left\{x \cdot \nu_i \cot 
\theta_i+y-\hat{c} t \right\}\cdot \min_{1\leq i\leq n}\left\{\sin \theta_{i} \right\} 
)}\left|e(x)-e_l\right|}{ e^{-2 v^{\star} \min_{1\leq i\leq n} \left\{x \cdot \nu_i \cot 
\theta_i+y-\hat{c} t \right\}}}\leq M_{1} \varepsilon.
	\end{align*}
If set $v^{\star}<v_{3}^{\star}:=\frac{3\kappa}{8} \beta$, then we know
	\begin{align*}
		O_{6}&= \frac{\left|\varepsilon h(\alpha x) \times U_{e_i}^\beta(\eta, x, y)\right|}{ e^{-2 
		v^{\star} \min_{1\leq i\leq n} \left\{x \cdot \nu_i \cot \theta_i+y-\hat{c} t 
		\right\}}}\leq\frac{\varepsilon  \bar{K} e^{-\frac{3 \kappa}{4}\beta (\min_{1\leq i\leq n} 
		\left\{x \cdot \nu_i \cot \theta_i+y-\hat{c} t \right\}-\widetilde{C})}}{ e^{-2 v^{\star} 
		\min_{1\leq i\leq n} \left\{x \cdot \nu_i \cot \theta_i+y-\hat{c} t \right\}}}\\
		&\leq \varepsilon \bar{K} e^{\frac{3 \kappa}{4} \beta\widetilde{C}} \cdot \frac{ e^{-\frac{3 
		\kappa}{4}\beta \min_{1\leq i\leq n} \left\{x \cdot \nu_i \cot \theta_i+y-\hat{c} t \right\}}}{ 
		e^{-2 v^{\star} \min_{1\leq i\leq n} \left\{x \cdot \nu_i \cot \theta_i+y-\hat{c} t \right\}}}\leq 
		\varepsilon \bar{K} e^{\frac{3 \kappa}{4} \beta\widetilde{C}}.
	\end{align*}
Thus, if we set $v^{\star}<\min\{ v_{2}^{\star}, v_{3}^{\star}\}$ and $C_{2}^{\star}=M_{1}+1+\bar{K} e^{\frac{3 \kappa}{4} \beta\widetilde{C}}$, one has 
\begin{equation*}
	\frac{\left|\bar{V}(t, x, y)-\underline{V}(t, x, y)\right|}{ e^{-2 v^{\star} \min_{1\leq i\leq n} \left\{x \cdot \nu_i \cot \theta_i+y-\hat{c} t \right\}}}\leq C_{2}^{\star}\varepsilon \;\text{ as }\; d(x, \widehat{\mathcal{R}}) \rightarrow +\infty.
\end{equation*}

\textbf{Situation 2:} $d(x, \widehat{\mathcal{R}})$ is bounded for all $x\in\mathbb{R}^{N-1}$.
Since $d(x, \widehat{\mathcal{R}})$ is bounded, we can get $y-\hat{c} t-\psi(x)\rightarrow+\infty$. Similar to arguments in Situation $1$, we can also obtain that there exist constants $v_{4}^{\star}>0$ and $C_{3}^{\star}>0$ such that 
\begin{equation*}
	\frac{\left|\bar{V}(t, x, y)-\underline{V}(t, x, y)\right|}{ e^{-2 v^{\star} \min_{1\leq i\leq n} \left\{x \cdot \nu_i \cot \theta_i+y-\hat{c} t \right\}}}\leq C_{3}^{\star}\varepsilon \;\text{ as }\; d(x, \widehat{\mathcal{R}}) \rightarrow +\infty.
\end{equation*}

In conclusion, it holds that there exist $0<v^{\star}<\min\{ v_{1}^{\star}, v_{2}^{\star}, v_{3}^{\star}, v_{4}^{\star}\}$ and $C^{\star}>\max\{C_{1}^{\star}, C_{2}^{\star}, C_{3}^{\star}\}$ such that \eqref{3.41.2} holds true.

\subsection{Proofs of existence and uniqueness}
\noindent
\textbf{Proof of Theorem \ref{Theorem 2.18}.} Since we get the supersolution $\bar{V}(t, x, y)$, we can easily show the existence for $e_0=(0,0, \ldots, 1)$. For general $e_0 \in \mathbb{S}^{N-1}$, we only need to change variables of $\bar{V}(t, x, y)$.

{\it Step 1: the existence for $e_0=(0,0, \ldots, 1)$.} Let $u_n(t, x, y)$ be the solution of following Cauchy problem
\begin{equation}\label{3.42}
\begin{cases}\partial_t u-\Delta_{x, y} u=f(x, y, u) & \text { when } t>-n,(x, y) \in \mathbb{R}^N ,\\ u(t, x, y)=\underline{V}(-n, x, y) & \text { when } t=-n,(x, y) \in \mathbb{R}^N.\end{cases}
\end{equation}
By Lemma \ref{Lemma 3.2} and the comparison principle, it holds that
\begin{equation}\label{3.43}
\underline{V}(t, x, y) \leq u_n(t, x, y) \leq \bar{V}(t, x, y), \text { for } t \geq-n \text { and }(x, y) \in \mathbb{R}^N.
\end{equation}
Since $\underline{V}(t, x, y)$ is a subsolution of \eqref{1.1}, using the comparison principle again, the sequence $u_n(t, x, y)$ is increasing in $n$. By parabolic estimates, applying Theorem $5.1.2$ of \cite{A. Lunardi1995}, there exists a constant $\Lambda$ independent of $n \in \mathbb{N}$ such that
\begin{equation*}
\left\|w_n(\cdot, \cdot, \cdot)\right\|_{C^{1+\frac{\hat{\theta}}{2}, 2+\hat{\theta}}\left([-n+1, +\infty) \times \mathbb{R}^N\right)} \leq \Lambda,
\end{equation*}
for some $\hat{\theta} \in(0,1)$ and all $n \in \mathbb{N}$. Letting $n \rightarrow \infty$, the sequence $\left\{u_n(t, x, y)\right\}_{n \in \mathbb{N}}$ converges to an entire solution $V(t, x, y)$ of \eqref{1.1}. Besides, it follows from \eqref{3.43} that
\begin{equation}\label{3.44}
\underline{V}(t, x, y) \leq V(t, x, y) \leq \bar{V}(t, x, y), \text { for }(t, x, y) \in \mathbb{R} \times \mathbb{R}^N.
\end{equation}
By \eqref{3.41}, one has that 
\begin{equation*}
\sqrt{V(t, x, y)-\underline{V}(t, x, y)} \leq C_{\star} \min \left\{1, e^{- v_{\star} \min_{1\leq i\leq n} \left\{x \cdot \nu_i \cot \theta_i+y-\hat{c} t \right\}}\right\} \; \text{ in }\mathbb{R} \times \mathbb{R}^N.
\end{equation*}
By \eqref{3.6} and \eqref{3.41.2}, letting $\varepsilon \rightarrow 0$ in $\bar{V}(t, x, y)$ yields $0 \leq V \leq 1$ and \eqref{2.14} holds. By \eqref{2.14} and the definition of $\underline{V}(t, x, y)$, we know that $V(t, x, y)$ is a transition front with sets
\begin{equation*}
\left\{\begin{array}{l}
	\Gamma_t=\left\{(x, y) \in \mathbb{R}^N ; \min _{1 \leq i \leq n}\left\{x \cdot \nu_i \cos \theta_i+(y-\hat{c} t) \sin \theta_i\right\}=0\right\}=\partial \mathcal{Q}+\hat{c} t e_0 ,\\
	\Omega_t^{-}=\left\{(x, y) \in \mathbb{R}^N ; \min _{1 \leq i \leq n}\left\{x \cdot \nu_i \cos \theta_i+(y-\hat{c} t) \sin \theta_i\right\}>0\right\}=\mathcal{Q}+\hat{c} t e_0 ,\\
	\Omega_t^{+}=\left\{(x, y) \in \mathbb{R}^N ; \min _{1 \leq i \leq n}\left\{x \cdot \nu_i \cos \theta_i+(y-\hat{c} t) \sin \theta_i\right\}<0\right\}=\mathbb{R}^N \backslash \overline{\mathcal{Q}}+\hat{c} t e_0.
\end{array}\right.
\end{equation*}
Clearly, $V(t, x, y)$ is an invasion of 0 by 1. It follows from \eqref{3.42} that for any $\tau>0, u_n(t+\tau, x, y)$ can solve the following equation 
\begin{equation*}
\begin{cases}\partial_t u-\Delta_{x, y} u=f(x, y, u) & \text { when } t>-n,(x, y) \in \mathbb{R}^N, \\ u(t, x, y)=u_n(-n+\tau, x, y) & \text { when } t=-n,(x, y) \in \mathbb{R}^N,\end{cases}
\end{equation*}
for each $n \in \mathbb{N}$. Moreover, \eqref{3.43} implies that
\begin{equation*}
u_n(-n+\tau, x, y) \geq \underline{V}(-n+\tau, x, y) \geq \underline{V}(-n, x, y)
\end{equation*}
for all $(x, y) \in \mathbb{R}^N$. Then, by the comparison principle, it holds that
\begin{equation*}
u_n(t+\tau, x, y) \geq u_n(t, x, y), \quad \forall(t, x, y) \in(-n,+\infty) \times \mathbb{R}^N, \forall \tau>0,
\end{equation*}
which implies $\partial_t u_n(t, x, y) \geq 0$ in $(-n,+\infty) \times \mathbb{R}^N$. Thus, letting $n \rightarrow \infty$, $\partial_t V(t, x, y) \geq 0$ in $\mathbb{R} \times \mathbb{R}^N$. By virtue of the strong maximum principle, it holds $\partial_t V(t, x, y) > 0$ in $\mathbb{R} \times \mathbb{R}^N$. 

{\it  Step 2: the existence for general $e_0 \in \mathbb{S}^{N-1}$.} Let $Y:=z \cdot e_0$ and 
$Z:=z-\left(z \cdot e_0\right) e_0$, where $z \in \mathbb{R}^{N}$. Then we establish a new 
coordinate system. Specifically, we take the unit vector $e_0$ as the $N$-th axis of the new 
coordinate system and select suitable $N-1$ mutually orthogonal unit vectors as the remaining axes 
so that all these vectors are orthogonal to $e_0$. In this new coordinate system, $Z$ and 
$z\in\mathbb{R}^N$ can be expressed as $(X, 0)=\left(X_1, X_2, \ldots, X_{N-1}, 0\right)$ and 
$\left(X, Y\right)=\left(X_1, X_2, \ldots, X_{N-1}, Y\right)$, respectively. Still define the 
subsolution of \eqref{1.1} as $\underline{V}(t, z)=\max _{1 \leq i \leq n}\left\{U_{e_i}\left(z \cdot 
e_i-c_{e_i} t, z\right)\right\}$. The supersolution $\bar{V}(t, z)$ of \eqref{1.1} is now given by
\begin{equation*}
\bar{V}(t, z):=U_{e(X)}(\xi, z)+\varepsilon h(\alpha X) \times\left[U_{e_i}^\beta(\eta, z) \omega(\xi)+(1-\omega(\xi))\right],
\end{equation*}
where
\begin{equation*}
\xi(t, z)=\frac{Y-\hat{c} t-\varphi(\alpha X) / \alpha}{\sqrt{1+|\nabla \varphi(\alpha X)|^2}} \quad \text{ and }  \quad \eta(t, z)=Y-\hat{c} t-\varphi(\alpha X) / \alpha.
\end{equation*}
We can also show $\bar{V}(t, z)$ is a supersolution of \eqref{1.1} in the new form. As Step $1$, we obtain the existence and monotonicity of the curved front.

\textbf{Proof of Theorem \ref{Theorem 2.19}.} Assume that $V(t, z)$ and $V_1(t, z)$ are both curved fronts of \eqref{1.1} satisfying \eqref{2.15}. Thus, they are both invasions of 0 by 1 with sets \eqref{2.13}. By \eqref{1.10}, there is $R>0$ such that $0<V(t, z), V_1(t, z) \leq \gamma_{\star}$ for $(t, z) \in \omega^{-}$ and $1-\gamma_{\star} \leq V(t, z), V_1(t, z)<1$ for $(t, z) \in \omega^{+}$, where
\begin{align*}
	 &\omega^{+}:=\left\{(t, z) \in \mathbb{R} \times \mathbb{R}^N ; \min _{1 \leq i \leq n}\left\{z \cdot e_i-c_{e_i} t\right\}<-R\right\}, \\
	 &\omega^{-}:=\left\{(t, z) \in \mathbb{R} \times \mathbb{R}^N ; \min _{1 \leq i \leq n}\left\{z \cdot e_i-c_{e_i} t\right\}>R\right\}, 
\end{align*}
and
\begin{equation*}
	\omega:=\left\{(t, z) \in \mathbb{R} \times \mathbb{R}^N ;-R \leq \min _{1 \leq i \leq n}\left\{z \cdot e_i-c_{e_i} t\right\} \leq R\right\}.
\end{equation*}
Below, similar to the proof in Lemma \ref{Lemma 2.15}, we compare $V(t, z)$ and $V_1(t, z)$ using the sliding method. We only provide an overview of the proof.

{\it Step 1: prove $V(t+\tau, x, y) \geq V_1(t, x, y)$ in $\mathbb{R} \times \mathbb{R}^N$ for $\tau$ large enough.} 
Since $V$ is a transition front connecting 0 and 1 with sets $\Gamma_t$ and $\Omega_t^{ \pm}$, it holds that:

(i) $V(t+\tau, x, y) \rightarrow 1$ uniformly as $\tau \rightarrow+\infty$ for $(t, x, y) \in \omega$;

(ii) $V(t+\tau, x, y) \rightarrow 0$ uniformly as $\tau \rightarrow-\infty$ for $(t, x, y) \in \omega$.
\\This indicates that
\begin{equation}\label{3.45}
	V(t+\tau, x, y) \geq V_1(t, x, y) \text { in } \omega \text { for large } \tau .
\end{equation}
Moreover, $V(t+\tau, x, y) \geq V_1(t, x, y)$ on $\partial \omega^{ \pm}$ for large $\tau$. Then, we prove that $V(t+\tau, x, y) \geq V_1(t, x, y)$ in $\omega^{ \pm}$ for large $\tau$. Let
\begin{equation*}
\varepsilon_*=\inf \left\{\varepsilon>0 ; V(t+\tau, x, y) \geq V_1(t, x, y)-\varepsilon \text { in } \omega^{-}\right\}.
\end{equation*}
Since $0\leq V(t, x), V_1(t, x) \leq 1$, $\varepsilon_*$ is well-defined. By the definition of $\omega^{ \pm}$, one has $\varepsilon_* \leq \gamma_{\star}$. Suppose $\varepsilon_*>0$, then
\begin{equation*}
V(t+\tau, x, y)>V_1(t, x, y)-\varepsilon_* \text { on } \partial \omega^{-}.
\end{equation*}
Because $V_1\leq \gamma_{\star}$ in $\omega^{-}$ and $f(x, y, \cdot)$ is nonincreasing in $(-\infty, 2 \gamma_{\star}]$, it holds that
\begin{equation*}
\partial_t\left(V_1-\varepsilon_*\right)-\Delta\left(V_1-\varepsilon_*\right)= f\left(V_1\right) \leq f\left(V_1-\varepsilon_*\right).
\end{equation*}
Define $V^{\star}(t, x, y):=V(t+\tau, x, y)-V_1(t, x, y)+\varepsilon_*$, then $V^{\star}$ is a nonegative function and satisfies $\partial_t V^{\star}-\Delta V^{\star} \geq b^{\star} V^{\star}$ in $\omega^{-}$, where $b^{\star}$ is some bounded function. By the definition of $\varepsilon_*$, there exists a sequence $\left\{\left(t_n, x_n, y_n\right)\right\}_{n \in \mathbb{N}}$ in $\omega^{-}$ such that
\begin{equation}\label{3.46}
V^{\star}(t_n, x_n, y_n)= V\left(t_n+\tau, x_n, y_n\right)-V_1\left(t_n, x_n, y_n\right)+\varepsilon_* \rightarrow 0 \text { as } n \rightarrow+\infty.
\end{equation}
Note that $d\left(\left(t_n, x_n, y_n\right), \partial \omega^{-}\right)$ is bounded. Otherwise, $V\left(t_n+\tau, x_n, y_n\right) \rightarrow 0$ and $V_1\left(t_n, x_n, y_n\right) \rightarrow 0$, then \eqref{3.46} can not be satisfied. Thus, there exists a sequence $\left\{(x_n^{\prime}, y_n^{\prime})\right\}$ such that $\left(t_n-1, x_n^{\prime}, y_n^{\prime}\right) \in \partial \omega^{-}$ and $\left|\left(x_n, y_n\right)-\left(x_n^{\prime}, y_n^{\prime}\right)\right|<+\infty$. By linear parabolic estimates, we can obtain $V^{\star}\left(t_n-1, x_n^{\prime}, y_n^{\prime}\right) \rightarrow 0$ from $V^{\star}\left(t_n, x_n, y_n\right) \rightarrow 0$. But it is impossible, because the above conclusion contradicts $V(t+\tau, x, y)>V_1(t, x, y)-\varepsilon_*$ on $\partial \omega^{-}$ and $\varepsilon_*>0$. Thus, $\varepsilon_*=0$ and $V(t+\tau, x, y) \geq V_1(t, x, y)$ in $\omega^{-}$. Similarly, we can also get $V(t+\tau, x, y) \geq V_1(t, x, y)$ in $\omega^{+}$. Combining \eqref{3.45}, one has $V(t+\tau, x, y) \geq V_1(t, x, y)$ in $\mathbb{R} \times \mathbb{R}^N$ for large $\tau$.

{\it Step 2:} Let
\begin{equation*}
\tau_*=\inf \left\{\tau \in \mathbb{R} ; V(t+\tau, x, y) \geq V_1(t, x, y) \text { in } \mathbb{R} \times \mathbb{R}^N\right\} .
\end{equation*}
Since $V, V_1$ satisfy \eqref{2.15} and $\underline{V}(t+\tau, x, y)>\underline{V}(t, x, y)$ for $\tau>0$, we can get $\tau_*$ is well-defined and $0 \leq \tau_*<+\infty$. Next we prove $\tau_*=0$.
Suppose $\tau_*>0$, then there might be two situations:
\begin{equation}\label{3.47}
\inf _\omega\left\{V\left(t+\tau_*, x, y\right)-V_1(t, x, y)\right\}>0,
\end{equation}
or
\begin{equation}\label{3.48}
\inf _\omega\left\{V\left(t+\tau_*, x, y\right)-V_1(t, x, y)\right\}=0.
\end{equation}

In the situation of \eqref{3.47}, there is a constant $\eta_0>0$ such that for any $\eta \in\left(0, \eta_0\right]$, $V\left(t+\tau_*-\eta, x, y\right)-V_1(t, x, y) \geq 0$ in $\omega$. By the similar argument in Step $1$, one has $V\left(t+\tau_*-\eta, x, y\right)-V_1(t, x, y) \geq 0$ in $\omega^{ \pm}$. But this contradicts the definition of $\tau_*$.

In the situation of \eqref{3.48}, there exists a sequence $\left\{\left(t_k, x_k, y_k\right)\right\}_{k \in \mathbb{N}} \in \omega$ such that
\begin{equation}\label{3.49}
\lim _{n \rightarrow+\infty}\left[V\left(t_k+\tau_*, x_k, y_k\right)-V_1\left(t_k, x_k, y_k\right)\right]=0 .
\end{equation}
If $d\left(\left(x_k, y_k\right), \mathcal{R}+\hat{c}t_k e_{0}\right) \rightarrow+\infty$, then $V\left(t_k+\tau_*, x_k, y_k\right)-\underline{V}\left(t_k+\tau_*, x_k, y_k\right) \rightarrow 0$ and $V_1\left(t_k, x_k, y_k\right)-\underline{V}\left(t_k, x_k, y_k\right) \rightarrow 0$. However, $\underline{V}\left(t_k+\tau, x_k, y_k\right)-\underline{V}\left(t_k, x_k, y_k\right)\geq\gamma(\tau)>0$, where $\gamma$ is a modulus function and $\tau>0$, this contradicts \eqref{3.49}. So it holds that
\begin{equation}\label{3.50}
\limsup _{k \rightarrow \infty} d\left(\left(x_k, y_k\right), \mathcal{R}+\hat{c}t_k e_{0}\right)<+\infty .
\end{equation}
Without loss of generality, assume that $\left(x_k, y_k\right) \in \mathcal{R}+\hat{c}t_k e_{0}$ and take a positive constant $r_{1}$ (to be determined below). Then we have
\begin{equation*}
\min _{1 \leq i \leq n}\left\{(x_k, y_k)\cdot e_{i} - \hat{c} t_k e_{0} \cdot e_{i}+r_1 e_{0} \cdot e_{i}\right\} \geq r_{1} \min _{1 \leq i \leq n} \{e_{0} \cdot e_{i}\}.
\end{equation*}
On the left-hand side of this inequality, suppose that the minimum is reached at $l$ and the minimum is $r_{2}$. Then $r_{2} \geq r_{1} \min _{1 \leq i \leq n} \{e_{0} \cdot e_{i}\}$. Denote $\left(x_{k}^{\prime}, y_k^{\prime}\right)=\left(x_{k}, y_{k}\right)-\hat{c} e_{0} +r_{1} e_{0} - r_{2} e_{l}$. Then
\begin{equation*}
\left(x_{k}^{\prime}, y_{k}^{\prime}\right) \cdot e_{l} -\hat{c} t_{k} e_{0} \cdot e_{l} + \hat{c} e_{0} \cdot e_{l} =0
\end{equation*}
and for any $j \neq l$,
\begin{equation*}
\left(x_{k}^{\prime}, y_{k}^{\prime}\right) \cdot e_{j} -\hat{c} t_{k} e_{0} \cdot e_{j} + \hat{c} e_{0} \cdot e_{j} \geq r_{2} (1 - e_{l} \cdot e_{j})\geq r_{1} \min _{1 \leq i \leq n} \{e_{0} \cdot e_{i}\} (1 - e_{l} \cdot e_{j}) >0.
\end{equation*}
Note that $0<e_{l} \cdot e_{j}<1$, since $e_l \neq e_j$ for $l \neq j$. Therefore, if $r_{1}$ is large enough, we can obtain that $\left(t_k-1, x_k^{\prime}, y_k^{\prime}\right) \in \omega$ and $d\left(\left(x_k^{\prime}, y_k^{\prime}\right), \mathcal{R}+\hat{c}(t_k-1) e_{0}\right)$ are very large. Similar to the proof of \eqref{3.50}, one has that $V\left(t_k-1+\tau_*, x_k^{\prime}, y_k^{\prime}\right)-V_1\left(t_k-1, x_k^{\prime}, y_k^{\prime}\right) \geq \gamma\left(\tau_*\right)>0$ for some modulus function $\gamma\left(\tau_*\right)$. On the other hand, it follows from the definition of $\left(x_k^{\prime}, y_k^{\prime}\right)$ that
\begin{equation*}
\limsup _{k \rightarrow+\infty}\left|\left(x_k, y_k\right)-\left(x_k^{\prime}, y_k^{\prime}\right)\right|<+\infty.
\end{equation*}
Then, by linear parabolic estimates, it holds that $V\left(t_k-1+\tau_*, x_k^{\prime}, y_k^{\prime}\right)-V_1\left(t_k-1, x_k^{\prime}, y_k^{\prime}\right) \rightarrow 0$, which is a contradiction. Thus, we obatin $\tau_*=0$, which means that $V(t, z) \geq V_1(t, z)$ for all $(t, z) \in \mathbb{R} \times \mathbb{R}^N$. 

We change positions of $V(t, z)$ and $V_1(t, z)$ and use the similar arguments again to obtain that $V(t, z) \leq V_1(t, z)$ for all $(t, x) \in \mathbb{R} \times \mathbb{R}^N$. Finally, one has $V_1(t, z) \equiv V(t, z)$, which completes the proof.
\begin{remark}\label{Remark 3.3}
According to Theorem \ref{Theorem 2.19}, we can obtain that
\begin{equation*}
	V(t, x, y)=V\left(t+L_N k / \hat{c}, x, y+L_N k\right) \text { in } \mathbb{R} \times \mathbb{R}^N, \;\forall k \in \mathbb{Z},
\end{equation*}
where $L_N$ given in the definition of $\mathbb{L}^N$ is the period of $y$.
\end{remark}
\begin{proof}
For any $k \in \mathbb{Z}$, it holds that
\begin{equation*}
	V_k(t, x, y):=V\left(t+L_N k / \hat{c}, x, y+L_N k\right)
\end{equation*}
is an entire solution of \eqref{1.1} and $0 \leq V_k \leq 1$, where $L_N$ is the period of $y$. By \eqref{3.4}, the values of $\xi$, $\eta$ and $x \cdot \nu_i \cos \theta_i+(y-\hat{c} t) \sin \theta_i$ at points $\left(t+L_N k / \hat{c}, x, y+L_N k\right)$ are invariant for each $k \in \mathbb{Z}$. Since $L_N$ is the period of $U_e(s, x, y)$ in $y$ for all $e \in \mathbb{S}^{N-1}$, we have
\begin{equation*}
	\bar{V}(t, x, y)=\bar{V}\left(t+L_{N} k / \hat{c}, x, y+L_{N} k\right), \quad \underline{V}(t, x, y)=\underline{V}\left(t+L_N k / \hat{c}, x, y+L_N k\right)
\end{equation*}
in $\mathbb{R} \times \mathbb{R}^N$ for all $k \in \mathbb{Z}$. Therefore, it follows from \eqref{3.6} that
\begin{equation*}
	\left|\bar{V}\left(t+L_{N} k / \hat{c}, x, y+L_{N} k\right)-\underline{V}\left(t+L_N k / \hat{c}, x, y+L_N k\right)\right| \leq 2\varepsilon, \; d\left((x, y), \mathcal{R}+\hat{c} t e_0\right) \rightarrow+\infty. 
\end{equation*}
By \eqref{3.44} and arbitrariness of $\varepsilon$, one has
\begin{equation*}
	\left|V_k(t, x, y)-\underline{V}(t, x, y)\right|=0,\; d\left((x, y), \mathcal{R}+\hat{c} t e_0\right) \rightarrow+\infty .
\end{equation*}
Thus, Theorem \ref{Theorem 2.19} implies that $V_k(t, x, y) \equiv V(t, x, y)$ in $\mathbb{R}\times\mathbb{R}^N$. This completes the proof.
\end{proof}

\section{Stability of curved fronts}
\noindent
In this section, we study the stability of curved fronts in Theorem \ref{Theorem 2.18} and still consider $e_0 = (0, 0, \ldots, 1)$ for convenience. Firstly, we construct super- and subsolutions for Cauchy problem \eqref{Cauchy problem} (ignore initial condition).
\begin{lemma}\label{Lemma 4.1}
 For any $\beta \in\left(0, \beta^*\right]$ and any $0<\varepsilon<\varepsilon_0^{+}(\beta)$, there exist positive constants $\lambda(\beta)$ and $\varrho(\beta, \lambda)$ such that for each $0<\alpha<\alpha_0^{+}(\beta, \varepsilon)$,
\begin{equation*}
	W_{\delta}^{+}(t, x, y):=\bar{V}(\tau, x, y)+\delta e^{-\lambda t} \times\left[U_{e_{i}}^\beta(\eta(\tau, x, y), x, y) \omega(\xi(\tau, x, y))+(1-\omega(\xi(\tau, x, y)))\right]
\end{equation*}
	is a supersolution of \eqref{1.1} for $t \geq 0$ and $(x, y) \in \mathbb{R}^N$, for all $\delta \in\left(0,  \gamma_{\star} / 4\right]$, where $\tau=\tau(t):=t-\varrho \delta e^{-\lambda t}+\varrho \delta$, and $\beta^*$, $\varepsilon_0^{+}(\beta)$, $\alpha_0^{+}(\beta, \varepsilon)$, $\bar{V}$ are given in Lemma \ref{Lemma 3.2}, and $\eta, \xi, \omega,  \gamma_{\star}$ are given in \eqref{1.5}, \eqref{3.3} and \eqref{3.4}. Note that $e_{i}$ here must be the same as $e_{i}$ in Lemma \ref{Lemma 3.2}.
\end{lemma}
\begin{proof}
Our approach is to find two numbers $X^{\prime}>1$ and $X^{\prime \prime}>1$ and consider the inequality
\begin{equation*}
\mathcal{L} W_{\delta}^{+}:=\partial_t	W_{\delta}^{+}-\Delta_{x, y} W_{\delta}^{+}-f\left(x, y, W_{\delta}^{+}\right) \geq 0, \quad \forall(t, x, y) \in[0,+\infty) \times \mathbb{R}^N
\end{equation*}
in three cases $\xi(\tau, x, y)>X^{\prime}, \xi(\tau, x, y)<-X^{\prime \prime}$, and $\xi(\tau, x, y) \in\left[-X^{\prime \prime}, X^{\prime}\right]$, respectively. Since $\bar{V}$ is a supersolution of \eqref{1.1} from Lemma \ref{Lemma 3.2}, we have
\begin{equation}\label{4.1}
\begin{aligned}
	\mathcal{L} W_{\delta}^{+} \geq & \varrho \delta \lambda e^{-\lambda t} \bar{V}_\tau +f\left(x, y, \bar{V}\right)-f\left(x, y, W_{\delta}^{+}\right) \\
	& +\left(\partial_t-\Delta_{x, y}\right)\left(\delta e^{-\lambda t} \times\left[U_{e_{i}}^\beta(\eta, x, y) \omega(\xi)+(1-\omega(\xi))\right]\right),
\end{aligned}
\end{equation}
in $\mathbb{R} \times \mathbb{R}^N$, where $\xi, \eta, \bar{V}$ and all of its derivatives are evaluated at $(\tau(t), x, y)$. For convenience, we define $P:=\left(\partial_t-\Delta_{x, y}\right)\left(\delta e^{-\lambda t} \times\left[U_{e_{i}}^\beta(\eta, x, y) \omega(\xi)+(1-\omega(\xi))\right]\right)$. 

\textbf{Case 1:} $\xi(\tau(t), x, y)>X^{\prime}$ and $t \geq 0$, where $X^{\prime}>1$ is to be chosen.
In this case, $\omega(\xi) \equiv 1$. Thus,
\begin{equation*}
	\mathcal{L} W_{\delta}^{+} \geq \varrho \delta \lambda e^{-\lambda t} \bar{V}_\tau+\left(\partial_t-\Delta_{x, y}\right)\left(\delta e^{-\lambda t} U_{e_{i}}^\beta(\eta)\right)+f\left(x, y, \bar{V}\right)-f\left(x, y, W_{\delta}^{+}\right).
\end{equation*}
Then one gets
	\begin{align}\label{4.2}		
		P=&-\lambda\delta e^{- \lambda t}U_{e_{i}}^\beta+\delta e^{- \lambda t}\beta U_{e_{i}}^{\beta-1} \partial_{\eta}U_{e_{i}}(-\hat{c})(1+\varrho \delta \lambda e^{- \lambda t}) \nonumber\\
		&-\delta e^{- \lambda t}\beta(\beta-1) U_{e_{i}}^{\beta-2}\left(\sum_{k=1}^{N-1}(\partial_{x_k}U_{e_{i}})^2 + (\partial_{y}U_{e_{i}})^2\right)\nonumber\\
		&-2\delta e^{- \lambda t}\beta(\beta-1) U_{e_{i}}^{\beta-2}\partial_{\eta}U_{e_{i}}\left(\sum_{k=1}^{N-1}\partial_{x_k}U_{e_{i}}\partial_{x_k}\eta + \partial_{y} U_{e_{i}}\partial_{y}\eta\right)\nonumber\\
		&-\delta e^{- \lambda t}\beta U_{e_{i}}^{\beta-1}\left(\sum_{k=1}^{N-1}\partial_{x_k x_k}U_{e_{i}} + \partial_{y y}U_{e_{i}}\right)-2\delta e^{- \lambda t}\beta U_{e_{i}}^{\beta-1}\left(\sum_{k=1}^{N-1}\partial_{x_k}\partial_{\eta}U_{e_{i}}\partial_{x_k}\eta + \partial_{y}\partial_{\eta} U_{e_{i}}\partial_{y}\eta\right)\nonumber\\
		&-\delta e^{- \lambda t}\beta(\beta-1) U_{e_{i}}^{\beta-2}(\partial_{\eta} U_{e_{i}})^2\left(\sum_{k=1}^{N-1}(\partial_{x_k}\eta)^2 + (\partial_{y}\eta)^2\right)\nonumber\\
		&-\delta e^{- \lambda t}\beta U_{e_{i}}^{\beta-1}\partial_{\eta \eta} U_{e_{i}}\left(\sum_{k=1}^{N-1}(\partial_{x_k}\eta)^2 + (\partial_{y}\eta)^2\right)-\delta e^{- \lambda t}\beta U_{e_{i}}^{\beta-1}\partial_{\eta} U_{e_{i}}\left(\sum_{k=1}^{N-1}\partial_{x_k x_k}\eta + \partial_{y y}\eta\right)\nonumber\\
		&\geq -\lambda\delta e^{- \lambda t}U_{e_{i}}^\beta +0 \\
		&\quad -\delta e^{- \lambda t} \beta U_{e_{i}}^{\beta}\left[\beta\frac{\sum_{k=1}^{N-1}\left(\partial_{x_k}U_{e_{i}}+\partial_{\eta}U_{e_{i}}\partial_{x_k}\eta\right)^2+\left(\partial_{y}U_{e_{i}}+\partial_{\eta}U_{e_{i}}\partial_{y}\eta\right)^2}{U_{e_{i}}^2}+\hat{c}\frac{\partial_{\eta}U_{e_{i}}}{U_{e_{i}}}\right.\nonumber\\
		&\quad -\frac{\sum_{k=1}^{N-1}\left(\partial_{x_k}U_{e_{i}}+\partial_{\eta}U_{e_{i}}\partial_{x_k}\eta\right)^2+\left(\partial_{y}U_{e_{i}}+\partial_{\eta}U_{e_{i}}\partial_{y}\eta\right)^2}{U_{e_{i}}^2}+ \frac{\Delta_{x, y}U_{e_{i}}+2 \nabla_{x, y}\partial_{\eta}U_{e_{i}}(-\nabla_{\zeta} \varphi(\alpha x), 1)}{U_{e_{i}}}\nonumber\\
		&\quad \left. +\frac{\partial_{\eta \eta}U_{e_{i}}}{U_{e_{i}}}\left(\sum_{k=1}^{N-1}(-\partial_{\zeta_k}\varphi(\alpha x))^2+1\right)-\frac{\partial_{\eta}U_{e_{i}}}{U_{e_{i}}}\left(\alpha \sum_{k=1}^{N-1}\partial_{\zeta_k \zeta_k}\varphi(\alpha x)\right)\right]\nonumber\\
		&\geq -\lambda\delta e^{- \lambda t}U_{e_{i}}^\beta +\delta e^{- \lambda t} \beta U_{e_{i}}^{\beta} \frac{\partial_{\eta}U_{e_{i}}}{U_{e_{i}}}\left(\alpha \sum_{k=1}^{N-1}\partial_{\zeta_k \zeta_k}\varphi(\alpha x)\right)\nonumber\\
		&\quad -\delta e^{- \lambda t} \beta U_{e_{i}}^{\beta}\left[\beta\frac{\sum_{k=1}^{N-1}\left(\partial_{x_k}U_{e_{i}}+\partial_{\eta}U_{e_{i}}\partial_{x_k}\eta\right)^2+\left(\partial_{y}U_{e_{i}}+\partial_{\eta}U_{e_{i}}\partial_{y}\eta\right)^2}{U_{e_{i}}^2}+\hat{c}\frac{\partial_{\eta}U_{e_{i}}}{U_{e_{i}}}\right.\nonumber\\
		&\quad -\frac{\sum_{k=1}^{N-1}\left(\partial_{x_k}U_{e_{i}}+\partial_{\eta}U_{e_{i}}\partial_{x_k}\eta\right)^2+\left(\partial_{y}U_{e_{i}}+\partial_{\eta}U_{e_{i}}\partial_{y}\eta\right)^2}{U_{e_{i}}^2}+ \frac{\Delta_{x, y}U_{e_{i}}+2 \nabla_{x, y}\partial_{\eta}U_{e_{i}}(-\nabla_{\zeta} \varphi(\alpha x), 1)}{U_{e_{i}}}\nonumber \\
		&\quad \left. +\frac{\partial_{\eta \eta}U_{e_{i}}}{U_{e_{i}}}\left(\sum_{k=1}^{N-1}(-\partial_{\zeta_k}\varphi(\alpha x))^2+1\right)\right] \nonumber\\
		&=: -\lambda\delta e^{- \lambda t}U_{e_{i}}^\beta + Q_1 +Q_2,\nonumber
	\end{align}
where $U_{e_i}$ and all of its derivatives are evaluated at $( \eta(t, x, y), x, y)$. From Lemma \ref{Lemma 3.2}, recall that
\begin{equation*}
	\beta^{*}\leq\frac{\hat{c}}{2 c_{e_i}\left((C_3+\max_{1\le i\le n} \{|\nu_i \cot \theta_i|\})^2+1\right)}.
\end{equation*}
By \eqref{3.16} and \eqref{3.17}, for any $\beta \in\left(0, \beta^{*}\right]$, there exists a sufficiently large number $X_1^{\prime}>1$ such that
\begin{equation}\label{4.3}
	\begin{aligned}
		Q_2&\geq-\delta e^{- \lambda t} \beta U_{e_{i}}^{\beta} \left[\frac{\hat{c}}{2 c_{e_i}\left((C_3+\max_{1\le i\le n} \{|\nu_i \cot \theta_i|\})^2+1\right)}\cdot c_{e_i}^2 \left(\sum_{k=1}^{N-1}(-\partial_{\zeta_k}\varphi(\alpha x))^2+1\right) -\hat{c} c_{e_i} \right]\\	
		&\geq-\delta e^{- \lambda t} \beta U_{e_{i}}^{\beta}\left(\frac{\hat{c} c_{e_i}}{2}-\hat{c} c_{e_i}\right)>\frac{\hat{c} c_{e_i}}{4}\delta e^{- \lambda t} \beta U_{e_{i}}^{\beta}
	\end{aligned}
\end{equation}
for all $(\eta, x, y) \in\left(X_1^{\prime},+\infty\right) \times \mathbb{R}^N$ and $t \geq 0$. It follows from \eqref{3.19}, \eqref{3.20}, \eqref{4.2} and \eqref{4.3} that for arbitrary $0<\alpha \leq \alpha_0^{+}(\beta, \varepsilon) \leq \alpha_1^{+}(\beta)$ ($\alpha_0^{+}(\beta, \varepsilon),\alpha_1^{+}(\beta)$ are given in Lemma \ref{Lemma 3.2}),
\begin{equation}\label{4.4}
Q_1+Q_2>\delta e^{- \lambda t} U_{e_{i}}^{\beta}\times \beta\frac{\hat{c}c_{e_i}}{8 }.
\end{equation}
Therefore, it follows from \eqref{4.1}-\eqref{4.4} that
\begin{equation*}
	\mathcal{L} W_{\delta}^{+} \geq \varrho \delta \lambda e^{-\lambda t} \bar{V}_\tau-\lambda\delta e^{- \lambda t}U_{e_{i}}^\beta+\delta e^{- \lambda t} U_{e_{i}}^{\beta}\times \beta\frac{\hat{c}c_{e_i}}{8 }+f\left(x, y, \bar{V}\right)-f\left(x, y, W_{\delta}^{+}\right).
\end{equation*}
By Theorem \ref{Theorem 2.8} and definitions of $\bar{V}$ and $W_{\delta}^{+}$, there exists a 
large enough constant 
$X_2^{\prime}>1$ such that 
\begin{equation*}
	f\left(x, y, \bar{V}\right)-f\left(x, y, W_{\delta}^{+}\right)= 0,\quad\forall(\xi, x, y) \in\left(X_2^{\prime},+\infty\right) \times \mathbb{R}^N \text{ and } t \geq 0.
\end{equation*}
Take $X^{\prime}=\max\{X_1^{\prime},X_2^{\prime}\}$, and $\beta^{*}$  and $\alpha_0^{+}(\beta, \varepsilon)$ are given in Lemma \ref{Lemma 3.2}. By \eqref{3.8}, we know
\begin{equation*}
	\mathcal{L} W_{\delta}^{+} \geq -\lambda\delta e^{- \lambda t}U_{e_{i}}^\beta+\delta e^{- \lambda t} U_{e_{i}}^{\beta}\times \beta\frac{\hat{c}c_{e_i}}{8 },
\end{equation*}
for any $(\xi, x, y) \in\left(X^{\prime},+\infty\right) \times \mathbb{R}^N$ and $t \geq 0$. Let $\lambda\in\left(0,\beta\frac{\hat{c}c_{e_i}}{8 }\right)$, then we prove $\mathcal{L} W_{\delta}^{+} >0$ in Case $1$.

\textbf{Case 2:} $\xi(\tau(t), x, y)<-X^{\prime\prime}$ and $t \geq 0$, where $X^{\prime\prime}>1$ is to be chosen.
In this case, $\omega(\xi) \equiv 0$. Then we have
\begin{equation*}
	\mathcal{L} W_{\delta}^{+} \geq \varrho \delta \lambda e^{-\lambda t} \bar{V}_\tau-\lambda\delta e^{- \lambda t} +f\left(x, y, \bar{V}\right)-f\left(x, y, W_{\delta}^{+}\right).
\end{equation*}
Recall that $\delta \in\left(0,  \gamma_{\star} / 4\right]$ and $\varepsilon<\gamma_{\star} / 3$, by definitions of $\bar{V}$ and $W_{\delta}^{+}$, there exists a sufficiently large constant 
$X^{\prime\prime}>1$ such that $W_{\delta}^{+}, \bar{V} \in [1-\gamma_{\star},1+\gamma_{\star}]$ when $\xi<-X^{\prime\prime}$ and $t \geq 0$. 
Thus one gets
\begin{equation*}
	f\left(x, y, \bar{V}\right)-f\left(x, y, W_{\delta}^{+}\right)>\frac{\kappa_1}{2} \delta  e^{-\lambda t},
\end{equation*}
then
\begin{equation*}
	\mathcal{L} W_{\delta}^{+} \geq -\lambda\delta e^{- \lambda t}+\frac{\kappa_1}{2} \delta  e^{-\lambda t},
\end{equation*}
for any $(\xi, x, y) \in\left(-\infty, -X^{\prime\prime}\right) \times \mathbb{R}^N$ and $t \geq 0$. Let $\lambda\in\left(0,\frac{\kappa_1}{2}\right)$, then we prove $\mathcal{L} W_{\delta}^{+} >0$ in Case $2$.

\textbf{Case 3:} $-X^{\prime\prime}<\xi(\tau(t), x, y)<X^{\prime}$ and $t \geq 0$.
	\begin{align*}
	P=&-\delta \lambda e^{- \lambda t}	-\delta  e^{- \lambda t}\beta U_{e_{i}}^{\beta-1} \partial_{\eta}U_{e_{i}}\hat{c}\omega-\delta e^{- \lambda t}U_{e_{i}}^{\beta}\omega^{\prime}\frac{\hat{c}}{\sqrt{1+|\nabla \varphi(\alpha x)|^2}}+	\delta  e^{- \lambda t} \omega^{\prime}\frac{\hat{c}}{\sqrt{1+|\nabla \varphi(\alpha x)|^2}}\\
	&-\delta e^{- \lambda t}\beta U_{e_{i}}^{\beta-1} \partial_{\eta}U_{e_{i}}\hat{c}\omega\varrho\delta\lambda e^{- \lambda t} -\delta e^{- \lambda t}U_{e_{i}}^{\beta}\omega^{\prime}\frac{\hat{c}}{\sqrt{1+|\nabla \varphi(\alpha x)|^2}}\varrho\delta\lambda e^{- \lambda t}\\
	&+\delta  e^{- \lambda t} \omega^{\prime}\frac{\hat{c}}{\sqrt{1+|\nabla \varphi(\alpha x)|^2}}\varrho\delta\lambda e^{- \lambda t}-\delta  e^{- \lambda t}\beta(\beta-1) U_{e_{i}}^{\beta-2}\omega\left(\sum_{k=1}^{N-1}(\partial_{x_k}U_{e_{i}})^2 + (\partial_{y}U_{e_{i}})^2\right)\\
	&-2\delta e^{- \lambda t}\beta(\beta-1) U_{e_{i}}^{\beta-2}\partial_{\eta}U_{e_{i}}\omega\left(\sum_{k=1}^{N-1}\partial_{x_k}U_{e_{i}}\partial_{x_k}\eta + \partial_{y} U_{e_{i}}\partial_{y}\eta\right)\\
	&-\delta e^{- \lambda t}\beta U_{e_{i}}^{\beta-1}\omega\left(\sum_{k=1}^{N-1}\partial_{x_k x_k}U_{e_{i}} + \partial_{y y}U_{e_{i}}\right)-2\delta e^{- \lambda t}\beta U_{e_{i}}^{\beta-1}\omega\left(\sum_{k=1}^{N-1}\partial_{x_k}\partial_{\eta}U_{e_{i}}\partial_{x_k}\eta + \partial_{y}\partial_{\eta} U_{e_{i}}\partial_{y}\eta\right)\\
	&-2\delta e^{- \lambda t}\beta U_{e_{i}}^{\beta-1} \omega^{\prime}\left(\sum_{k=1}^{N-1}\partial_{x_k}U_{e_{i}}\partial_{x_k}\xi+ \partial_{y} U_{e_{i}}\partial_{y}\xi\right)\\
	& -\delta  e^{- \lambda t}\beta(\beta-1) U_{e_{i}}^{\beta-2}(\partial_{\eta}U_{e_{i}})^2 \omega\left(\sum_{k=1}^{N-1}(\partial_{x_k}\eta)^2 + (\partial_{y}\eta)^2\right) \\
	&-\delta  e^{- \lambda t}\beta U_{e_{i}}^{\beta-1}\partial_{\eta \eta}U_{e_{i}} \omega\left(\sum_{k=1}^{N-1}(\partial_{x_k}\eta)^2 + (\partial_{y}\eta)^2\right)-\delta  e^{- \lambda t}\beta U_{e_{i}}^{\beta-1}\partial_{\eta}U_{e_{i}} \omega\left(\sum_{k=1}^{N-1}\partial_{x_k x_k}\eta + \partial_{y y}\eta\right)\\
	&-2\delta e^{- \lambda t}\beta U_{e_{i}}^{\beta-1} \partial_{\eta}U_{e_{i}} \omega^{\prime} \left(\sum_{k=1}^{N-1}\partial_{x_k}\eta \partial_{x_k}\xi + \partial_{y}\eta \partial_{y}\xi\right)- \delta e^{- \lambda t} U_{e_{i}}^{\beta} \omega^{\prime \prime} \left[\sum_{k=1}^{N-1}(\partial_{x_k}\xi)^2+(\partial_{y}\xi)^2\right]\\
	&-\delta e^{- \lambda t} U_{e_{i}}^{\beta} \omega^{\prime} \left[\sum_{k=1}^{N-1}\partial_{x_k x_k}\xi+\partial_{y y}\xi\right]+\delta e^{- \lambda t} \omega^{\prime \prime} \left[\sum_{k=1}^{N-1}(\partial_{x_k}\xi)^2+(\partial_{y}\xi)^2\right]\\
	&+\delta e^{- \lambda t}\omega^{\prime} \left[\sum_{k=1}^{N-1}\partial_{x_k x_k}\xi+\partial_{y y}\xi\right].
	\end{align*}
By calculations, one has
	\begin{align*}
	&-\delta  e^{- \lambda t}\beta U_{e_{i}}^{\beta-1} \partial_{\eta}U_{e_{i}}\hat{c}\omega\geq 0,\\ \\
	&-\delta e^{- \lambda t}U_{e_{i}}^{\beta}\omega^{\prime}\frac{\hat{c}}{\sqrt{1+|\nabla \varphi(\alpha x)|^2}}\geq -\delta e^{- \lambda t}U_{e_{i}}^{\beta}|\omega^{\prime}|\hat{c},\\ \\
	&\delta  e^{- \lambda t} \omega^{\prime}\frac{\hat{c}}{\sqrt{1+|\nabla \varphi(\alpha x)|^2}}\geq 0,\\ \\
	&-\delta e^{- \lambda t}\beta U_{e_{i}}^{\beta-1} \partial_{\eta}U_{e_{i}}\hat{c}\omega\varrho\delta\lambda e^{- \lambda t}\geq -(\delta e^{- \lambda t})^2 \omega\varrho\delta\lambda\beta \hat{c} U_{e_{i}}^{\beta-1}\partial_{\eta}U_{e_{i}}\geq 0,\\ \\
	&-\delta e^{- \lambda t}U_{e_{i}}^{\beta}\omega^{\prime}\frac{\hat{c}}{\sqrt{1+|\nabla \varphi(\alpha x)|^2}}\varrho\delta\lambda e^{- \lambda t}\geq -(\delta e^{- \lambda t})^2 \varrho\lambda \hat{c}\omega^{\prime} U_{e_{i}}^{\beta},\\ \\
	&\delta  e^{- \lambda t} \omega^{\prime}\frac{\hat{c}}{\sqrt{1+|\nabla \varphi(\alpha x)|^2}}\varrho\delta\lambda e^{- \lambda t}\geq (\delta  e^{- \lambda t})^2 \varrho\lambda \omega^{\prime} \hat{c}\geq 0,\\ \\
	&-\delta  e^{- \lambda t}\beta(\beta-1) U_{e_{i}}^{\beta-2}\omega\left(\sum_{k=1}^{N-1}(\partial_{x_k}U_{e_{i}})^2 + (\partial_{y}U_{e_{i}})^2\right)\\
	&\quad\geq -\delta  e^{- \lambda t}|\beta(\beta-1)| U_{e_{i}}^{\beta-2} |\nabla_{x, y}U_{e_{i}} |^2\geq -\delta  e^{- \lambda t}\beta N\left(\bar{K} \max\{e^{-\frac{3 \kappa}{4} |\eta|}, e^{-\kappa_2 |\eta|}\}\right)^2,\\ \\
	&-2\delta e^{- \lambda t}\beta(\beta-1) U_{e_{i}}^{\beta-2}\partial_{\eta}U_{e_{i}}\omega\left(\sum_{k=1}^{N-1}\partial_{x_k}U_{e_{i}}\partial_{x_k}\eta + \partial_{y} U_{e_{i}}\partial_{y}\eta\right)\\
	&\quad\geq-2\delta e^{- \lambda t}\beta U_{e_{i}}^{\beta-2}|\partial_{\eta}U_{e_{i}}|\sqrt{N\left(\bar{K} \max\{e^{-\frac{3 \kappa}{4} |\eta|}, e^{-\kappa_2 |\eta|}\}\right)^2}\left[1+(C_3 h(\alpha x)+\max_{1\le i\le n} \{|\nu_i \cot \theta_i|\})^2\right]^\frac{1}{2},\\ \\
	&-\delta e^{- \lambda t}\beta U_{e_{i}}^{\beta-1}\omega\left(\sum_{k=1}^{N-1}\partial_{x_k x_k}U_{e_{i}} + \partial_{y y}U_{e_{i}}\right)\\
	&\quad\geq -\delta e^{- \lambda t}\beta U_{e_{i}}^{\beta-1}\omega|\Delta_{x, y} U_{e_{i}}|\geq -\delta e^{- \lambda t}\beta\omega N\bar{K} \max\{e^{-\frac{3 \kappa}{4} |\eta|}, e^{-\kappa_2 |\eta|}\},\\ \\
	&-2\delta e^{- \lambda t}\beta U_{e_{i}}^{\beta-1}\omega\left(\sum_{k=1}^{N-1}\partial_{x_k}\partial_{\eta}U_{e_{i}}\partial_{x_k}\eta + \partial_{y}\partial_{\eta} U_{e_{i}}\partial_{y}\eta\right)\\
	&\quad\geq -2\delta e^{- \lambda t}\beta U_{e_{i}}^{\beta-1}\omega |\nabla_{x, y}\partial_{\eta}U_{e_{i}}|\cdot|\nabla_{x, y} \eta|\\
	&\quad\geq-2\delta e^{- \lambda t}\beta \left(N\left(\bar{K} \max\{e^{-\frac{3 \kappa}{4} |\eta|}, e^{-\kappa_2 |\eta|}\}\right)^2\right)^\frac{1}{2}\left[1+(C_3 h(\alpha x)+\max_{1\le i\le n} \{|\nu_i \cot \theta_i|\})^2\right]^\frac{1}{2},\\ \\
	&-2\delta e^{- \lambda t}\beta U_{e_{i}}^{\beta-1} \omega^{\prime}\left(\sum_{k=1}^{N-1}\partial_{x_k}U_{e_{i}}\partial_{x_k}\xi+ \partial_{y} U_{e_{i}}\partial_{y}\xi\right)\\
	&\quad\geq-2\delta e^{- \lambda t}\beta U_{e_{i}}^{\beta-1} \omega^{\prime}|\nabla_{x, y}U_{e_{i}}|\cdot|\nabla_{x, y} \xi|\\
	&\quad\geq-\delta e^{- \lambda t}\beta U_{e_{i}}^{\beta-1} \omega^{\prime} \left(N\left(\bar{K} \max\{e^{-\frac{3 \kappa}{4} |\eta|}, e^{-\kappa_2 |\eta|}\}\right)^2\right)^\frac{1}{2}(\alpha C_4(1+|\xi|) h(\alpha x)+\alpha C_4|\xi| h(\alpha x)),\\ \\
	&-\delta  e^{- \lambda t}\beta(\beta-1) U_{e_{i}}^{\beta-2}(\partial_{\eta}U_{e_{i}})^2 \omega\left(\sum_{k=1}^{N-1}(\partial_{x_k}\eta)^2 + (\partial_{y}\eta)^2\right)\\
	&\quad\geq-\delta  e^{- \lambda t}\beta(\beta-1) U_{e_{i}}^{\beta-2}(\partial_{\eta}U_{e_{i}})^2 \omega |\nabla_{x, y}\eta|^2\\
	&\quad\geq-\delta  e^{- \lambda t}\beta (\partial_{\eta}U_{e_{i}})^2 \left(\bar{K} \max\{e^{-\frac{3 \kappa}{4} |\eta|}, e^{-\kappa_2 |\eta|}\}\right)^2 \left[1+(C_3 h(\alpha x)+\max_{1\le i\le n} \{|\nu_i \cot \theta_i|\})^2\right],\\ \\
	&-\delta  e^{- \lambda t}\beta U_{e_{i}}^{\beta-1}\partial_{\eta \eta}U_{e_{i}} \omega\left(\sum_{k=1}^{N-1}(\partial_{x_k}\eta)^2 + (\partial_{y}\eta)^2\right)\\
	&\quad\geq-\delta  e^{- \lambda t}\beta \bar{K} \max\{e^{-\frac{3 \kappa}{4} |\eta|}, e^{-\kappa_2 |\eta|}\}\left[1+(C_3 h(\alpha x)+\max_{1\le i\le n} \{|\nu_i \cot \theta_i|\})^2\right],\\ \\
	&-\delta  e^{- \lambda t}\beta U_{e_{i}}^{\beta-1}\partial_{\eta}U_{e_{i}} \omega\left(\sum_{k=1}^{N-1}\partial_{x_k x_k}\eta + \partial_{y y}\eta\right)\geq-\delta  e^{- \lambda t}\beta \bar{K} \max\{e^{-\frac{3 \kappa}{4} |\eta|}, e^{-\kappa_2 |\eta|}\}(N-1)\alpha C_{3} h(\alpha x),\\ \\
	&-2\delta e^{- \lambda t}\beta U_{e_{i}}^{\beta-1} \partial_{\eta}U_{e_{i}} \omega^{\prime} \left(\sum_{k=1}^{N-1}\partial_{x_k}\eta \partial_{x_k}\xi + \partial_{y}\eta \partial_{y}\xi\right)\\
	&\quad\geq -\delta e^{- \lambda t}\beta \omega^{\prime}  \bar{K} \max\{e^{-\frac{3 \kappa}{4} |\eta|}, e^{-\kappa_2 |\eta|}\}\times\\
	& \qquad\qquad\qquad\qquad\left(1+(C_3 h(\alpha x)+\max_{1\le i\le n} \{|\nu_i \cot \theta_i|\})^2\right)^\frac{1}{2} (\alpha C_4(1+|\xi|) h(\alpha x)+\alpha C_4|\xi| h(\alpha x)),\\ \\
	&- \delta e^{- \lambda t} U_{e_{i}}^{\beta} \omega^{\prime \prime} \left[\sum_{k=1}^{N-1}(\partial_{x_k}\xi)^2+(\partial_{y}\xi)^2\right]\geq - \delta e^{- \lambda t} |\omega^{\prime \prime}| \left[\frac{1}{2}(\alpha C_4(1+|\xi|) h(\alpha x)+\alpha C_4|\xi| h(\alpha x))\right]^2,\\ \\
	&- \delta e^{- \lambda t} U_{e_{i}}^{\beta} \omega^{\prime} \left[\sum_{k=1}^{N-1}\partial_{x_k x_k}\xi+\partial_{y y}\xi\right]\geq - \delta e^{- \lambda t} \omega^{\prime}\alpha C_4(1+\alpha|\xi|) h(\alpha x),\\ \\
	&\delta e^{- \lambda t} \omega^{\prime \prime} \left[\sum_{k=1}^{N-1}(\partial_{x_k}\xi)^2+(\partial_{y}\xi)^2\right]\geq \delta e^{- \lambda t} |\omega^{\prime \prime}| \left[\frac{1}{2}(\alpha C_4(1+|\xi|) h(\alpha x)+\alpha C_4|\xi| h(\alpha x))\right]^2,\\ 
	&\text{and}\\
	&\delta e^{- \lambda t}\omega^{\prime} \left[\sum_{k=1}^{N-1}\partial_{x_k x_k}\xi+\partial_{y y}\xi\right]\geq\delta e^{- \lambda t}\omega^{\prime} \alpha C_4(1+\alpha|\xi|) h(\alpha x). 	
	\end{align*}
Then there exists a constant $C^{*}(\beta)>0$ such that 
\begin{equation*}
	\left(\partial_t-\Delta_{x, y}\right)\left(\delta e^{-\lambda t} \times\left[U_{e_{i}}^\beta(\eta, x, y) \omega(\xi)+(1-\omega(\xi))\right]\right)\geq -\delta \lambda e^{- \lambda t}-\delta e^{- \lambda t} C^{*}(\beta),
\end{equation*}
thus we have
\begin{equation}\label{4.5}
	\mathcal{L} W_{\delta}^{+} \geq \varrho \delta \lambda e^{-\lambda t} \bar{V}_\tau-\delta\lambda e^{- \lambda t}-\delta e^{- \lambda t} C^{*}(\beta) +f\left(x, y, \bar{V}\right)-f\left(x, y, W_{\delta}^{+}\right).
\end{equation}
It follows from \eqref{3.5}, \eqref{3.10} and \eqref{3.14} that
\begin{equation}\label{4.6}
\begin{aligned}
	\frac{\partial}{\partial \tau}\left(\bar{V}(\tau, x, y)\right) & =\partial_{\xi} U_{e(x)}(\xi, x, y) \xi_\tau+\varepsilon h(\alpha x) \frac{\partial}{\partial \tau}\left[\left(U_{e_{i}}^\beta(\eta, x, y)-1\right) \omega(\xi)\right] \\
	& \geq \partial_{\xi} U_{e(x)}(\xi, x, y) \frac{-\hat{c}}{\sqrt{1+|\nabla \varphi|^2}}
\end{aligned}
\end{equation}
in $\mathbb{R}\times\mathbb{R}^N$, where $\xi$ and $\eta$ are evaluated at $(\tau, x, y)$. Thus, by \eqref{4.6}, Theorem \ref{Theorem 2.12} and Lemma \ref{Lemma 2.17}, there is a constant $r^{*}>0$ such that
\begin{equation}\label{4.7}
\bar{V}_{\tau}(\tau, x, y)>r^{*} \hat{c} \text { in }\left\{(t, x, y): \xi(\tau(t), x, y) \;\in\left[-X^{\prime \prime}, X^{\prime}\right]\right\} .
\end{equation}
By \eqref{4.5} and \eqref{4.7}, we get
\begin{equation*}
  \begin{aligned}
	\mathcal{L} W_{\delta}^{+} &\geq \varrho \delta \lambda e^{-\lambda t} r^{*} \hat{c} -\delta\lambda e^{- \lambda t}-\delta e^{- \lambda t} C^{*}(\beta) +f\left(x, y, \bar{V}\right)-f\left(x, y, W_{\delta}^{+}\right)\\
	&\geq \varrho \delta \lambda e^{-\lambda t} r^{*} \hat{c} -\delta\lambda e^{- \lambda t}-\delta e^{- \lambda t} C^{*}(\beta) -\left\|f_u\right\|_{L^{\infty}}\delta e^{- \lambda t}\\
	&\geq \delta e^{- \lambda t}\left(\varrho\lambda e^{-\lambda t} r^{*} \hat{c}- \lambda - C^{*}(\beta)-\left\|f_u\right\|_{L^{\infty}} \right)
  \end{aligned}
\end{equation*}	
 in $\left\{(t, x, y): \xi(\tau(t), x, y) \in\left[-X^{\prime \prime}, X^{\prime}\right], t \geq 0\right\}$. Set 
 \begin{equation*}
 	 \varrho>\frac{\left\|f_u\right\|_{L^{\infty}}+\lambda+C^{*}(\beta)}{\lambda r^{*} \hat{c}},
 \end{equation*}
 then we show that $\mathcal{L} W_{\delta}^{+}>0$ in Case $3$. The proof of Lemma \ref{Lemma 4.1} is complete. 
\end{proof}

\begin{lemma}\label{Lemma 4.2}
Suppose that $V$ is the solution defined in Theorem \ref{Theorem 2.18}. Then for each $\beta \in\left(0, \beta^*\right]$, there exists a positive constant $\tilde{\alpha}_0^{+}(\beta)$ such that, for any $0<\alpha<\tilde{\alpha}_0^{+}(\beta)$ there exist positive constants $\lambda(\beta)$ and $\tilde{\varrho}(\beta, \lambda, \alpha)$ such that
\begin{equation*}
V_{\delta}^{+}(t, x, y ; T):=V(T+\tilde{\tau}, x, y)+\delta e^{-\lambda t} \times\left[U_{e_{i}}^\beta(\eta, x, y) \omega(\xi)+(1-\omega(\xi))\right]
\end{equation*}	
is a supersolution of \eqref{1.1} for $t \geq 0$ and $(x, y) \in \mathbb{R}^N$, for all $T \in \mathbb{R}$ and $\delta \in$ $\left(0,  \gamma_{\star} / 4 \right]$, where $\tilde{\tau}=\tilde{\tau}(t):=t-\tilde{\varrho}\delta e^{-\lambda t}+\tilde{\varrho} \delta$, and $\xi$, $\eta$ are evaluated at $(T+\tilde{\tau}, x, y)$, and $\beta^*$, $\lambda(\beta)$, $\eta$, $\omega$, $\gamma_{\star}$ are the same as those in Lemma \ref{Lemma 4.1}.

Besides, for any $\beta \in\left(0, \beta^*\right]$, there exists a positive constant $\hat{\alpha}_0^{+}(\delta)$ such that, for any $0<\alpha<\hat{\alpha}_0^{+}(\delta)$ there exist positive constants $\hat{\lambda}(\beta), \hat{\varrho}(\beta, \hat{\lambda}, \alpha)$ and $\delta^0(\beta, \hat{\lambda}, \hat{\varrho}, \alpha)$ such that
\begin{equation*}
V_{\delta}^{-}(t, x, y ; T):=V(T+\hat{\tau}, x, y)-\delta e^{-\hat{\lambda} t} \times\left[U_{e_{i}}^\beta(\eta, x, y) \omega(\xi)+(1-\omega(\xi))\right]
\end{equation*}	
is a subsolution of \eqref{1.1} for $t \geq 0$ and $(x, y) \in \mathbb{R}^N$, for all $T \in \mathbb{R}$ and $\delta \in\left(0, \delta^0\right]$, where $\hat{\tau}=\hat{\tau}(t):=t+\hat{\varrho} \delta e^{-\hat{\lambda} t}-\hat{\varrho} \delta$, and $\xi$, $\eta$ are evaluated at $(T+\hat{\tau}, x, y)$, and $\beta^*$, $\eta$, $\omega$, $\gamma_{\star}$ are the same as those in Lemma \ref{Lemma 4.1}.
\end{lemma}
\begin{proof}
{\it Step 1: we prove that $V_t \geq r$ in $\{(t, x, y):|\eta| \leq q\}$, where $r=r(\alpha, q)>0$ is a constant.}
Suppose by contradiction that there exists a sequence of points $\left\{\left(t_n, x_n, y_n\right)\right\}_{n \in \mathbb{N}}$ satisfying
\begin{equation}\label{4.8}
V_t\left(t_n, x_n, y_n\right) \rightarrow 0 \text { as } n \rightarrow \infty \text { and }\left|\eta\left(t_n, x_n, y_n\right)\right| \leq q \text { for all } n.
\end{equation}	
It follows from \eqref{4.7} that there exists a positive number $r_1=r_1(q)$ independent of $\alpha$ and $\varepsilon$ such that
\begin{equation}\label{4.9}
\partial_t \bar{V}>r_1 \text { in }\left\{(t, x, y):|\eta| \leq q+\hat{c}\right\}.
\end{equation}	
Set $\varepsilon=r_1 / 16$ and fix arbitrary $0<\alpha<\alpha_0^{+}(\beta, \varepsilon)=: \bar{\alpha}_0^{+}(\beta, q)$, where $\alpha_0^{+}(\beta, \varepsilon)$ is given in Lemma \ref{Lemma 3.2}. By virtue of \eqref{3.6}, we know that there exists a constant $\iota=\iota(\alpha)>0$ such that
\begin{equation}\label{4.10}
\left|\bar{V}(t, x, y)-V(t, x, y)\right| \leq \frac{r_1}{4}
\end{equation}	
for $\max_{1\leq i \leq N-1}|x_i| \geq \iota$ and $d\left((x,y), 
\mathcal{R}+\hat{c} t e_0\right) \geq \iota$. Without loss of generality, we suppose that $\{x_n\}_{n \in \mathbb{N}}$ satisfy $x_{n_{1}}\leq 0$ for each $n \in \mathbb{N}$, where $x_n=(x_{n_{1}}, x_{n_{2}}, \ldots, x_{n_{N-1}})$. Define
\begin{equation*}
\bar{y}_n:=y_n+\frac{\varphi\left(\alpha\left(x_n-\ell\right)\right)}{\alpha}-\frac{\varphi\left(\alpha
 x_n\right)}{\alpha} ,
\end{equation*}	
where $\ell:=(\ell_1,\ell_2,\ldots,\ell_{N-1})$, $\ell_k\geq 0$ $(k\in 
\{1,2,\ldots,N-1\})$ and  $\ell_1=\max_{1\leq k \leq N-1}\ell_{k}=\iota$. Then
\begin{equation*}
	\left|\bar{y}_n-y_n\right| \leq 
	\left|\frac{\varphi\left(\alpha\left(x_n-\ell\right)\right)}{\alpha}-\frac{\varphi\left(\alpha 
	x_n\right)}{\alpha}\right| \leq \iota (N-1) |\nabla \varphi | \leq \iota (N-1) \left(C_3+\max_{1\le 
	i\le n} \{|\nu_i \cot \theta_i|\}\right).
\end{equation*}	
It is easy to check that $\eta\left(t_n, x_n-\ell, \bar{y}_n\right) =\eta\left(t_n, x_n, y_n\right)$, 
which implies
\begin{equation}\label{4.11}
\left|\eta\left(t_n-\tau, x_n-\ell, \bar{y}_n\right)\right| \leq q+\hat{c} \; \text{ for all } \tau \in[0,1] 
\text{ and } n \in \mathbb{N}. 
\end{equation}	
Since $V$ is a solution of \eqref{1.1}, it can solve the equation 
\begin{equation}\label{4.12}
\left(\partial_t-\Delta_{x, y}\right) V_t-f_u(x, y, V) V_t=0
\end{equation}	
in $\mathbb{R} \times \mathbb{R}^N$, where $f_u(x, y, V)$ is bounded by \eqref{1.2}. By $V_{t}(t,x,y)>0$, \eqref{4.8} and \eqref{4.12}, it follows from parabolic estimates that
\begin{equation*}
V_t\left(t_n-\tau, x_n-\ell, \bar{y}_n\right) \rightarrow 0 \text { as } n \rightarrow \infty
\end{equation*}	
uniformly for $\tau \in[0,1]$. Then one has
\begin{equation*}
V\left(t_n, x_n-\ell, \bar{y}_n\right)-V\left(t_n-1, x_n-\ell, \bar{y}_n\right) \rightarrow 0 \text { as 
} n \rightarrow \infty.
\end{equation*}	
By virtue of \eqref{4.9}-\eqref{4.11}, we know
\begin{equation*}
	\begin{aligned}
    V\left(t_n, x_n-\ell, \bar{y}_n\right)\geq \bar{V}\left(t_n, x_n-\ell, 
    \bar{y}_n\right)-\frac{r_{1}}{4} \; \text{and} \;
    -V\left(t_n-1, x_n-\ell, \bar{y}_n\right)\geq -\bar{V}\left(t_n-1, x_n-\ell, 
    \bar{y}_n\right)-\frac{r_{1}}{4},
	\end{aligned}
\end{equation*}
and then it holds that
\begin{equation*}
	\begin{aligned}
		&V\left(t_n, x_n-\ell, \bar{y}_n\right)
		-V\left(t_n-1, x_n-\ell, \bar{y}_n\right)\\
		&\geq\bar{V}\left(t_n, x_n-\ell, \bar{y}_n\right)-\frac{r_{1}}{4}-\bar{V}\left(t_n-1, 
		x_n-\ell, \bar{y}_n\right)-\frac{r_{1}}{4}>r_{1}-\frac{r_{1}}{2}=\frac{r_{1}}{2}, \quad 
		\forall n \in \mathbb{N}.
	\end{aligned}
\end{equation*}
So we get a contradiction.

{\it Step 2: we prove that $V_{\delta}^{+}(t, x, y ; T)$ is a supersolution.}
The approach is to find two numbers $X^{\prime}>1$ and $X^{\prime \prime}>1$ and prove the inequality
	\begin{equation*}
		\mathcal{L} V_{\delta}^{+}:=\partial_t	V_{\delta}^{+}-\Delta_{x, y} V_{\delta}^{+}-f\left(x, y, V_{\delta}^{+}\right) \geq 0, \quad \forall(t, x, y) \in[0,+\infty) \times \mathbb{R}^N
	\end{equation*}
in three cases $\xi(\tilde{\tau}(t)+T, x, y)>X^{\prime}, \xi(\tilde{\tau}(t)+T, x, y)<-X^{\prime \prime}$, and $\xi(\tilde{\tau}(t)+T, x, y) \in\left[-X^{\prime \prime}, X^{\prime}\right]$, respectively. Since $V$ is a solution of \eqref{1.1}, it holds
	\begin{equation*}
		\begin{aligned}
			\mathcal{L} V_{\delta}^{+} \geq & \tilde{\varrho} \delta \lambda e^{-\lambda t} V_{\tilde{\tau}} +f\left(x, y, V\right)-f\left(x, y, V_{\delta}^{+}\right) \\
			& +\left(\partial_t-\Delta_{x, y}\right)\left(\delta e^{-\lambda t} \times\left[U_{e_{i}}^\beta(\eta, x, y) \omega(\xi)+(1-\omega(\xi))\right]\right)
		\end{aligned}
	\end{equation*}
in $\mathbb{R} \times \mathbb{R}^N$, where $\xi, \eta, V$ and all of its derivatives are evaluated at $(\tilde{\tau}(t)+T, x, y)$. 

\textbf{Case 1:} $\xi(\tilde{\tau}(t)+T, x, y)>X^{\prime}$ and $t \geq 0$, where $X^{\prime}>1$ is to be chosen.
In this case, $\omega(\xi) \equiv 1$. Thus,
\begin{equation*}
	\mathcal{L} V_{\delta}^{+} \geq \tilde{\varrho}\delta \lambda e^{-\lambda t} V_{\tilde{\tau}}+\left(\partial_t-\Delta_{x, y}\right)\left(\delta e^{-\lambda t} U_{e_{i}}^\beta(\eta)\right)+f\left(x, y, V\right)-f\left(x, y, V_{\delta}^{+}\right).
\end{equation*}
With similar arguments in Case $1$ of the proof in Lemma \ref{Lemma 4.1}, we can find a constant $X^{\prime}$ such that $V, V_{\delta}^{+}<\theta$. Let $\lambda\in\left(0,\beta\frac{\hat{c}c_{e_i}}{8 }\right)$, and $\beta^{*}$ and $\alpha_0^{+}(\beta, \varepsilon)$ are given in Lemma \ref{Lemma 3.2}, then we prove $\mathcal{L} V_{\delta}^{+} >0$ in Case $1$.
	
\textbf{Case 2:} $\xi(\tilde{\tau}(t)+T, x, y)<-X^{\prime\prime}$ and $t \geq 0$, where $X^{\prime\prime}>1$ is to be chosen.
In this case, $\omega(\xi) \equiv 0$. Then we have
\begin{equation*}
	\mathcal{L} V_{\delta}^{+} \geq \tilde{\varrho}\delta \lambda e^{-\lambda t} V_{\tilde{\tau}}-\lambda\delta e^{- \lambda t} +f\left(x, y, V\right)-f\left(x, y, V_{\delta}^{+}\right).
\end{equation*}
Recall that $\delta \in\left(0,  \gamma_{\star} / 4\right]$ and  $\alpha<\alpha_0^{+}\left(\beta, \varepsilon\right)$  in  $\bar{V}$  where  $\varepsilon<\gamma_{\star} / 3$, by definitions of $V$ and $V_{\delta}^{+}$, there exists a sufficiently large number 
$X^{\prime\prime}>1$ such that $V_{\delta}^{+}, V\in [1-\gamma_{\star},1+\gamma_{\star}]$ when $\xi<-X^{\prime\prime}$ and $t \geq 0$. 

Since $-\psi(x)=-\max _{1 \leq i \leq n} \{\psi_i(x)\}=\min_{1 \leq i \leq n}\{x \cdot \nu_i \cot \theta_i\}$, Theorem \ref{Theorem 2.16} and the definition of $\widehat{Q}_i$, we can get $\varphi(x)\geq\psi(x)$ for $x\in\mathbb{R}^{N-1}$ and 
\begin{equation*}
	\begin{aligned}
		-X^{\prime\prime}>\xi(\tilde{\tau}(t)+T, x, y)&\geq y-\hat{c} (\tilde{\tau}(t)+T)-\varphi(\alpha x) / \alpha \\
		&\geq y-\hat{c}(\tilde{\tau}(t)+T) -\psi(x)-\widetilde{C}= y-\hat{c} (\tilde{\tau}(t)+T) +\min_{1 \leq i \leq n}\{x \cdot \nu_i \cot \theta_i\} -\widetilde{C},
	\end{aligned}
\end{equation*}
where $\widetilde{C}=\frac{\sup_{\mathbb{R}^{N-1}}|\varphi(\alpha x)-\psi(\alpha x)|}{\alpha}$. Then it holds
\begin{equation*}
	-X^{\prime\prime}+\widetilde{C}> y-\hat{c} (\tilde{\tau}(t)+T) +\min_{1 \leq i \leq n}\{x \cdot \nu_i \cot \theta_i\}.
\end{equation*}
Assume that $y-\hat{c} (\tilde{\tau}(t)+T) +x \cdot \nu_j \cot \theta_j=\min_{1 \leq i \leq n}\{y-\hat{c} (\tilde{\tau}(t)+T) +x \cdot \nu_i \cot \theta_i\}$, for some constant $1 \leq j\leq n$. Let $X^{\prime\prime}\geq X^{*} + \widetilde{C}$ and $X^{*}>0$, then we obtain $-X^{\prime\prime}+ \widetilde{C} \leq -X^{*}<0 $ and
\begin{align*}
	&\min _{1 \leq i \leq n}\left\{x \cdot \nu_i \cos \theta_i+(y-\hat{c} (\tilde{\tau}(t)+T)) \sin \theta_i\right\}\\
	&\leq \left(y-\hat{c} (\tilde{\tau}(t)+T) +x \cdot \nu_j \cot \theta_j\right) \sin \theta_j < -X^{*} \cdot \sin \theta_j \leq -X^{*} \cdot \min _{1 \leq i \leq n}\{\sin \theta_i\}.
\end{align*}
Since $\theta_i\in(0,\pi/2]$, when $X^{*}>0$ is large enough, one has $ d\left((x, y), \mathcal{R}+\hat{c}(\tilde{\tau}(t)+T) e_0\right) \rightarrow+\infty$ and $\left|\bar{V}(\tilde{\tau}(t)+T, x, y)-\underline{V}(\tilde{\tau}(t)+T, x, y)\right|\leq 2\varepsilon$ by \eqref{3.6}. Then one gets 
\begin{equation*}
 V_{\delta}^{+}(t, x, y ; T)\geq V(T+\tilde{\tau}, x, y)\geq \underline{V}(T+\tilde{\tau}, x, y)\geq \bar{V}(\tilde{\tau}(t)+T, x, y)-2\varepsilon>U_{e(x)}(\xi, x, y)-\frac{ 2\gamma_{\star}}{3},
\end{equation*}
where $\xi$ is evaluated at $(T+\tilde{\tau}, x, y)$. Therefore, there exists a large enough constant $X^{**}>0$ such that $V_{\delta}^{+}(t, x, y ; T)\geq V(T+\tilde{\tau}, x, y)\geq 1-\gamma_{\star}$. Moreover, 
\begin{equation*}
	\begin{aligned}
		V(T+\tilde{\tau}, x, y)\leq V_{\delta}^{+}(t, x, y ; T)&\leq\bar{V}(\tilde{\tau}(t)+T, x, y)+\delta e^{-\lambda t}\\
		&\leq U_{e(x)}(\xi, x, y)+\varepsilon+\delta e^{-\lambda t}\leq U_{e(x)}(\xi, x, y)+\frac{ \gamma_{\star}}{3}+\frac{ \gamma_{\star}}{4}<1+\gamma_{\star},
	\end{aligned}
\end{equation*}
where $\xi$ is evaluated at $(T+\tilde{\tau}, x, y)$. Take $X^{\prime\prime}:=\max\{X^{*} + \widetilde{C}, X^{**}\}$, then we have
\begin{equation*}
f\left(x, y, V\right)-f\left(x, y, V_{\delta}^{+}\right)\geq -\frac{\kappa_1}{2} \cdot (-\delta e^{-\lambda t})=\frac{\kappa_1}{2}\delta e^{-\lambda t}
\end{equation*}
and 
\begin{equation*}
	\mathcal{L} V_{\delta}^{+} \geq -\lambda\delta e^{- \lambda t} +\frac{\kappa_1}{2}\delta e^{-\lambda t}.
\end{equation*}
Let $\lambda\in\left(0,\frac{\kappa_1}{2}\right)$, then we prove $\mathcal{L} V_{\delta}^{+} >0$ in Case $2$.
	
\textbf{Case 3:} $-X^{\prime\prime}<\xi(\tilde{\tau}(t)+T, x, y)<X^{\prime}$ and $t \geq 0$.
By Step $1$, there exists a positive constant $\tilde{\alpha}_1^{+}(\beta):=\bar{\alpha}_0^{+}\left(\beta, \max \left\{X_{*}^{\prime}, X_{*}^{\prime \prime}\right\}\right)$ such that, for any $0<\alpha<\tilde{\alpha}_1^{+}(\beta)$ there exists $\bar{r}=\bar{r}(\alpha)$ such that $V_{\tilde{\tau}}>\bar{r}$ in $\left\{(t, x, y):-X_{*}^{\prime \prime} \leq \eta(T+\tilde{\tau}(t), x, y) \leq X_{*}^{\prime}\right\}$, where $X_{*}^{\prime}:= X^{\prime}\sqrt{(C_3+\max_{1\le i\le n} \{|\nu_i \cot \theta_i|\})^2+1}$ and $X_{*}^{\prime\prime}:= X^{\prime\prime}\sqrt{(C_3+\max_{1\le i\le n} \{|\nu_i \cot \theta_i|\})^2+1}$. With similar arguments as those in Case $3$ of the proof of Lemma \ref{Lemma 4.1}, we can obtain 
\begin{equation*}
	\begin{aligned}
		\mathcal{L} V_{\delta}^{+} &\geq \tilde{\varrho} \delta \lambda e^{-\lambda t} V_{\tilde{\tau}}-\delta\lambda e^{- \lambda t}-\delta e^{- \lambda t} C_{0}^{*}(\beta) +f\left(x, y, V\right)-f\left(x, y, V_{\delta}^{+}\right)\\
		&\geq\delta e^{-\lambda t} \left(\tilde{\varrho} \lambda \bar{r} -\lambda - C_{0}^{*}(\beta) -\left\|f_u\right\|_{L^{\infty}} \right)
	\end{aligned}
\end{equation*}
in $\left\{(t, x, y):-X^{\prime \prime} \leq \xi(T+\tilde{\tau}(t), x, y) \leq X^{\prime}, t \geq 0\right\}$, where $C_{0}^{*}(\beta)$ is a positive constant. Set 
\begin{equation*}
	\tilde{\varrho}=\tilde{\varrho}(\beta, \lambda, \alpha)>\frac{\left\|f_u\right\|_{L^{\infty}}+\lambda+C_{0}^{*}(\beta)}{\lambda \bar{r}},
\end{equation*}
and then $\mathcal{L} V_{\delta}^{+}>0$ in $\left\{(t, x, y): \xi(T+\tilde{\tau}(t), x, y) \in\left[-X^{\prime \prime}, X^{\prime}\right], t \geq 0\right\}$. In conclusion, let $\tilde{\alpha}_0^{+}(\beta)<\min\{\tilde{\alpha}_1^{+}(\beta),\alpha_0^{+}(\beta, \varepsilon)\}$, where $\varepsilon<\gamma_{\star} / 3$. Then $V_{\delta}^{+}$ is a supersolution of \eqref{1.1} for $t \geq 0$ and $(x, y) \in \mathbb{R}^N$, for all $T \in \mathbb{R}$ and $\delta \in\left(0, \gamma_{\star} / 4\right]$.
	
{\it Step 3: we prove that $V_{\delta}^{-}(t, x, y ; T)$ is a subsolution.}
The approach is to find two numbers $\hat{X}^{\prime}>1$ and $\hat{X}^{\prime \prime}>1$ and consider the inequality
\begin{equation*}
	\mathcal{L} V_{\delta}^{-}:=\partial_t	V_{\delta}^{-}-\Delta_{x, y} V_{\delta}^{-}-f\left(x, y, V_{\delta}^{-}\right) \leq 0, \quad \forall(t, x, y) \in[0,+\infty) \times \mathbb{R}^N
\end{equation*}
in three cases $\xi(\hat{\tau}(t)+T, x, y)>\hat{X}^{\prime}, \xi(\hat{\tau}(t)+T, x, y)<-\hat{X}^{\prime \prime}$ and $\xi(\hat{\tau}(t)+T, x, y) \in\left[-\hat{X}^{\prime \prime}, \hat{X}^{\prime}\right]$, respectively. Since $V$ is a solution of \eqref{1.1}, one has
\begin{equation*}
	\begin{aligned}
		\mathcal{L} V_{\delta}^{-} \leq & -\hat{\varrho} \delta \hat{\lambda} e^{-\hat{\lambda} t} V_{\hat{\tau}} +f\left(x, y, V\right)-f\left(x, y, V_{\delta}^{-}\right) \\
		& +\left(\partial_t-\Delta_{x, y}\right)\left(-\delta e^{-\hat{\lambda} t} \times\left[U_{e_{i}}^\beta(\eta, x, y) \omega(\xi)+(1-\omega(\xi))\right]\right),
	\end{aligned}
\end{equation*}
in $\mathbb{R} \times \mathbb{R}^N$, where $\xi$, $\eta$, $V$ and all of its derivatives take value at $(\hat{\tau}(t)+T, x, y)$.

\textbf{Case 1:} $\xi(\hat{\tau}(t)+T, x, y)>\hat{X}^{\prime}$ and $t \geq 0$, where $\hat{X}^{\prime}>1$ is to be chosen.
In this case, $\omega(\xi) \equiv 1$. Thus,
\begin{equation*}
	\mathcal{L} V_{\delta}^{-} \leq -\hat{\varrho}\delta \hat{\lambda} e^{-\hat{\lambda} t} V_{\hat{\tau}}+\left(\partial_t-\Delta_{x, y}\right)\left(-\delta e^{-\hat{\lambda} t} U_{e_{i}}^\beta(\eta)\right)+f\left(x, y, V\right)-f\left(x, y, V_{\delta}^{-}\right).
\end{equation*}
Then we have 
\begin{align}\label{4.13}
	&\left(\partial_t-\Delta_{x, y}\right)\left(-\delta e^{-\hat{\lambda} t} U_{e_{i}}^\beta(\eta)\right)\nonumber\\
	&=\delta\hat{\lambda} e^{-\hat{\lambda} t}U_{e_{i}}^\beta-\delta e^{- \hat{\lambda} t}\beta U_{e_{i}}^{\beta-1} \partial_{\eta}U_{e_{i}}(-\hat{c})(1-\hat{\varrho} \delta \hat{\lambda} e^{- \hat{\lambda} t})\nonumber\\
	&\quad+\delta e^{- \hat{\lambda} t}\beta(\beta-1) U_{e_{i}}^{\beta-2}\left(\sum_{k=1}^{N-1}(\partial_{x_k}U_{e_{i}})^2 + (\partial_{y}U_{e_{i}})^2\right)\nonumber\\
	&\quad+2\delta e^{-\hat{\lambda} t}\beta(\beta-1) U_{e_{i}}^{\beta-2}\partial_{\eta}U_{e_{i}}\left(\sum_{k=1}^{N-1}\partial_{x_k}U_{e_{i}}\partial_{x_k}\eta + \partial_{y} U_{e_{i}}\partial_{y}\eta\right)\nonumber\\
	&\quad+\delta e^{- \hat{\lambda} t}\beta U_{e_{i}}^{\beta-1}\left(\sum_{k=1}^{N-1}\partial_{x_k x_k}U_{e_{i}} + \partial_{y y}U_{e_{i}}\right)+2\delta e^{- \hat{\lambda} t}\beta U_{e_{i}}^{\beta-1}\left(\sum_{k=1}^{N-1}\partial_{x_k}\partial_{\eta}U_{e_{i}}\partial_{x_k}\eta + \partial_{y}\partial_{\eta} U_{e_{i}}\partial_{y}\eta\right)\nonumber\\
	&\quad+\delta e^{-\hat{\lambda} t}\beta(\beta-1) U_{e_{i}}^{\beta-2}(\partial_{\eta} U_{e_{i}})^2\left(\sum_{k=1}^{N-1}(\partial_{x_k}\eta)^2 + (\partial_{y}\eta)^2\right)\nonumber\\
	&\quad+\delta e^{-\hat{\lambda} t}\beta U_{e_{i}}^{\beta-1}\partial_{\eta \eta} U_{e_{i}}\left(\sum_{k=1}^{N-1}(\partial_{x_k}\eta)^2 + (\partial_{y}\eta)^2\right)+\delta e^{- \hat{\lambda} t}\beta U_{e_{i}}^{\beta-1}\partial_{\eta} U_{e_{i}}\left(\sum_{k=1}^{N-1}\partial_{x_k x_k}\eta + \partial_{y y}\eta\right)\nonumber\\
	&=\delta\hat{\lambda} e^{-\hat{\lambda} t}U_{e_{i}}^\beta + \delta e^{-\hat{\lambda} t}U_{e_{i}}^\beta \left(\alpha\beta \frac{\partial_{\eta} U_{e_{i}}}{U_{e_{i}}}\sum_{k=1}^{N-1}\partial_{\zeta_k \zeta_k}\varphi(\alpha x)\right)\nonumber\\
	&\quad+\delta e^{-\hat{\lambda} t}U_{e_{i}}^\beta \beta\left[\frac{\Delta_{x, y}U_{e_{i}}+2\nabla_{x, y}\partial_{\eta}U_{e_{i}}(-\nabla_{\zeta}\varphi(\alpha x),1) }{U_{e_{i}}}\right.	\nonumber\\
	&\qquad +\beta \frac{\sum_{k=1}^{N-1}\left(\partial_{x_k} U_{e_i}-\partial_{\zeta_k}\varphi(\alpha x)\partial_{\eta} U_{e_i}\right)^2+\left(\partial_y U_{e_i}+\partial_{\eta} U_{e_i}\right)^2}{U_{e_i}^2} \nonumber\\
	& \qquad-\frac{\sum_{k=1}^{N-1}\left(\partial_{x_k} U_{e_i}-\partial_{\zeta_k}\varphi(\alpha x)\partial_{\eta} U_{e_i}\right)^2+\left(\partial_y U_{e_i}+\partial_{\eta} U_{e_i}\right)^2}{U_{e_i}^2} \nonumber\\
	& \qquad\left.+\frac{\partial_{\eta \eta} U_{e_i}}{U_{e_i}}\left(\sum_{k=1}^{N-1}(\partial_{\zeta_k}\varphi (\alpha x))^2+1\right)+\hat{c} \frac{\partial_{\eta} U_{e_i}}{U_{e_i}}\right]\nonumber\\
	&\quad-\delta e^{- \hat{\lambda} t}\beta U_{e_{i}}^{\beta-1} \partial_{\eta}U_{e_{i}}\hat{c}\hat{\varrho} \delta \hat{\lambda} e^{- \hat{\lambda} t}\nonumber\\
	&=:\delta\hat{\lambda} e^{-\hat{\lambda} t}U_{e_{i}}^\beta+N_{1}+N_{2}+N_{3}.
\end{align}
Recall that \eqref{3.16}, \eqref{3.17} and	$\beta_1^{*}=\frac{\hat{c}}{2 c_{e_i}\left((C_3+\max_{1\le i\le n} \{|\nu_i \cot \theta_i|\})^2+1\right)}$,
then for any $\beta\in(0,\beta_1^{*}]$ there exists a large enough number $\hat{X}_1^{\prime}>1$ such that
\begin{equation}\label{4.14}
	\begin{aligned}
	 N_{2}&\leq \delta e^{-\hat{\lambda} t}U_{e_{i}}^\beta \beta\left[\frac{\hat{c}}{2 c_{e_i}\left((C_3+\max_{1\le i\le n} \{|\nu_i \cot \theta_i|\})^2+1\right)} c_{e_{i}}^2 \left(\sum_{k=1}^{N-1}(\partial_{\zeta_k}\varphi(\alpha x))^2 +1\right) - \hat{c} c_{e_{i}} \right]\\
	 &\leq \delta e^{-\hat{\lambda} t}U_{e_{i}}^\beta \beta\left(\frac{\hat{c} c_{e_{i}}}{2}- \hat{c} c_{e_{i}}\right)< \delta e^{-\hat{\lambda} t}U_{e_{i}}^\beta \beta\left(-\frac{\hat{c} c_{e_{i}}}{4}\right).
	\end{aligned}
\end{equation}
Thus, for arbitrary $0<\alpha \leq \alpha_0^{+}(\beta, \varepsilon) \leq \alpha_1^{+}(\beta)$ ($\alpha_0^{+}(\beta, \varepsilon),\alpha_1^{+}(\beta)$ are given in Lemma \ref{Lemma 3.2} and Lemma \ref{Lemma 4.1}), it follows from \eqref{4.13} and \eqref{4.14} that
\begin{equation}\label{4.15}
	N_1 + N_2<- \delta e^{- \hat{\lambda} t} U_{e_{i}}^{\beta}\times \beta\frac{\hat{c}c_{e_i}}{8 }.
\end{equation}
It follows from \eqref{4.13} and \eqref{4.15} that 
\begin{equation*}
    \mathcal{L} V_{\delta}^{-} \leq -\hat{\varrho}\delta \hat{\lambda} e^{-\hat{\lambda} t} V_{\hat{\tau}} +\delta\hat{\lambda} e^{-\hat{\lambda} t}U_{e_{i}}^\beta- \delta e^{- \hat{\lambda} t} U_{e_{i}}^{\beta}\times \beta\frac{\hat{c}c_{e_i}}{8 }+N_{3}+f\left(x, y, V\right)-f\left(x, y, V_{\delta}^{-}\right).
\end{equation*}
By definitions of $V$ and $V_{\delta}^{-}$, we know
\begin{equation*}
	 V_{\delta}^{-}(\hat{\tau}, x, y; T) \leq V(T+\hat{\tau}, x, y)\leq \bar{V}(T+\hat{\tau}, x, y) \leq U_{e(x)}(\xi(T+\hat{\tau}, x, y), x, y) +\varepsilon.
\end{equation*}
Recall that $\alpha<\alpha_0^{+}\left(\beta, \varepsilon\right)$  in $\bar{V}$ where $\varepsilon<\gamma_{\star} / 3$, there exists a sufficiently large number $\hat{X}_2^{\prime}>1$ such that 
\begin{equation*}
	f\left(x, y, V\right)-f\left(x, y, V_{\delta}^{-}\right)= 0,\quad\forall(\xi, x, y) \in\left(X_2^{\prime},+\infty\right) \times \mathbb{R}^N.
\end{equation*}
Take $\hat{X}^{\prime}=\max\{\hat{X}_1^{\prime},\hat{X}_2^{\prime}\}$ and then
\begin{equation*}
	\mathcal{L} V_{\delta}^{-} \leq \delta\hat{\lambda} e^{-\hat{\lambda} t}U_{e_{i}}^\beta- \delta e^{- \lambda t} U_{e_{i}}^{\beta}\times \beta\frac{\hat{c}c_{e_i}}{8 }+N_{3}\leq \delta e^{-\hat{\lambda} t}U_{e_{i}}^\beta\left[\hat{\lambda}-\beta\frac{\hat{c}c_{e_i}}{8 }-\frac{\partial_{\eta}U_{e_{i}}}{U_{e_{i}}}\beta\hat{c}\hat{\varrho}\delta\hat{\lambda}\right],
\end{equation*}
in $\left\{(t, x, y): \xi(T+\tilde{\tau}(t), x, y) \in \left(\hat{X}^{\prime}, +\infty\right), t \geq 0\right\}$. Since $\frac{\partial_{\eta}U_{e_{i}}}{U_{e_{i}}}$ is bounded for $\xi(\hat{\tau}(t)+T, x, y)>\hat{X}^{\prime}$ and \eqref{3.16}, if $\delta<\delta_{1}^0:=\frac{1}{\hat{\varrho}}$, then for any
\begin{equation*}
	0<\hat{\lambda}<\hat{\lambda}(\beta):=\frac{\beta \hat{c} c_{e_i}}{8\left(1+\bar{K} e^{-\frac{3 \kappa}{4}X^{\prime} }\beta\hat{c}\right)},
\end{equation*}
one has $\mathcal{L} V_{\delta}^{-} <0$ in Case $1$.

\textbf{Case 2:} $\xi(\hat{\tau}(t)+T, x, y)<-\hat{X}^{\prime\prime}$ and $t \geq 0$, where $\hat{X}^{\prime\prime}>1$ is to be chosen.
In this case, $\omega(\xi) \equiv 0$. Then we obtain
\begin{equation*}
	\begin{aligned}
	\mathcal{L} V_{\delta}^{-} &\leq -\hat{\varrho}\delta \hat{\lambda} e^{-\hat{\lambda} t} V_{\hat{\tau}}+\hat{\lambda}\delta e^{- \hat{\lambda} t} +f\left(x, y, V\right)-f\left(x, y, V_{\delta}^{-}\right)\\
	&\leq \hat{\lambda}\delta e^{- \hat{\lambda} t} +f\left(x, y, V\right)-f\left(x, y, V_{\delta}^{-}\right).
	\end{aligned}
\end{equation*}
Similar to the Case $2$ of Step $2$, there exists a large enough constant $\hat{X}^{*}>0$ such that 
\begin{equation*}
	d\left((x, y), \mathcal{R}+\hat{c}(\hat{\tau}(t)+T) e_0\right) \rightarrow+\infty
\end{equation*}
 and $\left|\bar{V}(\hat{\tau}(t)+T, x, y)-\underline{V}(\hat{\tau}(t)+T, x, y)\right|\leq 2\varepsilon$ by \eqref{3.6}. Let $\delta<\delta_{2}^{0}:=\gamma_{\star} / 4$ and recall that $\varepsilon<\gamma_{\star} / 3$. There exists a constant $\hat{X}^{**}>0$ such that
\begin{equation*}
	\begin{aligned}
    V \geq V_{\delta}^{-}&= V(T+\hat{\tau}, x, y)-\delta e^{-\hat{\lambda} t} \geq \underline{V}(\hat{\tau}(t)+T, x, y)\geq \bar{V}(\hat{\tau}(t)+T, x, y)-2\varepsilon-\delta e^{-\hat{\lambda} t}\\
    &\geq U_{e(x)}(\xi, x, y)-2\varepsilon-\delta e^{-\hat{\lambda} t}> U_{e(x)}(\xi, x, y) - \frac{2\gamma_{\star}}{3}-\frac{\gamma_{\star}}{4}\geq U_{e(x)}(\xi, x, y) - \frac{11}{12}\gamma_{\star}\geq 1-\gamma_{\star}
	\end{aligned}
\end{equation*}
and 
\begin{equation*}
	V_{\delta}^{-}(\hat{\tau}, x, y;T)\leq V(T+\hat{\tau}, x, y) <1,
\end{equation*}
where $\xi$ is evaluated at $(T+\hat{\tau}, x, y)$. Take $\hat{X}^{\prime\prime}:=\max\{\hat{X}^{*} + \widetilde{C}, \hat{X}^{**}\}$, then one has
\begin{equation*}
	f\left(x, y, V\right)-f\left(x, y, V_{\delta}^{-}\right)\leq -\frac{\kappa_1}{2} \delta e^{-\hat{\lambda} t} \;\text{ and }\; \mathcal{L} V_{\delta}^{-} \leq \hat{\lambda}\delta e^{-\hat{\lambda} t} -\frac{\kappa_1}{2}\delta e^{-\hat{\lambda} t}.
\end{equation*}
Let $\hat{\lambda}\in\left(0,\frac{\kappa_1}{2}\right)$, then we prove $\mathcal{L} V_{\delta}^{-} <0$ in Case $2$.

\textbf{Case 3:} $-\hat{X}^{\prime\prime}<\xi(\hat{\tau}(t)+T, x, y)<\hat{X}^{\prime}$ and $t \geq 0$.
From Step $1$, there exists a constant $\hat{\alpha}_1^{+}(\beta):=\bar{\alpha}_0^{+}\left(\beta, \max \left\{\hat{X}_{*}^{\prime}, \hat{X}_{*}^{\prime \prime}\right\}\right) > 0$ such that, for any $0<\alpha<\hat{\alpha}_1^{+}(\beta)$ there exists $\hat{r}=\hat{r}(\alpha)$ such that $V_{\hat{\tau}}>\hat{r}$ in $\left\{(t, x, y):-\hat{X}_{*}^{\prime \prime} \leq \eta(T+\hat{\tau}(t), x, y) \leq \hat{X}_{*}^{\prime}\right\}$, where $\hat{X}_{*}^{\prime}:= \hat{X}^{\prime}\sqrt{(C_3+\max_{1\le i\le n} \{|\nu_i \cot \theta_i|\})^2+1}$ and $\hat{X}_{*}^{\prime\prime}:= \hat{X}^{\prime\prime}\sqrt{(C_3+\max_{1\le i\le n} \{|\nu_i \cot \theta_i|\})^2+1}$. With similar arguments as those in Case $3$ of the proof of Lemma \ref{Lemma 4.1}, we know
\begin{align}\label{4.17}
		\mathcal{L} V_{\delta}^{-} &\leq - \hat{\varrho} \delta \hat{\lambda} e^{-\hat{\lambda} t} V_{\tilde{\tau}}+\delta(\hat{\lambda}+C_{1}^{*}(\beta)) e^{- \hat{\lambda} t}+C_{2}^{*}(\beta)\hat{\varrho} \delta^{2} \hat{\lambda} e^{-\hat{\lambda} t} +f\left(x, y, V\right)-f\left(x, y, V_{\delta}^{+}\right) \nonumber\\
		&\leq\delta e^{-\hat{\lambda} t} \left(-\hat{\varrho} \hat{\lambda} \hat{r} +\hat{\lambda} + C_{1}^{*}(\beta)+ C_{2}^{*}(\beta)\hat{\varrho} \delta \hat{\lambda} +\left\|f_u\right\|_{L^{\infty}} \right)
\end{align}
in $\left\{(t, x, y):-\hat{X}^{\prime \prime} \leq \xi(T+\hat{\tau}(t), x, y) \leq \hat{X}^{\prime}, t\geq 0\right\}$. Set
\begin{equation*}
\delta \leq \delta_{3}^{0}:=\frac{\hat{r}}{2 C_{2}^{*}(\beta)} \;\text { and }\; \hat{\varrho}>\frac{2\left(\left\|f_u\right\|_{L^{\infty}}+\hat{\lambda}+C_{1}^{*}(\beta)\right)}{\hat{\lambda} \hat{r}},
\end{equation*}
then it follows from \eqref{4.17} that
\begin{equation*}
\mathcal{L} V_{\delta}^{-}<0 \text { in }\left\{(t, x, y): \xi(T+\hat{\tau}(t), x, y) \in\left[-\widehat{X}^{\prime \prime}, \widehat{X}^{\prime}\right], t \geq 0\right\} .
\end{equation*}
Finally, let $\delta^0:=\min \left\{\delta_1^0, \delta_2^0, \delta_3^0\right\}$ and $\hat{\alpha}_0^{+}(\beta)<\min\{\hat{\alpha}_1^{+}(\beta),\alpha_0^{+}(\beta, \varepsilon)\}$ where $\varepsilon<\gamma_{\star} / 3$, then $V_{\delta}^{-}$is a subsolution of \eqref{1.1} for $t \geq 0$ and $(x, y) \in \mathbb{R}^N$, for all $T \in \mathbb{R}$ and $\delta \in\left(0, \delta^0\right]$.
\end{proof}

Next, we are going to prove the stability of the curved front $V(t, x, y)$ in Theorem \ref{Theorem 2.18}.

\begin{proof}[Proof of Theorem \ref{Theorem 2.20}]
It follows from \eqref{3.4} and the fact that $\varphi(x)\geq\psi(x)$ in $\mathbb{R}^{N-1}$ that 
\begin{equation}\label{4.18}
	\eta(t,x,y)=y-\hat{c} t-\varphi(\alpha x) / \alpha\leq y-\hat{c} t-\psi(\alpha x)/ \alpha = \min_{1\leq i\leq n} \left\{y-\hat{c} t+x \cdot \nu_i \cot \theta_i\right\}.
\end{equation}

{\it Step 1: construct supersolutions of Cauchy problem \eqref{Cauchy problem} with initial value $u_0(x, y)$.} 
Define $\beta_1:=\min \left\{v / K, \beta^*\right\}$, where positive constants $v$, $K$ and $\beta^*$ are given in \eqref{2.17}, Theorem \ref{Theorem 2.3} and Lemma \ref{Lemma 4.1}, respectively. For any $\delta \in\left(0,  \gamma_{\star} / 4\right]$, it follows from \eqref{2.17} there exists a constant $R_{\delta}>0$ such that
\begin{equation}\label{4.19}
	u_0(x, y) \leq \underline{V}(0, x, y)+\delta\left(\frac{C_1}{2}\right)^{\beta_1} \min \left\{1, e^{- v \min_{1\leq i\leq n} \left\{x \cdot \nu_i \cot \theta_i+y \right\}}\right\}
\end{equation}
for all $d\left((x, y), \mathcal{R}\right)>R_{\delta}$, where the constant $C_1>0$ is given in Theorem \ref{Theorem 2.3}. We claim that
\begin{equation}\label{4.20}
	W_{\delta}^{+}(0, x, y) \geq u_0(x, y) \text { in } \mathbb{R}^N
\end{equation}
for all $\delta\in\left(0, \gamma_{\star} / 4\right]$, where parameters in Lemma \ref{Lemma 4.1} are taken as $\beta=\beta_1$, $\lambda=\lambda\left(\beta_1\right), \varrho=\varrho\left(\beta_1, \lambda\right), \forall \varepsilon \in\left(0, \varepsilon_0^{+}\left(\beta_1\right)\right)$, and $0<\alpha<\alpha_0^{+}\left(\beta_1, \varepsilon\right)$ is to be determined.
	
\textbf{Case 1: }$\min_{1\leq i\leq n} \left\{y-\hat{c} t+x \cdot \nu_i \cot \theta_i\right\}>0$. By Theorem \ref{Theorem 2.3}, there is a positive constant $X_*>0$ such that
\begin{equation}\label{4.21}
	U_{e_{i}}(\eta, x, y) \geq \frac{C_1}{2} e^{-c_{e_{i}} \eta}, \quad \forall(\eta, x, y) \in\left(X_*,+\infty\right) \times \mathbb{R}^N.
\end{equation}
Since $\beta_1 \leq v / K$, it hlods from Lemma \ref{Lemma 4.1}, \eqref{4.18}, \eqref{4.19}, \eqref{4.21} and $\varphi(x)\geq\psi(x)$ that
\begin{equation*}
   \begin{aligned}
	W_{\delta}^{+}(0, x, y) &\geq \bar{V}(0, x, y)+\delta U_{e_{i}}^{\beta_1}(\eta(\tau(0), x, y), x, y)\\
   	& \geq \underline{V}(0, x, y)+\delta\left(\frac{C_1}{2}\right)^{\beta_1} e^{-\beta_1 c_{e_{i}} \eta} \\
   	& \geq \underline{V}(0, x, y)+\delta\left(\frac{C_1}{2}\right)^{\beta_1} e^{-v \eta} \\
   	& \geq \underline{V}(0, x, y)+\delta\left(\frac{C_1}{2}\right)^{\beta_1} e^{-v \min_{1 \leq i \leq n}\{y+x \cdot \nu_i \cot \theta_i\}} \\
   	& \geq u_0(x, y)
   \end{aligned}
\end{equation*}
in $\left\{(x, y): \eta(0, x, y)>X_*, d\left((x, y), \mathcal{R}\right)>R_{\delta}\right\}$. Then it follows from Theorem \ref{Theorem 2.12} that
\begin{equation*}
W_{\delta}^{+}(0, x, y) \geq \bar{V}(0, x, y)+{\delta} U_{e_{i}}^{\beta_1}\left(X_*, x, y\right) \geq \bar{V}(0, x, y)+r_2
\end{equation*}
in $\left\{(x, y): \eta(0, x, y) \leq X_*\right\}$ for some constant $r_2>0$. Therefore, even if it means increasing $R_{\delta}$, it follows from \eqref{2.17} that
\begin{equation*}
	 u_0(x, y)-\underline{V}(0, x, y) \leq r_2 \text { in }\left\{(x, y): \eta(0, x, y) \leq X_*, d\left((x, y), \mathcal{R}\right)>R_{\delta}\right\}
\end{equation*}
and
\begin{equation*}
W_{\delta}^{+}(0, x, y)\geq \bar{V}(0, x, y)+r_2 \geq \underline{V}(0, x, y)+r_2 \geq u_0(x, y) \text { in }\left\{(x, y): \eta(0, x, y) \leq X_*, d\left((x, y), \mathcal{R}\right)>R_{\delta}\right\}.
\end{equation*}
By \cite[Lemma 4.3]{Guo H2025}, we can obtain
\begin{equation*}
\xi(0, x, y)=\frac{y-\varphi(\alpha x) / \alpha}{\sqrt{1+|\nabla \varphi(\alpha x)|^2}} \rightarrow-\infty \text { as } \alpha \rightarrow 0^{+} \text{ uniformly in } \left\{(x, y):d\left((x, y), \mathcal{R}\right)\leq R_{\delta}\right\}.
\end{equation*}
By \eqref{4.18}, Theorem \ref{Theorem 2.8} and Lemma \ref{Lemma 4.1}, there is a constant 
$\alpha_5^{+}(\delta)>0$ such that for each $0<\alpha<\alpha_5^{+}$,
\begin{equation}\label{4.22}
W_{\delta}^{+}(0, x, y) \geq U_{e(x)}(\xi, x, y)+\delta U_{e_{i}}^{\beta_1}\left(\eta, x, y\right) \geq 1 \geq u_0(x, y)
\end{equation}
in $\left\{(x, y):d\left((x, y), \mathcal{R}\right)\leq R_{\delta}\right\}$. Thus, we show $W_{\delta}^{+}(0, x, y) \geq u_0(x, y)$ in Case $1$.
	
\textbf{Case 2: }$\min_{1\leq i\leq n} \left\{y-\hat{c} t+x \cdot \nu_i \cot \theta_i\right\}\leq 0$. It follows from \eqref{4.18} and Theorem \ref{Theorem 2.12} that
\begin{equation*}
W_{\delta}^{+}(0, x, y) \geq \bar{V}(0, x, y)+\delta U_{e_{i}}^{\beta_1}(0, x, y) \geq \underline{V}(0, x, y)+r_3
\end{equation*}
for some constant $r_3>0$. By \eqref{2.17}, even if it means increasing $R_{\delta}$, one knows
\begin{equation*}
	u_0(x, y)-\underline{V}(0, x, y) \leq r_3 \text { in }\left\{(x, y): d\left((x, y), \mathcal{R}\right)>R_{\delta}\right\}
\end{equation*}
and
\begin{equation*}
W_{\delta}^{+}(0, x, y)\geq \underline{V}(0, x, y)+r_3  \geq u_0(x, y) \text { in }\left\{(x, y): d\left((x, y), \mathcal{R}\right)>R_{\delta}\right\}.
\end{equation*}
Similar to arguments of \eqref{4.22}, there exists a constant $\alpha_6^{+}(\delta)>0$ such that for any $0<\alpha<\alpha_6^{+}$,
\begin{equation*}
W_{\delta}^{+}(0, x, y) \geq u_0(x, y) \text { in }\left\{(x, y): d\left((x, y), \mathcal{R}\right)\leq R_{\delta}\right\}
\end{equation*}
Thus we get $W_{\delta}^{+}(0, x, y) \geq u_0(x, y)$ in Case $2$.

Finally, we obtain that \eqref{4.20} is true for all $0<\alpha<\min \left\{\alpha_0^{+}\left(\beta_1, \varepsilon\right), \alpha_5^{+}(\delta), \alpha_6^{+}(\delta)\right\}$.

{\it Step 2: construct a time sequence.}
By \eqref{2.16}, Lemma \ref{Lemma 4.1} and Step $1$, using the comparison principle, one has
\begin{equation}\label{4.23}
\underline{V}(t, x, y) \leq u(t, x, y) \leq W_{\delta}^{+}(t, x, y) \text { in }[0,+\infty) \times \mathbb{R}^N
\end{equation}
for any $\delta \in\left(0,  \gamma_{\star} / 4\right]$. Let $t_m:=L_{N} m / \hat{c}$ and
\begin{equation*}
u_m(t, x, y):=u\left(t+t_m, x, y+L_{N} m\right) \text { in } \mathbb{R} \times \mathbb{R}^N
\end{equation*}
for all $m \in \mathbb{N}$, where $L_{N}$ is the period of $y$, then one has $t_m \rightarrow+\infty$ as $m \rightarrow +\infty$. By parabolic estimates, we can find a sequence of functions $\left\{u_{m_k}\right\}_{k \in \mathbb{N}}$ that converges locally uniformly to a function
$u_{\infty}(t, x, y)$ in $\mathbb{R} \times \mathbb{R}^N$, which is an entire solution of \eqref{1.1}. Since $\underline{V}\left(t+t_{m_k}, x, y+L_{N} m_k\right)=\underline{V}(t, x, y)$ and $\bar{V}\left(t+t_{m_k}, x, y+L_{N} m_k\right)=\bar{V}(t, x, y)$, it follows from \eqref{4.23} that
\begin{equation}\label{4.24}
\underline{V}(t, x, y) \leq u_{m_k}(t, x, y) \leq \bar{V}\left(t-\varrho \delta e^{-\lambda\left(t+t_{m_k}\right)}+\varrho \delta, x, y\right)+\delta e^{-\lambda\left(t+t_{m_k}\right)}
\end{equation}
in $\left[-t_{m_k},+\infty\right) \times \mathbb{R}^N$ for any $k \in \mathbb{N}$. Letting $k \rightarrow +\infty$, one has
\begin{equation}\label{4.25}
\underline{V}(t, x, y) \leq u_{\infty}(t, x, y) \leq \bar{V}(t+\varrho \delta, x, y) \text { in } \mathbb{R} \times \mathbb{R}^N.
\end{equation}
Let $w_n(t, x, y ; g(x, y))$ be the unique solution of the Cauchy problem
\begin{equation*}
\begin{cases}\partial_t w-\Delta_{x, y} w=f(x, y, w) & \text { when } t>-n,(x, y) \in \mathbb{R}^N, \\ w(t, x, y)=g(x, y) & \text { when } t=-n,(x, y) \in \mathbb{R}^N,\end{cases}
\end{equation*}
for all $n \in \mathbb{N}$. By \eqref{4.25} and the comparison principle, we obtain
\begin{equation}\label{4.26}
w_{m_k}\left(t, x, y ; \underline{V}\left(-m_k, x, y\right)\right) \leq u_{\infty}(t, x, y) \leq w_{m_k, \varepsilon, \delta}\left(t, x, y ; \bar{V}\left(-m_k+\varrho \delta, x, y\right)\right)
\end{equation}
in $\left[-m_k,+\infty\right) \times \mathbb{R}^N$ for all $k \in \mathbb{N}$, where $\varepsilon$ and $\delta$ are given in $W_\delta^{+}$. Recall the definition of $V(t, x, y)$ in the proof of Theorem \ref{Theorem 2.18}, then we know
\begin{equation}\label{4.27}
V(t, x, y) \leq u_{\infty}(t, x, y) \text { in } \mathbb{R}\times\mathbb{R}^N .
\end{equation}
Furthermore, it follows from the comparison principle that
\begin{equation*}
\underline{V}(t+\varrho \delta, x, y) \leq w_{m_k, \varepsilon, \delta}\left(t, x, y ; \bar{V}\left(-m_k+\varrho \delta, x, y\right)\right) \leq \bar{V}(t+\varrho \delta, x, y)
\end{equation*}
in $\left[-m_k,+\infty\right) \times \mathbb{R}^N$ for all $k \in \mathbb{N}$. By parabolic estimates, the sequence $\left\{w_{m_k, \varepsilon, \delta}\right\}$ converges, up to a subsequence, locally uniformly to an entire solution $w_{\infty}(t, x, y)$ of \eqref{1.1} as $k \rightarrow \infty$, $\varepsilon \rightarrow 0$ and $\delta \rightarrow 0$, which satisfies by \eqref{3.6} that
\begin{equation}\label{4.28}
\left|w_{\infty}(t, x, y)-\underline{V}(t, x, y)\right|\rightarrow 0,\quad d\left((x, y), \mathcal{R}+\hat{c} t e_0\right) \rightarrow+\infty,
\end{equation}
and $0 \leq w_{\infty} \leq 1$. By the virtue of Theorem \ref{Theorem 2.19} and \eqref{4.28}, one has $w_{\infty} \equiv V$ in $\mathbb{R}\times\mathbb{R}^N$. This implies that $u_{\infty} \equiv V$ in $\mathbb{R}\times\mathbb{R}^N$ by \eqref{4.26} and \eqref{4.27}.

Note that $\varrho(\beta, \lambda)$ is independent of $\delta$ in Lemma \ref{Lemma 4.1}. For any $\vartheta>0$, choose a parameter $\delta_1(\vartheta) \in\left(0,  \gamma_{\star} / 4\right]$ such that
\begin{equation}\label{4.29}
	\left|\bar{V}\left(t-\varrho \delta_1 e^{-\lambda\left(t+t_{m_k}\right)}+\varrho \delta_1, x, y\right)+\delta_1 e^{-\lambda\left(t+t_{m_k}\right)}-\bar{V}(t, x, y)\right| \leq \delta_1+\left\|\partial_t \bar{V}\right\|_{L^{\infty}} \varrho \delta_1 <\frac{\vartheta^4}{2}
\end{equation}
for any $(t, x, y) \in[0,+\infty) \times \mathbb{R}^N$ and $k \in \mathbb{N}$. By \eqref{4.24} and \eqref{4.29}, we know
\begin{equation}\label{4.30}
\underline{V}(t, x, y) \leq u_{m_k}(t, x, y) \leq \bar{V}(t, x, y)+\frac{\vartheta^4}{2} \text { in }[0,+\infty) \times \mathbb{R}^N
\end{equation}
for any $k \in \mathbb{N}$. According to \eqref{3.6}, \eqref{3.44} and \eqref{4.30}, there exists a constant $R_{\vartheta}>0$ such that
\begin{equation}\label{4.31}
	\begin{aligned}
	&\left|u_{m_k}(0, x, y)-V(0, x, y)\right|\\
    &\leq\left|u_{m_k}(0, x, y)-\bar{V}(0, x, y)\right| + \left|\bar{V}(0, x, y)-V(0, x, y)\right| \\
	&\leq\frac{\vartheta^4}{2}+\left|\bar{V}(0, x, y)-\underline{V}(0, x, y) \right|\leq  \vartheta^4\quad\text { in }\left\{(x, y): d\left((x, y), \mathcal{R}\right)> R_{\delta}\right\}
	\end{aligned}
\end{equation}
for all $k \in \mathbb{N}$. Since $\left\{u_{m_k}\right\}_{k \in \mathbb{N}}$ converges locally uniformly to $u_{\infty}(t, x, y) \equiv V(t, x, y)$ in $\mathbb{R} \times \mathbb{R}^N$, there exists a constant $k_0(\vartheta) \in \mathbb{N}$ such that
\begin{equation}\label{4.32}
\left|u_{m_k}(0, x, y)-V(0, x, y)\right| \leq \vartheta^4 \text { in }\left\{(x, y): d\left((x, y), \mathcal{R}\right)\leq R_{\delta}\right\}
\end{equation}
for all $k \geq k_0$. Letting $T_{\vartheta}:=t_{m_{k_0}}$, by virtue of \eqref{4.31}, \eqref{4.32} and Remark \ref{Remark 3.3}, it holds that for any $\vartheta>0$,
\begin{equation}\label{4.33}
\left|u\left(T_{\vartheta}, x, y\right)-V\left(T_{\vartheta}, x, y\right)\right|=\left|u_{m_{k_0}}\left(0, x, y-L_N m_{k_0}\right)-V\left(0, x, y-L_N m_{k_0}\right)\right| \leq \vartheta^4
\end{equation}
for all $(x, y) \in \mathbb{R}^N$.
	
{\it Step 3: construct suitable super- and subsolutions to show the stability of $V(t, x, y)$.} 
Fix arbitrary parameters of $W_{\delta}^{+}$ in Step $1$, we obtain from \eqref{3.44} and \eqref{4.23} that 
\begin{equation*}
|u(t, x, y)-V(t, x, y)| \leq\left|W_{\delta}^{+}(t, x, y)-\underline{V}(t, x, y)\right| \text { in }[0,+\infty) \times \mathbb{R}^N.
\end{equation*}
Similar to arguments in \eqref{3.41}, one knows 
\begin{equation}\label{4.34}
\sqrt{\left|W_{\delta}^{+}(t, x, y)\right|+\left|\underline{V}(t, x, y)\right|} \leq \Lambda^{*} \min \left\{1, e^{- \tilde{v} \min_{1\leq i\leq n} \left\{x \cdot \nu_i \cot \theta_i+y-\hat{c} t \right\}}\right\}.
\end{equation}
in $\mathbb{R} \times \mathbb{R}^N$, for some constants $\tilde{v}>0$ and $\Lambda^{*}>0$.

\textbf{Situation 1:} $\min_{1\leq i\leq n} \left\{x \cdot \nu_i \cot \theta_i+y-\hat{c} t \right\}>0$ and $\xi(\tau, x, y)>1$, where $\tau=t-\varrho \delta e^{-\lambda t}+\varrho \delta$ is given in Lemma \ref{Lemma 4.1}.
By Theorem \ref{Theorem 2.8}, one has
\begin{align*}
	\left|\underline{V}(t, x, y)\right|\leq\bar{K} e^{-\frac{3 \kappa}{4}\min_{1 \leq i \leq n}\{x \cdot \nu_i \cos \theta_i+(y-\hat{c} t) \sin \theta_i\} }\leq\bar{K} e^{-2\tilde{v} \min_{1 \leq i \leq n}\{x \cdot \nu_i \cot \theta_i+(y-\hat{c} t)\} },
\end{align*}
where $\tilde{v}\leq 3 \kappa \min_{1 \leq i \leq n} \{\sin\theta_i\} / 8$,
\begin{equation*}
	\xi(\tau, x, y)=\frac{y-\hat{c} \tau-\varphi(\alpha x) / \alpha}{\sqrt{1+|\nabla \varphi(\alpha x)|^2}}\geq\frac{y-\hat{c} \tau-\psi(x)-\widetilde{C}}{\sqrt{1+|\nabla \varphi(\alpha x)|^2}}\geq \frac{\min_{1 \leq i \leq n}\{y-\hat{c} \tau+x \cdot \nu_i \cot \theta_i\}-\widetilde{C}}{\sqrt{1+|\nabla \varphi(\alpha x)|^2}}
\end{equation*}
and
	\begin{align*}
	W_{\delta}^{+}(t, x, y)&\leq \left|U_{e(x)}(\xi(\tau, x, y), x, y)\right|+\left| U_{e_{i}}^\beta(\eta(\tau, x, y), x, y)\right|+\left| U_{e_{i}}^\beta(\eta(\tau, x, y), x, y)\right|\\
	&\leq\bar{K} e^{-\frac{3 \kappa }{4}\cdot\frac{y-\hat{c} \tau-\varphi(\alpha x) / \alpha}{\sqrt{1+|\nabla \varphi(\alpha x)|^2}}}+2\left(\bar{K} e^{-\frac{3 \kappa}{4}(y-\hat{c} \tau-\varphi(\alpha x) / \alpha)}\right)^\beta\\
	&\leq\bar{K} e^{-\frac{3 \kappa }{4}\cdot\frac{\min_{1 \leq i \leq n}\{y-\hat{c} \tau+x \cdot \nu_i \cot \theta_i\}-\widetilde{C}}{\sqrt{1+|\nabla \varphi(\alpha x)|^2}}}+2\left(\bar{K} e^{-\frac{3 \kappa}{4}(\min_{1 \leq i \leq n}\{y-\hat{c} \tau+x \cdot \nu_i \cot \theta_i\}-\widetilde{C})}\right)^\beta\\
	&\leq\bar{K} e^{-\frac{3 \kappa}{4}\cdot\frac{\min_{1 \leq i \leq n}\{y-\hat{c} t+x \cdot \nu_i \cot \theta_i\}-\widetilde{C}-\hat{c}\varrho \delta}{\sqrt{1+|\nabla \varphi(\alpha x)|^2}}}+2\bar{K} e^{-\frac{3 \kappa}{4}\beta(\min_{1 \leq i \leq n}\{y-\hat{c} t+x \cdot \nu_i \cot \theta_i\}-\widetilde{C}-\hat{c}\varrho \delta)}\\
	&\leq\bar{K}e^{\frac{3 \kappa}{4}(\widetilde{C}+\hat{c}\varrho \delta)}  e^{-\frac{3 \kappa}{4}\cdot\frac{\min_{1 \leq i \leq n}\{y-\hat{c} t+x \cdot \nu_i \cot \theta_i\}}{\sqrt{1+|\nabla \varphi(\alpha x)|^2}}}+2\bar{K}e^{\frac{3 \kappa}{4}\beta(\widetilde{C}+\hat{c}\varrho \delta)} \cdot e^{-\frac{3 \kappa}{4}\beta(\min_{1 \leq i \leq n}\{y-\hat{c} t+x \cdot \nu_i \cot \theta_i\})}\\
	&\leq 3\Lambda_{1}^{*} e^{-2 \tilde{v} \min_{1\leq i\leq n} \left\{x \cdot \nu_i \cot \theta_i+y-\hat{c} t \right\}},
	\end{align*}
where $\Lambda_{1}^{*}:=\max\{\bar{K},\bar{K}e^{\frac{3 \kappa}{4}(\widetilde{C}+\hat{c}\varrho \delta)}\}$, $\tilde{v}\leq\frac{3 \kappa }{8}\min_{1\leq i \leq n}\{\sin\theta_i, \beta, \frac{1}{\sqrt{(C_3+\max_{1\le i\le n} \{|\nu_i \cot \theta_i|\})^2+1}}\}$ and $\widetilde{C}=\frac{\sup_{\mathbb{R}^{N}}|\varphi(\alpha x)-\psi(\alpha x)|}{\alpha}$.

\textbf{Situation 2:} $\min_{1\leq i\leq n} \left\{x \cdot \nu_i \cot \theta_i+y-\hat{c} t \right\}>0$ and $0 \leq \xi(\tau, x, y)\leq1$, where $\tau=t-\varrho \delta e^{-\lambda t}+\varrho \delta$ is given in Lemma \ref{Lemma 4.1}.
By Theorem \ref{Theorem 2.8}, we have
\begin{align*}
	\left|\underline{V}(t, x, y)\right|\leq\bar{K} e^{-\frac{3 \kappa}{4}\min_{1 \leq i \leq n}\{x \cdot \nu_i \cos \theta_i+(y-\hat{c} t) \sin \theta_i\} }\leq\bar{K} e^{-2\tilde{v} \min_{1 \leq i \leq n}\{x \cdot \nu_i \cot \theta_i+(y-\hat{c} t)\} },
\end{align*}
where $\tilde{v}\leq 3 \kappa \min_{1 \leq i \leq n} \{\sin\theta_i\} / 8$. Moreover, one has
\begin{equation*}
	1\geq\xi(\tau, x, y)=\frac{y-\hat{c} \tau-\varphi(\alpha x) / \alpha}{\sqrt{1+|\nabla \varphi(\alpha x)|^2}}\geq\frac{y-\hat{c} \tau-\psi(x)-\widetilde{C}}{\sqrt{1+|\nabla \varphi(\alpha x)|^2}}\geq \frac{\min_{1 \leq i \leq n}\{y-\hat{c} \tau+x \cdot \nu_i \cot \theta_i\}-\widetilde{C}}{\sqrt{1+|\nabla \varphi(\alpha x)|^2}}\geq 0
\end{equation*}
and thereby
\begin{equation*}
0 < \min_{1 \leq i \leq n}\{y-\hat{c} \tau+x \cdot \nu_i \cot \theta_i\}\leq\sqrt{1+|\nabla \varphi(\alpha x)|^2}+\widetilde{C}\leq \sqrt{1+(C_3+\max_{1\le i\le n} \{|\nu_i \cot \theta_i|\})^2}+\widetilde{C}=:\hat{C}.
\end{equation*}
Then, if $\xi(\tau, x, y)\geq 0$, we know 
\begin{equation*}
	\begin{aligned}
		W_{\delta}^{+}(t, x, y)&\leq \left|U_{e(x)}(\xi(\tau, x, y), x, y)\right|+2\\
		&\leq \bar{K}e^{\frac{3 \kappa}{4}(\widetilde{C}+\hat{c}\varrho \delta)} \cdot e^{-\frac{3 \kappa}{4}\frac{\min_{1 \leq i \leq n}\{y-\hat{c} t+x \cdot \nu_i \cot \theta_i\}}{\sqrt{1+|\nabla \varphi(\alpha x)|^2}}}\\
		&\quad+2e^{\frac{3 \kappa}{4}\min_{1 \leq i \leq n}\{y-\hat{c} t+x \cdot \nu_i \cot \theta_i\}}\cdot e^{-\frac{3 \kappa}{4}\min_{1 \leq i \leq n}\{y-\hat{c} t+x \cdot \nu_i \cot \theta_i\}}\\
		&\leq 3\Lambda_{2}^{*} e^{-2 \tilde{v} \min_{1\leq i\leq n} \left\{x \cdot \nu_i \cot \theta_i+y-\hat{c} t \right\}},
	\end{aligned}
\end{equation*}
where $\Lambda_{2}^{*}:=\max\{\bar{K}, \bar{K}e^{\frac{3 \kappa}{4}(\widetilde{C}+\hat{c}\varrho \delta)}, e^{\frac{3 \kappa}{4}\hat{C}}\}$, $\tilde{v}\leq\frac{3 \kappa }{8}\min_{1\leq i \leq n}\{\sin\theta_
i, 1, \frac{1}{\sqrt{(C_3+\max_{1\le i\le n} \{|\nu_i \cot \theta_i|\})^2+1}}\}$ and $\widetilde{C}=\frac{\sup_{\mathbb{R}^{N}}|\varphi(\alpha x)-\psi(\alpha x)|}{\alpha}$.

\textbf{Situation 3:} $\min_{1\leq i\leq n} \left\{x \cdot \nu_i \cot \theta_i+y-\hat{c} t \right\}>0$ and $\xi(\tau, x, y)< 0$, where $\tau=t-\varrho \delta e^{-\lambda t}+\varrho \delta$ is given in Lemma \ref{Lemma 4.1}.
By Theorem \ref{Theorem 2.8}, one has
\begin{equation*}
	\left|\underline{V}(t, x, y)\right|\leq\bar{K} e^{-\frac{3 \kappa}{4}\min_{1 \leq i \leq n}\{x \cdot \nu_i \cos \theta_i+(y-\hat{c} t) \sin \theta_i\} }\leq\bar{K} e^{-2 \tilde{v}\min_{1 \leq i \leq n}\{x \cdot \nu_i \cot \theta_i+(y-\hat{c} t)\} }
\end{equation*}
for all $\tilde{v} \leq 3 \kappa \min_{1 \leq i \leq n} \{\sin\theta_i\} / 8$. Furthermore, we know
\begin{equation*}
	0\geq\xi(\tau, x, y)=\frac{y-\hat{c} \tau-\varphi(\alpha x) / \alpha}{\sqrt{1+|\nabla \varphi(\alpha x)|^2}}\geq\frac{y-\hat{c} \tau-\psi(x)-\widetilde{C}}{\sqrt{1+|\nabla \varphi(\alpha x)|^2}}= \frac{\min_{1 \leq i \leq n}\{y-\hat{c} \tau+x \cdot \nu_i \cot \theta_i\}-\widetilde{C}}{\sqrt{1+|\nabla \varphi(\alpha x)|^2}},
\end{equation*}
\begin{equation*}
	y-\hat{c} \tau-\varphi(\alpha x) / \alpha\leq y-\hat{c} \tau-\psi(x)=\min_{1 \leq i \leq n}\{y-\hat{c} \tau+x \cdot \nu_i \cot \theta_i\}\leq \widetilde{C},
\end{equation*}
and 
\begin{equation*}
	0<\min_{1 \leq i \leq n}\{y-\hat{c} t+x \cdot \nu_i \cot \theta_i\}\leq \widetilde{C}-\hat{c}\varrho \delta e^{-\lambda t} + \hat{c}\varrho \delta \leq \widetilde{C}+ \hat{c}\varrho \delta=: \widetilde{C}_{*}.
\end{equation*}
Thus, if $\xi(\tau, x, y)\leq 0$, we also get
	\begin{align*}
		W_{\delta}^{+}(t, x, y)&\leq \left|U_{e(x)}(\xi(\tau, x, y), x, y)\right|+2\\
		&\leq 1+\bar{K} e^{\kappa_2 \frac{y-\hat{c} \tau-\varphi(\alpha x) / \alpha}{\sqrt{1+|\nabla \varphi(\alpha x)|^2}}}+2\\
		&\leq \bar{K} e^{\kappa_2 \frac{y-\hat{c} \tau-\varphi(\alpha x) / \alpha}{\sqrt{(C_3+\max_{1\le i\le n} \{|\nu_i \cot \theta_i|\})^2+1}}}+3\\
		&\leq \bar{K} e^{\kappa_2 \frac{y-\hat{c} \tau-\psi(x)}{\sqrt{(C_3+\max_{1\le i\le n} \{|\nu_i \cot \theta_i|\})^2+1}}}+3\\
		&\leq \bar{K} e^{\kappa_2 \frac{\min_{1 \leq i \leq n}\{y-\hat{c} t+x \cdot \nu_i \cot \theta_i\}+\hat{c}(\varrho \delta e^{-\lambda t}-\varrho \delta)}{\sqrt{(C_3+\max_{1\le i\le n} \{|\nu_i \cot \theta_i|\})^2+1}}}+3 \\
		&\leq \bar{K} e^{\kappa_2 \frac{\min_{1 \leq i \leq n}\{y-\hat{c} t+x \cdot \nu_i \cot \theta_i\} }{\sqrt{(C_3+\max_{1\le i\le n} \{|\nu_i \cot \theta_i|\})^2+1}}}+3\\
		&\leq \bar{K} e^{\kappa_2 \frac{\min_{1 \leq i \leq n}\{y-\hat{c} t+x \cdot \nu_i \cot \theta_i\} }{\sqrt{(C_3+\max_{1\le i\le n} \{|\nu_i \cot \theta_i|\})^2+1}}}+3e^{\kappa_2 \min_{1 \leq i \leq n}\{y-\hat{c} t+x \cdot \nu_i \cot \theta_i\}} e^{-\kappa_2 \min_{1 \leq i \leq n}\{y-\hat{c} t+x \cdot \nu_i \cot \theta_i\}}\\
		&\leq \bar{K} e^{\kappa_2 \min_{1 \leq i \leq n}\{y-\hat{c} t+x \cdot \nu_i \cot \theta_i\}} +3e^{\kappa_2 \min_{1 \leq i \leq n}\{y-\hat{c} t+x \cdot \nu_i \cot \theta_i\}}\cdot e^{-\kappa_2 \min_{1 \leq i \leq n}\{y-\hat{c} t+x \cdot \nu_i \cot \theta_i\}}\\
		&\leq 4 \Lambda_{3}^{*} e^{-2 \tilde{v} \min_{1\leq i\leq n} \left\{x \cdot \nu_i \cot \theta_i+y-\hat{c} t \right\}}
	\end{align*}
where $\Lambda_{3}^{*}:=\max\{\bar{K}, \bar{K}e^{2 \kappa_2 \widetilde{C}_{*}}, e^{2 \kappa_2 \widetilde{C}_{*}}\}$, $\widetilde{C}=\frac{\sup_{\mathbb{R}^{N}}|\varphi(\alpha x)-\psi(\alpha x)|}{\alpha}$ and $\tilde{v}\leq\frac{ \kappa_2 }{2}$. In conclusion, let $\Lambda^{*}:=\max\{4\Lambda_{1}^{*},4\Lambda_{2}^{*},5\Lambda_{3}^{*}\}$ and $\tilde{v}\leq\min\{\frac{ \kappa_2 }{2}, \frac{3 \kappa }{8}\min_{1\leq i \leq n}\{\sin\theta_i,\beta, 1, \frac{1}{\sqrt{(C_3+\max_{1\le i\le n} \{|\nu_i \cot \theta_i|\})^2+1}}\}\}$, then we obtain \eqref{4.34}. 

Therefore, by \eqref{4.33}, one gets
\begin{equation}\label{4.35}
\left|u\left(T_{\vartheta}, x, y\right)-V\left(T_{\vartheta}, x, y\right)\right| \leq \vartheta^2 \Lambda^{*} \min \left\{1, e^{- \tilde{v} \min_{1\leq i\leq n} \left\{x \cdot \nu_i \cot \theta_i+y-\hat{c} t \right\}}\right\}, \quad \forall(x, y) \in \mathbb{R}^N, \quad \forall\vartheta>0.
\end{equation}
Denote $\beta_2:=\min \left\{\tilde{v} / K, \beta^*\right\}$, where $K$ is given in Theorem \ref{Theorem 2.3}. Following Theorem \ref{Theorem 2.12}, there exists a constant $r_4>0$ such that
\begin{equation}\label{4.36}
U_{e_{i}}^{\beta_2}(0, x, y) \geq U_{e_{i}}^{\beta_2}\left(X_{*}, x, y\right) \geq r_4, \quad \forall(x, y) \in \mathbb{R}^N.
\end{equation}
Define
\begin{equation}\label{4.37}
\Gamma:=\min \left\{1,\left(\frac{C_1}{2}\right)^{\beta_2} \frac{\min \left\{ \delta^{0}, \gamma_{\star} / 4\right\}}{\Lambda^{*}},\left(\frac{2}{C_1}\right)^{\beta_2} r_4\right\}
\end{equation}
where $C_1$ is given in Theorem \ref{Theorem 2.3}. We point out that
\begin{equation}\label{4.38}
V_{\delta}^{-}\left(0, x, y ; T_{\vartheta}\right) \leq u\left(T_{\vartheta}, x, y\right) \leq V_{\delta}^{+}\left(0, x, y ; T_{\vartheta}\right)
\end{equation}
in $\mathbb{R}^N$ for all $0<\vartheta<\Gamma$, where $V_{\delta}^{-}$ and $V_{\delta}^{+}$ are defined in Lemma \ref{Lemma 4.2}, and let $\beta=\beta_2$, $\delta=\bar{\delta}(\vartheta):= \Lambda^{*} \vartheta\left(2 / C_1\right)^{\beta_2}$ in $V_{\delta}^{-}$ and $V_{\delta}^{+}$.

\textbf{Case 1:} $\min_{1\leq i\leq n} \left\{x \cdot \nu_i \cot \theta_i+y-\hat{c} T_{\vartheta} \right\}>0$. Since $\beta_2 \leq \tilde{v} / K$, by virtue of Lemma \ref{Lemma 4.2}, \eqref{4.18}, \eqref{4.21}, \eqref{4.35} and \eqref{4.37}, one has for any $0<\vartheta<\Gamma$,
\begin{equation*}
\begin{aligned}
	V_{\bar{\delta}}^{+}\left(0, x, y ; T_{\vartheta}\right) & \geq V\left(T_{\vartheta}, x, y\right)+\bar{\delta} U_{e_{i}}^{\beta_2}(\eta, x, y) \\
	& \geq V\left(T_{\vartheta}, x, y\right)+\bar{\delta}\left(\frac{C_1}{2}\right)^{\beta_2} e^{-\beta_2 c_{e_{i}} \eta} \\
	& \geq V\left(T_{\vartheta}, x, y\right)+\Lambda^{*} \vartheta e^{-\tilde{v} (y-\hat{c} T_{\vartheta}-\varphi(\alpha x) / \alpha)} \\
	& \geq V\left(T_{\vartheta}, x, y\right)+\Lambda^{*} \vartheta^2 e^{-\tilde{v} \min_{1\leq i\leq n} \left\{x \cdot \nu_i \cot \theta_i+y-\hat{c} T_{\vartheta} \right\}} \\
	& \geq u\left(T_{\vartheta}, x, y\right)
\end{aligned}
\end{equation*}
in $\left\{(x, y): \eta\left(T_{\vartheta}, x, y\right)>X_*\right\}$. According to 
\eqref{4.35}-\eqref{4.37}, we know 
\begin{equation*}
V_{\bar{\delta}}^{+}\left(0, x, y ; T_{\vartheta}\right) \geq V\left(T_{\vartheta}, x, y\right)+\bar{\delta} U_{e_{i}}^{\beta_2}\left(X_*, x, y\right) \geq V\left(T_{\vartheta}, x, y\right)+\bar{\delta} r_4 \geq u\left(T_{\vartheta}, x, y\right)
\end{equation*}
in $\left\{(x, y): \eta\left(T_{\vartheta}, x, y\right) \leq X_*\right\}$ for all $0<\vartheta<\Gamma$.

\textbf{Case 2:} $\min_{1\leq i\leq n} \left\{x \cdot \nu_i \cot \theta_i+y-\hat{c} T_{\vartheta} 
\right\}\leq 0$. By \eqref{4.18} and \eqref{4.35}-\eqref{4.37}, we obtain
\begin{equation*}
	V_{\bar{\delta}}^{+}\left(0, x, y ; T_{\vartheta}\right) \geq V\left(T_{\vartheta}, x, y\right)+\bar{\delta} U_{e_{i}}^{\delta_2}(0, x, y) \geq V\left(T_{\vartheta}, x, y\right)+\bar{\delta} r_4 \geq u\left(T_{\vartheta}, x, y\right)
\end{equation*}
in Case $2$ for any $0<\vartheta<\Gamma$.

In conclusion, $V_{\bar{\delta}}^{+}\left(0, x, y ; T_{\vartheta}\right) \geq u\left(T_{\vartheta}, x, y\right)$ in $\mathbb{R}^N$, where parameters in Lemma \ref{Lemma 4.2} are taken as $\beta=\beta_2$, $\alpha=\tilde{\alpha}_0^{+}\left(\beta_2\right) / 2$, $\lambda=\lambda\left(\beta_2\right)$, $\varrho=\tilde{\varrho}\left(\beta_2, \lambda, \alpha\right)$, $\delta=\bar{\delta}=\Lambda^{*} \vartheta\left(2 / C_1\right)^{\beta_2}$. Similarly, we can also get $V_{\bar{\delta}}^{-}\left(0, x, y ; T_{\vartheta}\right) \leq u\left(T_{\vartheta}, x, y\right)$ in $\mathbb{R}^N$ for all $0<\vartheta<\Gamma$. Therefore, \eqref{4.38} is ture. We obtain from \eqref{4.38}, Lemma \ref{Lemma 4.2} and the comparison principle that for any $0<\vartheta<\Gamma$,
\begin{equation*}
	V_{\bar{\delta}}^{-}\left(s, x, y ; T_{\vartheta}\right) \leq u\left(T_{\vartheta}+s, x, y\right) \leq V_{\bar{\delta}}^{+}\left(s, x, y ; T_{\vartheta}\right), \quad \forall(s, x, y) \in[0,+\infty) \times \mathbb{R}^N .
\end{equation*}
Letting $\vartheta \rightarrow 0$ which indicates $\bar{\delta} \rightarrow 0$, it follows from Lemma \ref{Lemma 4.2} that
\begin{equation*}
u(t, x, y) \rightarrow V(t, x, y) \text { as } t \rightarrow+\infty
\end{equation*}
uniformly in $(x, y) \in \mathbb{R}^N$. For general $e_0\in\mathbb{S}^{N-1}$, we can change the subsolutions and supersolutions in the same way as in the proof of Theorem \ref{Theorem 2.18}. The proof is complete.
\end{proof}

\section{Appendix}
\noindent
In this section, we complete the proof of Lemma \ref{Lemma 2.15}.
\begin{lemma}\label{Lemma 5.1}
If $u$ is a sub-invasion of 0 by 1 with sets $\left(\Omega_t^{ \pm}\right)_{t \in \mathbb{R}}$ and $\left(\Gamma_t\right)_{t \in \mathbb{R}}$, then
\begin{equation}\label{5.1}
	\sup \left\{d\left(z, \Gamma_{t-\tau}\right) ; t \in \mathbb{R}, z \in \Gamma_t\right\}<+\infty \text{ for any constant } \tau>0.
\end{equation}
\end{lemma}
\begin{proof}
Take suitable constants $\varepsilon>0$, $T_{\varepsilon}>0$, $R_{\varepsilon}>0$ and $\delta_{\varepsilon}>0$ such that \cite[Lemmas 3.1 and 3.2]{Suobing Zhang2025} hold. Let $\tilde{\varepsilon}:=\delta_{\varepsilon}$ and $M_{\tilde{\varepsilon}}$ be defined by \eqref{1.10}. Since $\Gamma_t \subset \Omega_{t-\tau}^{-}$ for any $\tau>0$, we only need to prove \eqref{5.1} holds for $\tau \geq T_{\varepsilon}$. Take $R \geq \max \left\{R_{\varepsilon},(K+\varepsilon) \tau\right\}$ where $K$ is given in Theorem \ref{Theorem 2.3}.

Assume by contradiction that
\begin{equation*}
\sup \left\{d\left(z, \Gamma_{t-\tau}\right) ; t \in \mathbb{R}, z \in \Gamma_t\right\}= +\infty,
\end{equation*}
then there exist $t_0 \in \mathbb{R}$ and $z_0 \in \Gamma_{t_0} \subset \Omega_{t_0-\tau_0}^{-}$ such that
\begin{equation}\label{5.2}
d\left(z_0, \Gamma_{t_0-\tau_0}\right) \geq r_{M_{\tilde{\varepsilon}}}+M_{\tilde{\varepsilon}}+R,
\end{equation}
where $r_{M_{\tilde{\varepsilon}}}>0$ satisfy
\begin{equation*}
\sup \left\{d\left(y, \Gamma_{t_0}\right) ; y \in \Omega_{t_0}^{+} \cap B\left(z_0, r_{M_{\tilde{\varepsilon}}}\right)\right\} \geq M_{\tilde{\varepsilon}} 
\end{equation*}
and this is possible by \eqref{1.8}. Thus, there exists $y_0 \in \mathbb{R}^N$ such that
\begin{equation}\label{5.3}
y_0 \in \Omega_{t_0}^{+},\left|y_0-z_0\right| \leq r_{M_{\tilde{\varepsilon}}} \text { and } d\left(y_{0}, \Gamma_{t_0}\right) \geq M_{\tilde{\varepsilon}}.
\end{equation}
By \eqref{1.10} and the definition of $\delta_{\varepsilon}$ in \cite[Lemma 3.2]{Suobing Zhang2025}, one has that
\begin{equation}\label{5.4}
u\left(t_0, y_0\right) \geq 1-\tilde{\varepsilon}\geq 1-\frac{\gamma_{\star}}{8}>1-\gamma_{\star} \geq 1-\frac{\theta}{2}>1-\theta \geq \gamma_{\star}.
\end{equation}
Besides, by \eqref{5.2} and \eqref{5.3}, we can get $B\left(y_0, R\right) \subset \Omega_{t_0-\tau}^{-}$ and $d\left(B\left(y_0, R\right), \Gamma_{t_0-\tau}\right) \geq M_{\tilde{\varepsilon}}$, which further shows
\begin{equation}\label{5.5}
u\left(t_0-\tau, y\right) \leq \tilde{\varepsilon} < \gamma_{\star} \text { for all } y \in B\left(y_0, R\right).
\end{equation}
For the above $R$ and $\tilde{\varepsilon}$, let $v_R(t, z)$ be the solution of the following equation
\begin{equation*}
\left\{\begin{array}{l}
	\left(v_R\right)_t- \Delta v_R=f\left(z, v_R\right), \quad t>0,\; z \in \mathbb{R}^N ,\\
	v_R(0, z)=\tilde{\varepsilon} \text { for }|z|<R, \quad v_R(0, z)=1 \text { for }|z| \geq R.
\end{array}\right.
\end{equation*}
By \cite[Lemma 3.2]{Suobing Zhang2025}, it holds that
\begin{equation*}
v_R(t, z) \leq \gamma_{\star}, \text { for all } T_{\varepsilon} \leq t \leq \frac{R}{K+\varepsilon} \text { and }|z| \leq R-(K+\varepsilon) t.
\end{equation*}
Due to $u(t, z)$ is a subsolution and $R \geq(K+\varepsilon) \tau$, we can obtain from \eqref{5.5} and the comparison principle that
\begin{equation*}
u\left(t_0, y\right) \leq v_R\left(\tau, y-y_0\right) \text { for all } y \in \mathbb{R}^N.
\end{equation*}
Then one has $u\left(t_0, y_0\right) \leq v_R(\tau, 0) \leq \gamma_{\star}$ which contradicts \eqref{5.4}.
\end{proof}

\begin{proof}[Proof of Lemma \ref{Lemma 2.15}]
Similar to \cite{H. Berestycki2012}, we omit some details and only present the main proof steps.

{\it Step 1: dividing $\mathbb{R}^N$ into multiple parts.}
 We denote $\widetilde{u}_s(t, z)=\widetilde{u}(t+s, z)$ for $(t, z) \in \mathbb{R} \times \mathbb{R}^N$ and obtain from \eqref{1.10} that for $0<\gamma_{\star}\leq \min \{\theta / 2, 1-\theta\}<\frac{1}{2}$, there exist constants $M>0$ and $\widetilde{M}>0$ such that
\begin{equation}\label{5.6}
	\left\{\begin{array}{l}
		\forall t \in \mathbb{R}, \forall z \in \Omega_t^{+},\left(d\left(z, \Gamma_t\right) \geq M\right) \Rightarrow(u(t, z) \geq 1-\gamma_{\star} / 2), \\
		\forall t \in \mathbb{R}, \forall z \in \Omega_t^{-},\left(d\left(z, \Gamma_t\right) \geq M\right) \Rightarrow(u(t, z) \leq \gamma_{\star}),
	\end{array}\right.
\end{equation}
and
\begin{equation}\label{5.7}
\left\{\begin{array}{l}
	\forall t \in \mathbb{R}, \forall z \in \widetilde{\Omega}_t^{+},\left(d\left(z, \widetilde{\Gamma}_t\right) \geq \widetilde{M}\right) \Rightarrow(\widetilde{u}(t, z) \geq 1-\gamma_{\star} / 2) ,\\
	\forall t \in \mathbb{R}, \forall z \in \widetilde{\Omega}_t^{-},\left(d\left(z, \widetilde{\Gamma}_t\right) \geq \widetilde{M}\right) \Rightarrow(\widetilde{u}(t, z) \leq \gamma_{\star}).
\end{array}\right.
\end{equation}
By assumptions of Lemma \ref{Lemma 2.15}, we can know $u$ and $\widetilde{u}$ are sub-invasion and super-invasion of $0$ by $1$ respectively, and $\widetilde{\Omega}_t^{-} \subset \Omega_t^{-}$. Then there is a constant $s_0$ large enough such that
\begin{equation*}
	\widetilde{\Omega}_{t+s}^{+} \supset \Omega_t^{+} \text {and } d\left(\widetilde{\Gamma}_{t+s}, \Gamma_t\right) \geq M+\widetilde{M} \text { for all } t \in \mathbb{R} \text { and } s \geq s_0.
\end{equation*}
We infer from above that
\begin{equation}\label{5.8}
\left(z \in \Omega_t^{+}\right) \text {or }\left(x \in \Omega_t^{-} \text { and } d\left(z, \Gamma_t\right) \leq M\right) \Rightarrow\left(\widetilde{u}_s(t, z) \geq 1-\gamma_{\star}\right).
\end{equation}
Define
\begin{equation*}
\omega_M^{-}:=\left\{(t, z) \in \mathbb{R} \times \mathbb{R}^N ; z \in \Omega_t^{-} \text { and } d\left(z, \Gamma_t\right) \geq M\right\}  \text { and } \; \omega_M^{+}:=\mathbb{R} \times \mathbb{R}^N \backslash \omega_M^{-}.
\end{equation*}
According to \eqref{5.6}-\eqref{5.8}, one has
\begin{equation}\label{5.9}
\widetilde{u}_s(t, z)-u(t, z) \geq 1-2 \gamma_{\star}>0 \text { on } \partial \omega_{M}^{-}.
\end{equation}

{\it Step 2: $\widetilde{u}_s \geq u$ in $\omega_M^{-}$ for all $s \geq s_0$.}
Fix $s \geq s_0$ and define $\varepsilon^*=\inf \left\{\varepsilon>0 ; \widetilde{u}_s \geq u-\varepsilon\text{ in } \omega_M^{-}\right\}$, then $\varepsilon^*$ is a well-defined nonnegative constant. We only have to show $\varepsilon^*=0$. Assume by contradiction that $\varepsilon^*>0$. There are a sequence of positive numbers $\left\{\varepsilon_n\right\}_{n \in \mathbb{N}}$ and a sequence of points $\left\{\left(t_n, z_n\right)\right\}_{n \in \mathbb{N}}$ in $\omega_M^{-}$ such that
\begin{equation}\label{5.10}
\varepsilon_n \rightarrow \varepsilon^* \text { as } n \rightarrow+\infty \text { and } \widetilde{u}_s\left(t_n, z_n\right)<u\left(t_n, z_n\right)-\varepsilon_n \text { for all } n \in \mathbb{N}.
\end{equation}
By \eqref{5.9}, \eqref{5.10} and the positivity of $\varepsilon^*$, there is $\rho>0$ such that
\begin{equation*}
\liminf _{n \rightarrow+\infty} d\left(z_n, \Gamma_{t_n}\right) \geq M+2 \rho.
\end{equation*}
Since $u$ is a sub-invasion of 0 by 1, there exists $n_0$ such that
\begin{equation}\label{5.11}
z \in \Omega_t^{-} \text { and } d\left(z, \Gamma_t\right) \geq M \text { for all } n \geq n_0, x \in \overline{B\left(z_n, \rho\right)} \text { and } t \leq t_n .
\end{equation}
Therefore, without loss of generality, we suppose that $\rho<\tau$, where $\tau$ is given in Lemma \ref{Lemma 5.1}. Then there exists a sequence $\left\{y_n\right\}_{n \in \mathbb{N}}$ in $\mathbb{R}^N$ such that
\begin{equation}\label{5.12}
y_n \in \Omega_{t_n-\tau+\rho}^{-} \text { and } d\left(z_n, \Gamma_{t_n-\tau+\rho}\right)-d\left(z_n, y_n\right)=d\left(y_n, \Gamma_{t_n-\tau+\rho}\right)=M+\rho \text { for all } n \geq n_0 .
\end{equation}
Defining a $C^1$ path $P_n:[0,1] \rightarrow \Omega_{t_n-\tau+\rho}^{-}$ with $P_n(0)=z_n, P_n(1)=y_n$, by \eqref{5.11} and \eqref{5.12}, we know that for each $n \geq n_0$,
\begin{equation*}
	E_n:=\left[t_n-\tau, t_n\right] \times \overline{B\left(z_n, \rho\right)} \cup\left[t_n-\tau, t_n-\tau+\rho\right] \times\left\{z \in \mathbb{R}^N ; d\left(z, P_n([0,1])\right) \leq \rho\right\}
\end{equation*}
is included in $\omega_M^{-}$. Thus, 
\begin{equation*}
v:=\widetilde{u}_s-\left(u-\varepsilon^*\right) \geq 0 \text { in } E_n \text { for all } n \geq n_0 .
\end{equation*}
Since $f(z, \cdot)$ is nonincreasing in $(-\infty, \gamma_{\star}]$, we can get
\begin{equation*}
v_t(t, z)- \Delta v(t, z)+L v(t, z) \geq 0 \text { in } E_n \text { for all } n \geq n_0,
\end{equation*}
where $L=\|f(z, u)\|_{L^{\infty}\left(\mathbb{L}^N \times[-1-\varepsilon^*,1+\varepsilon^*]\right)}<+\infty$.

Now we pay attention to the distance between the sequences just defined. Firstly, we show that $\left\{d\left(z_n, \Gamma_{t_n}\right)\right\}_{n \in \mathbb{N}}$ is bounded. Otherwise, up to extraction of subsequences, we have that $d\left(z_n, \Gamma_{t_n}\right) \rightarrow+\infty$, then $u\left(t_n, z_n\right), \widetilde{u}_s\left(t_n, z_n\right) \rightarrow 0$. But, it follows from \eqref{5.10} that $u\left(t_n, z_n\right)>\widetilde{u}_s\left(t_n, z_n\right)+\varepsilon_n \rightarrow \varepsilon^*>0$ as $n \rightarrow+\infty$, which is a contradiction. Besides, we can obtain from \eqref{5.1} that there exists a sequence $\left\{\widetilde{z}_n\right\}_{n \in \mathbb{N}}$ in $\mathbb{R}^N$ such that $\widetilde{z}_n \in \Gamma_{t \rightarrow \tau}$ for each $n \in \mathbb{N}$ and $\sup \left\{d\left(z_n, \widetilde{z}_n\right) ; n \in \mathbb{N}\right\}<+\infty$. Therefore, for each $n \geq n_0$,
\begin{equation*}
d\left(z_n, y_n\right)= d\left(z_n, \Gamma_{t_n-\tau+\rho}\right)-(M+\rho) \leq d\left(z_n, \Gamma_{t_n-\tau}\right)-(M+\rho) \leq d\left(z_n, \widetilde{z}_n\right)-(M+\rho)<+\infty.
\end{equation*}
Since $v\left(t_n, z_n\right) \rightarrow 0$ as $n \rightarrow+\infty$, by the linear parabolic estimates, we obtain
\begin{equation*}
v\left(t_n-\tau, y_n\right) \rightarrow 0 \text { as } n \rightarrow+\infty.
\end{equation*}
However, it follows from \eqref{5.12} that there is a sequence $\left\{\bar{z}_n\right\}_{n \in \mathbb{N}}$ such that
\begin{equation*}
\bar{z}_n \in \Omega_{t_n-\tau+\rho}^{-}, d\left(\bar{z}_n, y_n\right)=\rho \text { and } d\left(\bar{z}_n, \Gamma_{t_n-\tau+\rho}\right)=M \text { for all } n \geq n_0.
\end{equation*}
Since $\partial_{t} \widetilde{u}$ and $\nabla_{z} \widetilde{u}$ are all globally bounded in $\mathbb{R} \times \mathbb{R}^N$ and $d\left(\Gamma_{t_n-\tau}, \Gamma_{t_n-\tau+\rho}\right)$ is bounded by Lemma \ref{Lemma 5.1}, even if it means decreasing $\rho$, one has that
\begin{equation}\label{5.13}
\rho \cdot(\left\|\nabla_z \widetilde{u}\right\|_{L^{\infty}\left(\mathbb{R} \times \mathbb{R}^N\right)}+ \left\|\partial_{t}  \widetilde{u}\right\|_{L^{\infty}\left(\mathbb{R} \times \mathbb{R}^N\right)})\leq \varepsilon^* / 2.
\end{equation}
Finally, for all $n \geq n_0$, we know from \eqref{5.6}, \eqref{5.8} and \eqref{5.13} that
\begin{equation*}
	\begin{aligned}
    v\left(t_n-\tau, y_n\right)&=\widetilde{u}_s\left(t_n-\tau, y_n\right)-u\left(t_n-\tau, y_n\right)+\varepsilon^* \\
    &\geq\widetilde{u}_s\left(t_n-\tau+\rho, \bar{z}_n\right)-\varepsilon^* / 2 -u\left(t_n-\tau, y_n\right)+\varepsilon^*\\
    &\geq 1- \gamma_{\star}+\varepsilon^*/2- \gamma_{\star}>0.		
	\end{aligned}
\end{equation*}
This is a contradiction to $\varepsilon^*>0$, and hence $\varepsilon^*=0$.

{\it Step 3: $\widetilde{u}_s \geq u$ in $\omega_M^{+}$ for all $s \geq s_0$. }
Fix $s \geq s_0$ and define $\varepsilon_*=\inf \left\{\varepsilon>0 ; \widetilde{u}_s \geq u-\varepsilon\right.$ in $\left.\omega_M^{+}\right\}$, then we show $\varepsilon_*=0$. Suppose by contradiction that $\varepsilon_*>0$. There exists a sequence of positive constants $\left\{\varepsilon_n\right\}_{n \in \mathbb{N}}$ and a sequence of points $\left\{\left(t_n, z_n\right)\right\}_{n \in \mathbb{N}}$ in $\omega_M^{+}$ such that
\begin{equation}\label{5.14}
\varepsilon_n \rightarrow \varepsilon_* \text { as } n \rightarrow+\infty \text { and } \widetilde{u}_s\left(t_n, z_n\right)<u\left(t_n, z_n\right)-\varepsilon_n \text { for all } n \in \mathbb{N}.
\end{equation}
Similar to arguments in Step $2$, we know $\left\{d\left(z_n, \Gamma_{t_n}\right)\right\}_{n \in \mathbb{N}}$ is bounded. By \eqref{5.1}, there exists a sequence $\left\{\widetilde{z}_n\right\}_{n \in \mathbb{N}}$ in $\mathbb{R}^N$ such that $\widetilde{z}_n \in \Gamma_{t_n-\tau}$ for each $n \in \mathbb{N}$ and $\sup \left\{d\left(z_n, \widetilde{z}_n\right) ; n \in \mathbb{N}\right\}<+\infty$. Following from \eqref{1.8}, there are $r>0$ and $\left\{y_n\right\}_{n \in \mathbb{N}}$ in $\mathbb{R}^N$ such that
\begin{equation*}
y_n \in \Omega_{t_n-\tau}^{-}, d\left(\widetilde{z}_n, y_n\right)=r \text { and } d\left(y_n, \Gamma_{t_n-\tau}\right) \geq M \text { for all } n \in \mathbb{N}.
\end{equation*}
Then there is a sequence $\left\{\bar{z}_n\right\}_{n \in \mathbb{N}}$ in $\mathbb{R}^N$ such that 
\begin{equation*}
\bar{z}_n \in \Omega_{t_n-\tau}^{-} \text{ and } d\left(\bar{z}_n, \Gamma_{t_n-\tau}\right)=d\left(y_n, \Gamma_{t_n-\tau}\right)-d\left(y_n, \bar{z}_{n}\right)=M \text{ for all } n \in \mathbb{N}.
\end{equation*}
Based on the above, we know that $\left\{d\left(z_n, \bar{z}_n\right)\right\}_{n \in \mathbb{N}}$ is bounded. Choose $\rho^{*}>0$ and $K^{*} \in \mathbb{N} \backslash\{0\}$ such that
\begin{equation*}
	\rho^{*} \cdot(2\left\|\nabla_z \widetilde{u}\right\|_{L^{\infty}\left(\mathbb{R} \times \mathbb{R}^N\right)}+ \left\|\partial_{t}  \widetilde{u}\right\|_{L^{\infty}\left(\mathbb{R} \times \mathbb{R}^N\right)})\leq \varepsilon^* 
\end{equation*}
and
\begin{equation*}
K^{*} \rho^{*} \geq \max \left(\tau, \sup \left\{d\left(z_n, \bar{z}_n\right) ; n \in \mathbb{N}\right\}\right).
\end{equation*}
Define
\begin{equation*}
E_{n, i}=\left[t_n-\frac{i+1}{K^{*}} \tau, t_n-\frac{i}{K^{*}} \tau\right] \times \overline{B\left(X_{n, i}, 2 \rho^{*}\right)} \text{ for all } n \in \mathbb{N}  \text{ and }  0 \leq i \leq K^{*}-1 ,
\end{equation*}
where $\left\{X_{n, i}\right\}_{0 \leq i \leq K^{*}}$ is a sequence in $\mathbb{R}^N$ such that
\begin{equation*}
X_{n, 0}=z_n, X_{n, K^{*}}=\bar{z}_n \text { and } d\left(X_{n, i}, X_{n, i+1}\right) \leq \rho^{*} \text { for any } 0 \leq i \leq K^{*}-1.
\end{equation*}
According to \eqref{5.14} and the global boundedness of $\partial_{t} (\widetilde{u}_s-u)$ and $\nabla_{z}\left(\widetilde{u}_s-u\right)$, one has $w:=$ $\widetilde{u}_s-\left(u-\varepsilon_*\right)<\varepsilon_*$ for large $n$. This shows that $E_{n, 0} \subset \omega_M^{+}$ for large $n$, because we know from Step $2$ that $w \geq \varepsilon_*$ in $\omega_M^{-}$. Since $f(z, \cdot)$ is nonincreasing in $[1-\gamma_{\star}, +\infty)$ by \eqref{1.6}, similar to arguments in Step $2$, it holds that 
\begin{equation*}
w\left(t_n-\tau / K^{*}, X_{n, 1}\right) \rightarrow 0 \text { as } n \rightarrow+\infty.
\end{equation*}
An immediate induction yields $w\left(t_n- i \tau / K^{*}, X_{n, i}\right)\rightarrow 0$ as $n \rightarrow+\infty$ for any $ 1 \leq i \leq K^{*}$ and then $w\left(t_n-\tau, \bar{z}_n\right) \rightarrow 0$ as $n \rightarrow+\infty$. By the definition of $\bar{z}_n$, one has $\left(t_n-\tau, \bar{z}_n\right) \in \omega_M^{-}$ and thereby $w\left(t_n-\tau, \bar{z}_n\right) \geq \varepsilon_*$. This is a contradiction, which means $\varepsilon_*=0$.

{\it Step 4: existence of the smallest $T$.} 
It follows from Steps $1$-$3$ that $\widetilde{u}_s \geq u$ in $\mathbb{R} \times$ 
$\mathbb{R}^N$ for all $s \geq s_0$. Define
\begin{equation*}
s_*:=\inf \left\{s \in \mathbb{R} ; \widetilde{u}_s \geq u \text { in } \mathbb{R} \times \mathbb{R}^N\right\}.
\end{equation*}
Since $0<u(t, z)<1$ for all $(t, z) \in \mathbb{R} \times \mathbb{R}^N$ and $\widetilde{u}_s\left(t_0, z_0\right) \rightarrow 0$ as $s \rightarrow-\infty$ for all $\left(t_0, z_0\right) \in \mathbb{R} \times \mathbb{R}^N$, we can get $-\infty<s_* \leq s_0$. Therefore, there exists a sequence $\{s_{n}\}_{n \in \mathbb{N}}$ which satisfy $s_*\leq s_{n}<s_*+\frac{1}{n}$ and $\widetilde{u}_{s_n} \geq u$ in $\mathbb{R} \times\mathbb{R}^N$. Moreover, one gets that ${s_{n}}$ converges to $s_*$, because $s_{n}\rightarrow s_*$ as $n \rightarrow +\infty$. We also have that for any fixed $(t, z) \in \mathbb{R} \times \mathbb{R}^N$, the mapping $s \mapsto \widetilde{u}_s (t, z)=\widetilde{u}(t+s, z)$ is a continuous function with respect to $s$. So it holds that $\widetilde{u}_{s_{n}} (t, z) \rightarrow \widetilde{u}_{s_*}(t, z)$ as $n\rightarrow +\infty$. Then we obtain that $u(t, z) \leq \lim_{n\rightarrow +\infty}\widetilde{u}_{s_n}=\widetilde{u}_{s_*}(t, z)$. Due to the arbitrariness of $(t, z)$, one has
\begin{equation*}
\widetilde{u}_{s_*}(t, z)=\widetilde{u}\left(t+s_*, z\right) \geq u(t, z) \text { for all }(t, z) \in \mathbb{R} \times \mathbb{R}^N.
\end{equation*}
This implies that $s_*$ is the smallest $T$. Take $T=s_*$ in the following proof.

In particular,
\begin{equation}\label{5.15}
	\left(z \in \Omega_t^{+} \text { and } d\left(z, \Gamma_t\right) \geq M\right) \Longrightarrow\left(\widetilde{u}_T(t, z) \geq u(t, z) \geq 1-\frac{\gamma_{\star}}{2}\right).
\end{equation}
Suppose 
\begin{equation}\label{5.16}
\inf \left\{\widetilde{u}_T(t, z)-u(t, z) ; d\left(z, \Gamma_t\right) \leq M \right\}>0,
\end{equation}
 then there exists a constant $\eta_0>0$ such that for any $\eta \in\left(0, \eta_0\right]$, $\widetilde{u}_{T-\eta}-u \geq 0$ for all $(t, z) \in \mathbb{R} \times \mathbb{R}^N$ such that $d\left(z, \Gamma_t\right) \leq M$, since $\partial_{t}\widetilde{u}$ is globally bounded. By \eqref{5.15}, we can assume that $\eta_0>0$ (even if it means decreasing $\eta_0$) such that
\begin{equation*}
 \left(z \in \Omega_t^{+} \text { and } d\left(z, \Gamma_t\right) \geq M\right) \Longrightarrow\left(\widetilde{u}_{T-\eta}(t, z) \geq 1-\gamma_{\star}\right)
\end{equation*}
for all $\eta \in\left[0, \eta_0\right]$. With similar arguments as those in Step $2$, one has 
 \begin{equation*}
 \widetilde{u}_{T-\eta}(t, z) \geq u(t, z)\text{ for all } \eta \in\left[0, \eta_0\right]\text{ and } (t, z) \in \mathbb{R} \times\mathbb{R}^{N} \text{ with } z \in \Omega_t^{-}\text{ and } d\left(z, \Gamma_t\right) \geq M.
\end{equation*}
Then we know that for all $\eta \in\left[0, \eta_0\right]$, $\widetilde{u}_{T-\eta}(t, z) \geq u(t, z)$ for all $(t, z) \in \mathbb{R} \times \mathbb{R}^N$ such that $z \in \Omega_t^{-}$ or $z \in \Omega_t^{+}$ and $d\left(z, \Gamma_t\right) \leq M$. As arguments in Step $3$, one has
 \begin{equation*}
	\widetilde{u}_{T-\eta}(t, z) \geq u(t, z)\text{ for all } \eta \in\left[0, \eta_0\right]\text{ and } (t, z) \in \mathbb{R} \times\mathbb{R}^{N} \text{ with } z \in \Omega_t^{+}\text{ and } d\left(z, \Gamma_t\right) \geq M.
\end{equation*}
 Finally, it holds that
 \begin{equation*}
 \widetilde{u}_{T-\eta} \geq u \text { in } \mathbb{R} \times \mathbb{R}^{N}
\end{equation*}
 for each $\eta \in\left[0, \eta_0\right]$. This contradicts the minimality of $T$ and then \eqref{5.16} is false. Thus, we know that $\inf \left\{\widetilde{u}_T(t, z)-u(t, z) ; d\left(z, \Gamma_t\right) \leq M\right\}=0$ and then there is a sequence $\left\{\left(t_n, z_n\right)\right\}_{n \in \mathbb{N}}$ of $\mathbb{R} \times \mathbb{R}^N$ such that
\begin{equation*}
d\left(z_n, \Gamma_{t_n}\right) \leq M \text { for all } n \in \mathbb{N} \text { and } \widetilde{u}_T\left(t_n, z_n\right)-u\left(t_n, z_n\right) \rightarrow 0 \text { as } n \rightarrow+\infty.
\end{equation*}
The proof is complete.
\end{proof}

\section*{Acknowledgments}

This work was partially supported by NSF of China (12171120) and by the Fundamental Research 
Funds for the Central Universities (No. 2023FRFK030022, 2022FRFK060028).

\section*{ Data availability statements}  

We do not analyze or generate any datasets, because our work proceeds within a theoretical and 619
mathematical approach.

\section*{ Conflict of interest} 

There is no conflict of interest to declare.

\end{document}